\documentclass[pdflatex,sn-mathphys-num]{sn-jnl}

\usepackage{geometry}
\usepackage{graphicx}%
\usepackage{threeparttable}
\usepackage{multirow}%
\usepackage{amssymb,amsmath,amsfonts,amsthm,mathrsfs,mathtools}
\usepackage{bm}
\usepackage[title]{appendix}%
\usepackage{xcolor}
\usepackage{textcomp}
\usepackage{manyfoot}
\usepackage{booktabs}%
\usepackage{algorithm}%
\usepackage{algorithmicx}%
\usepackage{algpseudocode}%
\usepackage{listings}%
\usepackage{enumitem}
\numberwithin{equation}{section}

\usepackage{tabularx}
\usepackage{microtype}
\usepackage{ragged2e}

\definecolor{journalblue}{RGB}{22,58,110}
\definecolor{accentteal}{RGB}{0,112,118}
\definecolor{pitchgray}{RGB}{78,84,92}
\definecolor{rulegray}{RGB}{190,197,205}
\definecolor{panelblue}{RGB}{239,244,250}

\theoremstyle{thmstyleone}%
\newtheorem{theorem}{Theorem}[section]
\newtheorem{proposition}[theorem]{Proposition}

\newtheorem{corollary}[theorem]{Corollary}
\theoremstyle{definition}

\newtheorem{assumption}[theorem]{Assumption}

\theoremstyle{remark}
\newtheorem{remark}[theorem]{Remark}

\newcommand{\R}{\mathbb{R}}
\newcommand{\dd}{\,\mathrm{d}}
\newcommand{\e}[1]{{\times}10^{#1}}
\DeclareMathOperator{\dist}{dist}
\DeclareMathOperator{\diag}{diag}
\DeclareMathOperator{\Var}{Var}

\begin{document}
\small
\title{A transport-only null model for apparent heterogeneity in diffusively dosed organoid arrays}

\author[1]{\fnm{Jiguang} \sur{Yu}}
\email{jyu678@bu.edu}
\equalcont{These authors contributed equally to this work.}

\author*[2]{\fnm{Louis Shuo} \sur{Wang}}
\email{wang.s41@northeastern.edu}
\equalcont{These authors contributed equally to this work.}

\affil[1]{
  \orgdiv{College of Engineering}, 
  \orgname{Boston University}, 
  \orgaddress{
    \city{Boston}, 
    \state{MA}, 
    \postcode{02215}, 
    \country{United States}
  }
}

\affil[2]{
  \orgdiv{Department of Mathematics}, 
  \orgname{Northeastern University}, 
  \orgaddress{
    \city{Boston}, 
    \state{MA}, 
    \postcode{02115}, 
    \country{United States}
  }
}

\abstract{
Spatial transport can create apparent biological heterogeneity even when organoids are intrinsically identical. We develop a transport-to-phenotype null model for diffusively dosed liver-cancer organoid arrays. The model couples bulk diffusion and clearance to partially accessible adsorption, reversible surface residence, productive internalization, and intracellular state dynamics. Matched asymptotics reduce the perforated-domain problem to a Green-function system, while renewal resolvents describe desorption, re-adsorption, and residence-time effects. Across $2000$ random ten-organoid arrays with localized dosing, the predicted transport-only maturation coefficient of variation has median $0.623$; one-factor design changes move this median between $0.27$ and $0.86$. After matching array-mean exposure, distributed dosing reduces the baseline spread approximately fivefold. The analysis also shows that, in a conservative reflecting chamber, desorption changes uptake timing and allocation but not total eventual uptake; reductions in total uptake require a competing loss channel. Residence laws with equal means can nevertheless produce different transient phenotypes. The spatial reduction is verified against finite-element solutions of the full PDE, and the time reconstruction against numerical Laplace inversion. Finally, a large-batch theorem shows that increasing batch size averages independent process variation but not shared line or batch effects. The framework provides a geometry-specific null against which measured organoid heterogeneity can be assessed.
}



\maketitle

\section{Introduction}
\label{sec:introduction}

Patient-derived liver-cancer organoids are increasingly used to study tumour biology and treatment response \cite{broutier2017human,artegiani2019probing,wang2025analysis,sun2019modelling,tang2022human,drost2017translational,zhao2022organoids,fatehullah2016organoids,hofer2021engineering}. When nominally identical organoids are subjected to the same dosing protocol, however, their measured phenotypes can differ substantially. Such variation may reflect biological heterogeneity, measurement error, or unequal intracellular exposure. The third mechanism is easily overlooked: organoids compete for a shared diffusible pool, and their positions relative to the dosing source, the chamber boundary, and one another determine how much material reaches and enters each compartment. Consequently, spatially heterogeneous delivery can produce heterogeneous phenotypes even when every organoid has identical biological parameters.

Extracellular concentration is not itself a measure of biologically effective exposure. A soluble molecule must diffuse to an organoid, encounter an accessible surface site, remain bound for a finite residence time, and then either desorb or be productively internalized \cite{lauffenburger1996receptors,wang2025analysis1}. Reaction--diffusion and cellular-pharmacology models have long coupled diffusion to binding, internalization, and trafficking \cite{rippley1995effects,liang2025global,wenning1999coupled}; tumour-spheroid and hepatic models have additionally resolved penetration, membrane transport, and metabolism \cite{goodman2008spatiotemporal,gao2013predictive,wang2026algebraic,leedale2020multiscale}. Restricted penetration is especially relevant to dense liver-tumour tissue \cite{hernandez2013role,jain2014role,wang2026damage,minchinton2006drug}. What is still lacking for organoid-array experiments is a usable transport-only null: given measured organoid centres and a specified delivery protocol, how much between-organoid variation can transport generate when the biological and kinetic parameters are held fixed?

We construct such a null by coupling diffusion in a bounded chamber to encounter-dependent adsorption, reversible surface residence, productive internalization, and a five-dimensional cell-state model. The biologically effective input is the productive internalization flux, rather than the extracellular concentration or the first surface encounter. In two dimensions, matched asymptotics reduce the perforated-domain problem to an $N\times N$ Green-function system evaluated at the organoid centres; the corresponding three-dimensional reduction is expressed through monopole strengths and shape-dependent capacitances. Renewal equations then compose repeated adsorption, residence, desorption, and re-adsorption cycles, including a parameter that interpolates between well-mixed resetting and spatially
continued search. The resulting fluxes drive a positively invariant finite-dimensional phenotype system and, subsequently, a nested line--batch--process acceptance model.

For the specified ensemble of 2000 ten-organoid arrays dosed from a localized port, the model predicts a median transport-only maturation coefficient of variation of $62\%$ and a median disease-score coefficient of variation of $84\%$. In a one-factor-at-a-time design audit, the median maturation coefficient of variation ranges from $27\%$ to $86\%$. These values are conditional predictions of the stated model and layout law, not empirical estimates or universal decision thresholds. Their purpose is to define a layout-specific reference against which observed heterogeneity can be tested. A measurement that exceeds this reference is not explained by the specified transport-only model; failure to exceed it does not establish biological homogeneity.

The analysis also resolves four mechanistic ambiguities. First, in a conservative chamber with reflecting walls and no bulk loss, the total eventual productive-uptake probability satisfies $\sum_j\pi_{jm}=1$; for a single target, $\pi_{1m}=1$. Desorption can delay delivery and redistribute uptake among competing targets, but it reduces total eventual uptake only through interaction with clearance or another loss channel. Second, residence-time laws with the same mean and the same per-encounter internalization probability can produce different uptake histories and phenotypes. The transient therefore depends on
the shape of the residence law, not only on its mean; this comparison is made by numerical inversion of the exact transform rather than by a two-moment surrogate. Third, raw uptake does not isolate competition when changing target spacing also changes source distance. For the symmetric two-target reduction, a competition-normalized ratio cancels the forcing exactly and depends only on the target--target interaction. Fourth, reset and continued-search protocols can be distinguished from endpoint allocations primarily when retention is heterogeneous; time-resolved uptake is generically informative even under
homogeneous retention.

Renewal theory for partially reactive targets, encounter-based adsorption, Dirichlet-to-Neumann reductions, and narrow-capture asymptotics are established methods \cite{bressloff2025diffusion,bressloff2025asymptotic,
bressloff2026renewal,bressloff2022spectral,bressloff2025random,wang2026breakdown,
bressloff2021asymptotic,grebenkov2020diffusion}. In particular, the target-to-target re-adsorption matrix, its Neumann resummation, and its matched-asymptotic evaluation are inherited from
\cite{bressloff2025asymptotic}. The reset--continued-search interpolation is likewise built from standard resetting kernels \cite{kusmierz2014first,riascos2020random,gonzalez2021diffusive,yu2026rigorous,
janson2012hitting}. We use these ingredients rather than claim them individually. Our contribution is their composition into a verified transport-to-phenotype null model for liver-cancer organoid arrays, together with the theoretical, numerical, and experimental-design consequences of that composition.

\paragraph{Contributions.}
\begin{enumerate}[itemsep=2pt]
\item We formulate a transport-only null that maps measured organoid centres and dosing geometry to productive uptake and phenotype variation under identical biology, and give a direct recipe for comparison with an observed array (\S\ref{sec:null}).

\item We prove that spatially uniform dosing makes the transformed forcing identical at every target in the reduced model. After matching array-mean exposure, uniform dosing reduces the median maturation spread by a factor of $5.0$ in the baseline ensemble; the corresponding point-to-uniform ratio ranges
from $2.1$ to $8.4$ across the tested designs (\S\ref{subsec:dosing_design}).

\item We separate conservative uptake from loss-mediated retention, quantify residence-law-shape effects, isolate competition from source distance, and derive and exercise an identifiability result for partial remixing (Corollary~\ref{cor:remix} and \S\ref{subsec:retention}).

\item We verify the spatial reduction against an exact radial benchmark and full finite-element solutions of interacting perforated-domain problems with $N=2,5,10$ targets in conservative and lossy regimes. For circular targets the productive-flux error is $O(\varepsilon^2)$. We separately validate the time-domain reconstruction by numerical Laplace inversion and audit parameter uncertainty using Latin hypercube sampling, partial rank correlations, and Sobol indices (\S\ref{subsec:verify}--\ref{subsec:gammaverify} and \S\ref{subsec:sens}).

\item We propagate transport and biological uncertainty to a batch-acceptance criterion. Increasing batch size averages only independent within-batch variation; shared line and batch effects yield a generally nondegenerate large-batch acceptance probability (\S\ref{sec:acceptance}).
\end{enumerate}

The model is intended as a testable baseline rather than a calibrated digital twin. Its parameters are literature-scaled, and quantitative application to a specific organoid line requires calibration and uncertainty propagation. The central question is deliberately narrower: before observed heterogeneity is
attributed to biology, how much can the measured geometry and delivery protocol already explain?

\section{The model}
\label{sec:model}

Let $\Omega\subset\R^d$, $d=2,3$, be a bounded extracellular domain containing $N$ small, well-separated compartments $\mathcal U_j=\{\mathbf x:|\mathbf x-\mathbf x_j|<\varepsilon\ell_j\}$, $j=1,\dots,N$, with $0<\varepsilon\ll1$; the extracellular region is $\Omega_\varepsilon=\Omega\setminus\bigcup_j\mathcal U_j$. For species $m=1,\dots,M$ let $u_m(\mathbf x,t)$ be the extracellular concentration. The internal state of compartment $j$ is $\boldsymbol\upsilon_j=(\upsilon_{j,\mathrm{pr}},\upsilon_{j,\mathrm{df}}, \upsilon_{j,\mathrm{mt}},\upsilon_{j,\mathrm{ap}},\upsilon_{j,\mathrm{ds}})$: progenitor, differentiated and mature fractions, an apoptotic/stressed fraction, and a disease-phenotype score.

\subsection{Bulk transport and encounter-dependent adsorption}
In $\Omega_\varepsilon$, 
\begin{equation}
\partial_tu_m=D_m\Delta u_m-\gamma_mu_m+I_m(\mathbf x,t),\qquad
D_m\nabla u_m\cdot\mathbf n=0\ \text{ on }\partial\Omega,
\label{eq:bulk}
\end{equation}
with initial data $u_m(\cdot,0)=u_{m,0}$ and $q_{jm}(\cdot,0)=q_{jm,0}$, where $q_{jm}$ is the surface-bound density on $\partial\mathcal U_j$ introduced next.

Let $\mathbf n_j$ be the unit normal on $\partial\mathcal U_j$ directed into the extracellular domain. For $\mathbf y\in\partial\mathcal U_j$, 
\begin{subequations}\label{eq:surface}
\begin{align}
D_m\nabla u_m\cdot\mathbf n_j&=\kappa_{jm}u_m-\gamma^d_{jm}q_{jm},\\
\partial_tq_{jm}&=\kappa_{jm}u_m-(\gamma^d_{jm}+\bar\gamma_{jm})q_{jm},
\end{align}
\end{subequations}
with adsorption rate $\kappa_{jm}$, desorption rate $\gamma^d_{jm}$ and productive internalization rate $\bar\gamma_{jm}$. The flux balance equates the diffusive flux leaving the fluid to the net exchange; the term $\bar\gamma_{jm}q_{jm}$ removes bound molecules without returning them, and that asymmetry is the modelling distinction on which everything else rests. With amount in moles and length in $L$, $[u_m]=\mathrm{mol}\,L^{-d}$, $[q_{jm}]=\mathrm{mol}\,L^{-(d-1)}$, $[\kappa_{jm}]=LT^{-1}$, and $[\gamma_m]=[\gamma^d_{jm}]=[\bar\gamma_{jm}]=T^{-1}$.

The biologically effective input is the productive internalization flux
\begin{equation*}
\mathcal J_{jm}(t)=\bar\gamma_{jm}\int_{\partial\mathcal U_j}q_{jm}\dd S,
\qquad[\mathcal J_{jm}]=\mathrm{mol}\,T^{-1}.
\end{equation*}

\begin{remark}[Mass balance]\label{rem:mass}
With $B_m^{\mathrm{bulk}}=\int_{\Omega_\varepsilon}u_m$, $B_{jm}^{\mathrm{surf}}=\int_{\partial\mathcal U_j}q_{jm}$ and $C_{jm}^{\mathrm{int}}=\int_0^t\mathcal J_{jm}$, the divergence theorem and \eqref{eq:surface} give $\dot B_{jm}^{\mathrm{surf}}=\kappa_{jm}\int_{\partial\mathcal U_j}u_m-(\gamma^d_{jm}+\bar\gamma_{jm})B_{jm}^{\mathrm{surf}}$ and $\dot C_{jm}^{\mathrm{int}}=\bar\gamma_{jm}B_{jm}^{\mathrm{surf}}$; adsorption and desorption cancel on summation, so $\mathcal S_m=B_m^{\mathrm{bulk}}+\sum_jB_{jm}^{\mathrm{surf}}$ obeys 
\begin{equation}
\dot{\mathcal S}_m=\int_{\Omega_\varepsilon}I_m-\gamma_m\!\int_{\Omega_\varepsilon}\!u_m-\sum_j\mathcal J_{jm}.
\label{eq:massbal}
\end{equation}
After a pulse with $I_m=0$, $\mathcal S_m$ is nonincreasing, whereas the free state need not be nonincreasing, because desorption returns molecules to the free state. 
\end{remark}

We introduce the encounter-dependent generalization.
A constant adsorption rate $\kappa_{jm}$ in \eqref{eq:surface} implicitly assumes memoryless surface contacts. To incorporate progressive receptor activation or surface remodelling, we follow \cite{bressloff2025diffusion,yu2026beyond} by introducing the boundary local time $\mathcal L^\partial_j(t)$ at $\partial\mathcal U_j$. Adsorption occurs when $\mathcal L^\partial_j(t)$ exceeds an independent threshold $\widehat\Lambda^\partial_{jm}$ with survival function $\Psi^\partial_{jm}(\lambda)$ and density $\psi^\partial_{jm}(\lambda)$. An exponential threshold $\Psi^\partial_{jm}(\lambda)=e^{-\kappa_{jm}\lambda/D_m}$ recovers the classical Robin law \eqref{eq:surface}, whereas a non-exponential threshold captures non-Markovian encounter dynamics. Spectrally, if $\mu_{n,m}(s)$ is a Dirichlet-to-Neumann eigenvalue for an isolated extended target \cite{bressloff2022spectral,yu2026microscopic}, the standard Robin factor $\vartheta_{n,jm}(s)=(\kappa_{jm}/D_m)/(\mu_{n,m}(s)+\kappa_{jm}/D_m)$ is replaced by $\vartheta^\partial_{n,jm}(s)=\widetilde\psi^\partial_{jm}(\mu_{n,m}(s))$, which reduces to the former in the exponential case. Stochastically gated accessibility is a distinct mechanism acting before adsorption; no result below uses a gated calculation, and we collect it in Appendix~\ref{app:gating}.

\subsection{Surface residence, re-adsorption, and resetting}
\label{subsec:surface_kinetics}

After adsorption, a molecule stays bound for a random residence time with density $\phi_{jm}$, after which it either desorbs with probability $\sigma_{jm}$ or is productively internalized with probability $1-\sigma_{jm}$. In the Markovian limit, $\phi_{jm}(\tau)=(\gamma^d_{jm}+\bar\gamma_{jm})e^{-(\gamma^d_{jm}+\bar\gamma_{jm})\tau}$ and $\sigma_{jm}=\gamma^d_{jm}/(\gamma^d_{jm}+\bar\gamma_{jm})$, so that
\begin{equation}
\langle\tau\rangle_{jm}=\frac1{\gamma^d_{jm}+\bar\gamma_{jm}},\qquad
\gamma^d_{jm}=\frac{\sigma_{jm}}{\langle\tau\rangle_{jm}},\qquad
\bar\gamma_{jm}=\frac{1-\sigma_{jm}}{\langle\tau\rangle_{jm}} .
\label{eq:kinetic}
\end{equation}

More generally, multistate transient kinetics can be represented via an underlying Markov chain. If an adsorbed molecule enters the first of $K_{jm}$ transient states with generator $\mathbf K^{\mathrm{res}}_{jm}$ (where $\mathbf 1^\top\mathbf K^{\mathrm{res}}_{jm}+(\mathbf r^d_{jm}+\mathbf r^i_{jm})^\top=\mathbf 0^\top$) and exit-rate vectors $\mathbf r^{d}_{jm}$, $\mathbf r^{i}_{jm}$, then the Laplace-transformed residence densities are $\widetilde\phi^{x}_{jm}(s)=(\mathbf r^{x}_{jm})^\top(s\mathbf I-\mathbf K^{\mathrm{res}}_{jm})^{-1}\mathbf e_1$ for $x\in\{d,i\}$.

We introduce spatial rebinding and search interpolation.
When a desorbed molecule leaves the surface, it re-enters the fluid to resume searching. Under complete, well-mixed resetting, the molecule forgets its previous encounter location. Conditioning on the first search cycle with kernels $\mathfrak S_{jm}=(1-\sigma_{jm})\phi_{jm}*J_{jm}$ and $\mathfrak F_m=\sum_k\sigma_{km}\phi_{km}*J_{km}$ gives $\mathcal J_{jm}=\mathfrak S_{jm}+\mathfrak F_m*\mathcal J_{jm}$, yielding the transformed productive flux
\begin{equation}
\widetilde{\mathcal J}_{jm}(s)
=\frac{(1-\sigma_{jm})\widetilde\phi_{jm}(s)\widetilde J_{jm}(s)}
{1-\sum_{k}\sigma_{km}\widetilde\phi_{km}(s)\widetilde J_{km}(s)} ,
\label{eq:reset}
\end{equation}
where $J_{jm}$ is the irreversible adsorption flux, and $\widetilde J_{jm}(0)=\overline\pi_{jm}$ is the first-adsorption splitting probability.

To retain spatial memory of the encounter location, let $P_{jk,m}(s)$ be the transformed irreversible adsorption flux into organoid $j$ for a search originating from a surface-averaged desorption point on organoid $k$; the surface average is the leading monopole approximation, exact to $O(\varepsilon)$ for
circular targets with homogeneous surface kinetics. Defining $\bm\Sigma_m=\diag(\sigma_{jm})$, $\bm W_m=\mathbf I-\bm\Sigma_m$, and $\bm\Phi_m(s)=\diag(\widetilde\phi_{jm}(s))$, summing over successive desorption and rebinding events gives \cite{bressloff2025asymptotic,cai2026optimal}
\begin{equation}
\widetilde{\bm{\mathcal J}}^{\mathrm{cont}}_m
=\bm W_m\bm\Phi_m\bigl[\mathbf I-\mathbf P_m\bm\Sigma_m\bm\Phi_m\bigr]^{-1}
\widetilde{\mathbf J}_m ,
\label{eq:cont}
\end{equation}
which converges because $\|\mathbf P_m(s)\bm\Sigma_m\|_1\le\max_k\sigma_{km}<1$ for $s\ge0$: the entries of $\mathbf P_m(s)$ are Laplace transforms of nonnegative sub-probability densities, so its column sums are bounded by their values at $s=0$. Equality $\|\mathbf P_m(s)\bm\Sigma_m\|_1 =\max_k\sigma_{km}$ is in general justified only at $s=0$ under Assumption~\ref{ass:cons}, where $\mathbf P_m(0)$ is column stochastic; for $s>0$, or in the lossy case, the column sums are strictly smaller and the bound is strict.  

To account for partial spatial remixing near organoid surfaces, we interpolate between well-mixed resetting ($\mathbf P^{\mathrm{reset}}_m(s)=\widetilde{\mathbf J}_m(\mathbf x_0,s)\mathbf1^\top$) and spatial continuation \eqref{eq:cont} via
\begin{equation*}
\mathbf P^{(\eta)}_m(s)=(1-\eta_m)\mathbf P^{\mathrm{reset}}_m(s)
+\eta_m\mathbf P^{\mathrm{cont}}_m(s),\qquad 0\le\eta_m\le1 ,
\end{equation*}
where $1-\eta_m$ represents the fraction of dissociations leading to true separation \cite{goldstein1995approximating,goldstein1999influence,wang2026elliptic}. For well-separated two-dimensional targets ($\nu=-1/\log\varepsilon$), same-organoid rebinding scales as $P_{kk,m}(0)=1-O(\nu)$ while inter-target transfer scales as $P_{jk,m}(0)=O(\nu)$ ($j\neq k$) \cite{bressloff2025asymptotic,liu2025bidirectional}. Thus, local re-adsorption remains $O(1)$ and cannot be neglected as $\varepsilon\to0$.

\subsection{Intracellular state dynamics and phenotype coupling}
\label{subsec:state}

Productive internalization provides the coupling from extracellular transport to the intracellular phenotype. In general, we write the state dynamics of organoid $j$ as
\begin{equation*}
\frac{\dd\boldsymbol\upsilon_j}{\dd t}
=
\mathbf G\!\left(
\boldsymbol\upsilon_j,
\{\mathcal J_{jm}(t)\}_{m=1}^{M};
\boldsymbol\theta_j
\right),
\end{equation*}
where $\boldsymbol\theta_j$ collects the biological parameters and $\mathbf G$ is locally Lipschitz in the state and may depend nonlinearly on the productive fluxes. This formulation includes both flux-linear responses and the saturating Hill response used below.

If the model contains a single active species, the species index $m$ is suppressed and we write $\mathcal J_j(t)$, or simply $\mathcal J(t)$ when the organoid index is fixed. If several active species contribute to a common phenotypic response, the effective input must be defined explicitly. Here we use the weighted total productive flux
\begin{equation*}
\mathcal J_j(t)
=
\sum_{m=1}^{M}\omega_m\mathcal J_{jm}(t),
\qquad
\omega_m\geq0,
\end{equation*}
where the fixed weights $\omega_m$ account for relative potency and, when necessary, unit conversion. The single-species case corresponds to $M=1$ and $\omega_1=1$.

The concrete model used throughout consists of the maturation cascade
\[
\upsilon_{j,\mathrm{pr}}
\longrightarrow
\upsilon_{j,\mathrm{df}}
\longrightarrow
\upsilon_{j,\mathrm{mt}},
\]
coupled to an apoptotic/stress burden $\upsilon_{j,\mathrm{ap}}$ and a dimensionless disease score $\upsilon_{j,\mathrm{ds}}$: 
\begin{subequations}
\label{eq:transformed_problem}
\begin{align}
\dot\upsilon_{j,\mathrm{pr}}
&=
r_{\mathrm{grow}}\upsilon_{j,\mathrm{pr}}
\left[
1-
\left(
\upsilon_{j,\mathrm{pr}}
+\upsilon_{j,\mathrm{df}}
+\upsilon_{j,\mathrm{mt}}
\right)
\right]
-
(k_{PD}+b_{PD}\mathcal J_j)\upsilon_{j,\mathrm{pr}},\\
\dot\upsilon_{j,\mathrm{df}} &=
(k_{PD}+b_{PD}\mathcal J_j)\upsilon_{j,\mathrm{pr}}
-
(k_{DM}+b_{DM}\mathcal J_j)\upsilon_{j,\mathrm{df}}
-
d_0\upsilon_{j,\mathrm{ds}}\upsilon_{j,\mathrm{df}},\\
\dot\upsilon_{j,\mathrm{mt}} &=
(k_{DM}+b_{DM}\mathcal J_j)\upsilon_{j,\mathrm{df}}
-
d_0\upsilon_{j,\mathrm{ds}}\upsilon_{j,\mathrm{mt}}, \label{eq:state}\\
\dot\upsilon_{j,\mathrm{ap}} &=
d_0\upsilon_{j,\mathrm{ds}}
\left(
\upsilon_{j,\mathrm{df}}
+\upsilon_{j,\mathrm{mt}}
\right)
+
k_{\mathrm{stress}}\upsilon_{j,\mathrm{ds}}
-
k_{\mathrm{ap,clr}}\upsilon_{j,\mathrm{ap}},\\
\dot\upsilon_{j,\mathrm{ds}} &=
\alpha_Q\upsilon_{j,\mathrm{ds}}
(1-\upsilon_{j,\mathrm{ds}})
-
\alpha_R
\frac{\mathcal J_j^{\,n_{\mathrm H}}}
{\Theta^{n_{\mathrm H}}+\mathcal J_j^{\,n_{\mathrm H}}}
\upsilon_{j,\mathrm{ds}}.
\end{align}
\end{subequations}
The first three equations describe proliferation followed by irreversible differentiation and maturation. Productive uptake increases both transition rates through $b_{PD}\mathcal J_j$ and $b_{DM}\mathcal J_j$, whereas the disease score induces loss of differentiated and mature cells. The fourth equation records the resulting apoptotic/stress burden and its clearance. The last equation combines logistic disease progression with a saturating rescue response. The Hill function
\[
H(\mathcal J_j)
=
\frac{\mathcal J_j^{\,n_{\mathrm H}}}
{\Theta^{n_{\mathrm H}}+\mathcal J_j^{\,n_{\mathrm H}}}
\]
has half-maximal response at $\mathcal J_j=\Theta$, maximal rescue rate $\alpha_R$, and steepness controlled by $n_{\mathrm H}$.

The rescue term depends on productive intracellular uptake $\mathcal J_j$, not directly on the extracellular concentration $u_m$. This is the central transport-to-phenotype assumption: extracellular availability affects the state only after arrival, adsorption, surface processing, and productive internalization have occurred.

For completeness, let
\[
S_j
=
\upsilon_{j,\mathrm{pr}}
+\upsilon_{j,\mathrm{df}}
+\upsilon_{j,\mathrm{mt}}.
\]
Summing the first three equations gives
\begin{equation*}
\dot S_j
=
r_{\mathrm{grow}}\upsilon_{j,\mathrm{pr}}(1-S_j)
-
d_0\upsilon_{j,\mathrm{ds}}
\left(
\upsilon_{j,\mathrm{df}}
+\upsilon_{j,\mathrm{mt}}
\right).
\end{equation*}
Hence $\dot S_j\leq0$ on $S_j=1$. Together with the inward-pointing vector field on the coordinate boundaries and
\[
\left.\dot\upsilon_{j,\mathrm{ds}}\right|_{\upsilon_{j,\mathrm{ds}}=1}
=
-\alpha_RH(\mathcal J_j)\leq0,
\]
this shows that nonnegative initial data satisfying $S_j(0)\leq1$ and $0\leq\upsilon_{j,\mathrm{ds}}(0)\leq1$ remain in those bounds. Moreover,
\[
\dot\upsilon_{j,\mathrm{ap}}
\leq
d_0+k_{\mathrm{stress}}
-k_{\mathrm{ap,clr}}\upsilon_{j,\mathrm{ap}},
\]
so $\upsilon_{j,\mathrm{ap}}$ is uniformly bounded whenever $k_{\mathrm{ap,clr}}>0$.
Choose $B_{\mathrm{ap}}\geq \dfrac{d_0+k_{\mathrm{stress}}}{k_{\mathrm{ap,clr}}}$
and define
\begin{equation*}
\mathcal K=\left\{(p,d,m,a,q)\in\mathbb R_+^5:
p+d+m\leq1,\ 0\leq q\leq1,\ 0\leq a\leq B_{\mathrm{ap}}\right\}.
\label{eq:app_invariant_set}
\end{equation*}
On $\mathcal K$, the right-hand side of \eqref{eq:transformed_problem} is Lipschitz in the state and in bounded productive inputs, providing the existence, uniqueness, and continuous-dependence properties required by Proposition~\ref{prop:lip}:
\begin{proposition}[Positive invariance]
\label{prop:app_positive_invariance}
For every initial state in $\mathcal K$ and every
$\mathcal J\in L^1_+(0,T)$, system \eqref{eq:transformed_problem} has a unique
absolutely continuous solution on $[0,T]$, and that solution remains in
$\mathcal K$.
\end{proposition}

The Hill function is globally Lipschitz on $[0,\infty)$. Its sharp derivative
bound is
\begin{equation}
L_H=\sup_{x\geq0}|H'(x)|=
\begin{cases}
\Theta^{-1},&n_{\mathrm H}=1,\\[2mm]
\dfrac{(n_{\mathrm H}+1)^2}{4n_{\mathrm H}\Theta}
\left(\dfrac{n_{\mathrm H}-1}{n_{\mathrm H}+1}\right)^{
(n_{\mathrm H}-1)/n_{\mathrm H}},&n_{\mathrm H}>1.
\end{cases}
\label{eq:app_Hill_Lipschitz}
\end{equation}
Set
\begin{align}
L_0=\max\{&3r_{\mathrm{grow}}+k_{PD},
k_{PD}+k_{DM}+2d_0,
k_{DM}+2d_0,\nonumber\\
&3d_0+k_{\mathrm{stress}}+k_{\mathrm{ap,clr}},
\alpha_Q+\alpha_R\},\qquad
L_1=b_{PD}+b_{DM},
\label{eq:app_L_constants}\\
C_{\mathcal J}&=\max\{b_{PD}+b_{DM},\alpha_RL_H\}.
\end{align}

\begin{proposition}[Lipschitz transport-to-state map]
\label{prop:app_input_state}
Let $\boldsymbol\upsilon$ and $\widehat{\boldsymbol\upsilon}$ solve
\eqref{eq:transformed_problem} with inputs $\mathcal J,\widehat{\mathcal J}\in
L^1_+(0,T)$ and initial states in $\mathcal K$. Put
\[
\mathcal J_*(t)=\max\{\mathcal J(t),\widehat{\mathcal J}(t)\},
\qquad M_T=\int_0^T\mathcal J_*(t)\,\mathrm dt.
\]
Then
\begin{equation}
\|\boldsymbol\upsilon(t)-\widehat{\boldsymbol\upsilon}(t)\|_\infty
\leq \exp\!\left(L_0t+L_1\int_0^t\mathcal J_*(\tau)\,\mathrm d\tau\right) \times\left[
\|\boldsymbol\upsilon(0)-\widehat{\boldsymbol\upsilon}(0)\|_\infty
+C_{\mathcal J}\int_0^t|\mathcal J-\widehat{\mathcal J}|\,\mathrm d\tau
\right].
\label{eq:app_input_state_bound}
\end{equation}
Consequently,
\begin{equation}
\sup_{0\leq t\leq T}
\|\boldsymbol\upsilon(t)-\widehat{\boldsymbol\upsilon}(t)\|_\infty
\leq C_T\left[
\|\boldsymbol\upsilon(0)-\widehat{\boldsymbol\upsilon}(0)\|_\infty
+\|\mathcal J-\widehat{\mathcal J}\|_{L^1(0,T)}\right],
\label{eq:app_gronwall_bound}
\end{equation}
where $C_T=e^{L_0T+L_1M_T}\max\{1,C_{\mathcal J}\}$.
\end{proposition}

\begin{proof}
On $\mathcal K$, the maximum row sum of the state Jacobian is bounded by
$L_0+L_1\mathcal J$. At fixed state, the input perturbation is bounded by
$C_{\mathcal J}|\mathcal J-\widehat{\mathcal J}|$. Therefore the upper Dini
derivative satisfies
\[
D^+\|\boldsymbol\upsilon-\widehat{\boldsymbol\upsilon}\|_\infty
\leq(L_0+L_1\mathcal J_*)
\|\boldsymbol\upsilon-\widehat{\boldsymbol\upsilon}\|_\infty
+C_{\mathcal J}|\mathcal J-\widehat{\mathcal J}|.
\]
Apply the integral form of Gronwall's inequality.
\end{proof}

All twelve kinetic constants and the initial state are fixed in \S\ref{sec:numerics}. In the transport-only null model they are also held common across organoids, so between-organoid differences arise from the
geometry-dependent productive fluxes rather than from imposed biological heterogeneity. Figure~\ref{fig:schematic} summarizes the three layers of the model --- bulk transport, surface kinetics and the cell state --- together with the nested variability model of \S\ref{sec:acceptance}.  

\begin{figure}[htbp]
\centering
\includegraphics[width=\linewidth]{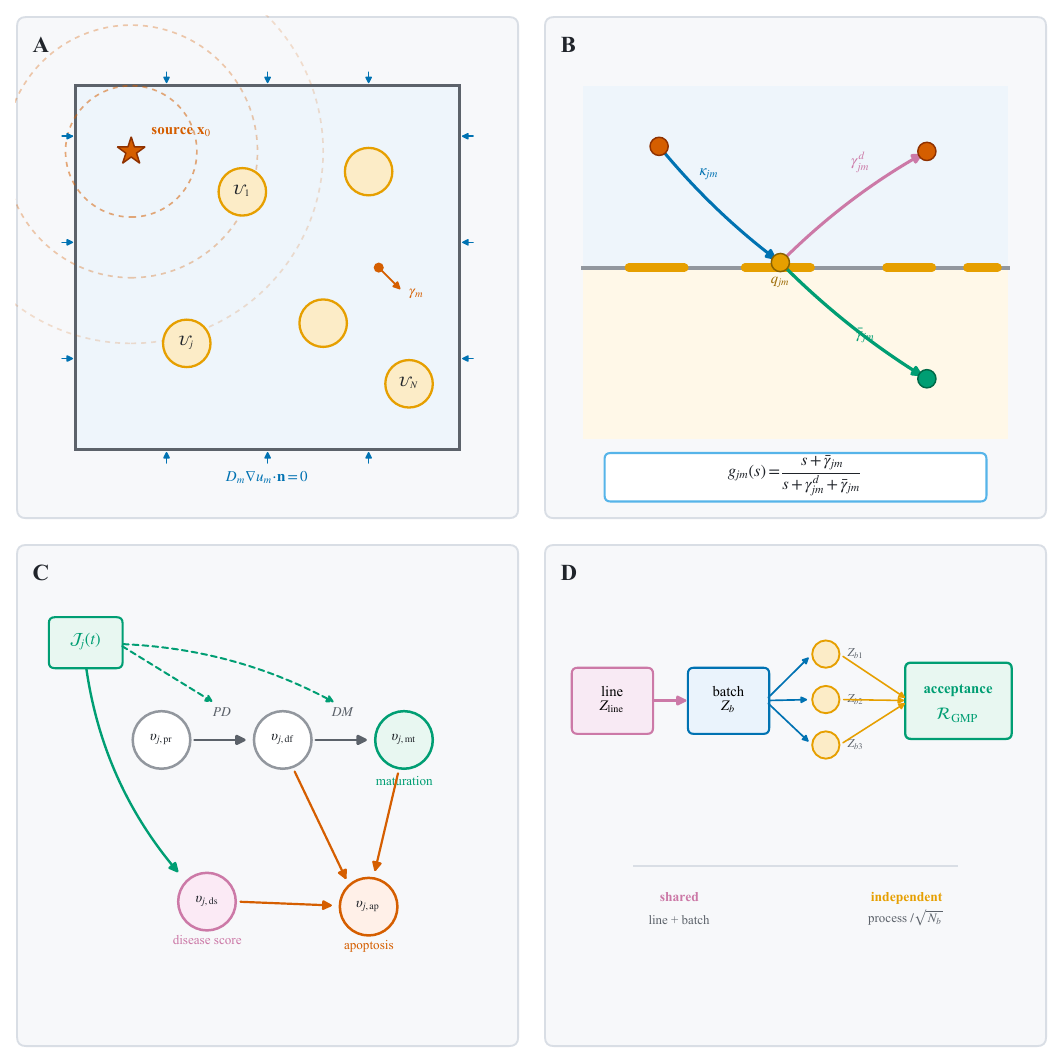}
\caption{Model schematic. (A) Small, well-separated organoids $\mathcal U_j$ in a bounded chamber are exposed to a localized dose at $\mathbf x_0$. The outer boundary is reflecting, whereas bulk clearance
$\gamma_m$ acts volumetrically. (B) At an accessible surface patch, a freely diffusing molecule adsorbs, subsequently desorbs or is productively internalized, producing the frequency-dependent filter $g_{jm}(s)$ in \eqref{eq:grobin}. (C) The productive flux $\mathcal J_j(t)$ modulates the progenitor--differentiated--mature cascade, suppresses the disease score and thereby changes apoptosis; extracellular concentration is not used directly as the cellular input. (D) Shared line and batch effects and independent organoid-level process effects propagate through the phenotype model to the batch acceptance probability $\mathcal R_{\mathrm{GMP}}$. Increasing batch size averages only the independent process component.}
\label{fig:schematic}
\end{figure}

\section{Reduction and formulae}
\label{sec:methods}
Two hypotheses recur. 
\begin{assumption}[Dilute, well separated]\label{ass:dilute}
$r_j=\varepsilon\ell_j$ with $\varepsilon\to0$, and there is a fixed
$d_*>0$, independent of $\varepsilon$, with
\begin{equation}
\min_{i\neq j}|\mathbf x_i-\mathbf x_j|\ \ge\ d_*,\qquad
\min_j\dist(\mathbf x_j,\partial\Omega)\ \ge\ d_*,\qquad
\min_j\dist(\mathbf x_j,\mathcal S_m)\ \ge\ d_*
\label{eq:sep}
\end{equation}
for every $\varepsilon$ under consideration, where the source datum is assumed to split into a bounded part and finitely many atoms, 
\begin{equation*}
f^{\mathrm{src}}_m=u_{m,0}+\widetilde I_m
=f^{\mathrm{reg}}_m+\sum_{\alpha=1}^{K_m}q_{\alpha m}\,
\delta_{\mathbf y_{\alpha m}},\qquad f^{\mathrm{reg}}_m\in L^\infty(\Omega),
\end{equation*}
and $\mathcal S_m=\{\mathbf y_{\alpha m}\}_{\alpha=1}^{K_m}$ is the set of atoms.
\end{assumption}

Three points about \eqref{eq:sep}. It is stated as a lower bound rather than as $O(1)$ because $O(1)$ is an upper bound and does not exclude separations that shrink with $\varepsilon$. The third condition is a condition on the atoms of the source only: what it controls is the boundedness of $\Gamma_m$ on the matching annuli, and a bounded $f^{\mathrm{reg}}_m$ delivers that however large its support, so for the localized port of \eqref{eq:dosing} $K_m=1$ and the condition reads $\min_j|\mathbf x_j-\mathbf x_0|\ge d_*$, while for the spatially uniform source of \S\ref{subsec:dosing_design} $K_m=0$ and it is vacuous. And that third condition is independent of the first two: an array can be perfectly well separated internally and still violate it.

\begin{assumption}[Conservative search]\label{ass:cons}
The outer boundary is reflecting and $\gamma_m=0$, so every released molecule is eventually adsorbed and $\sum_k\overline\pi_{km}=1$.
\end{assumption}

Assumption~\ref{ass:cons} is stated separately because it is exactly the hypothesis under which the normalized renewal formulae hold; experiments with $\gamma_m>0$ are governed instead by Remark~\ref{rem:lossy} and are
labelled lossy throughout.

\subsection{Robin condition and matched asymptotics in two dimensions}

To translate the surface exchange dynamics into an effective boundary condition for the bulk field under Assumption 1, we write $\chi_{jm}(s)=s+\gamma^d_{jm}+\bar\gamma_{jm}$ and transform \eqref{eq:surface} with $q_{jm,0}=0$, yielding
\begin{equation}
D_m\nabla\widetilde u_m\cdot\mathbf n_j=\kappa_{jm}g_{jm}(s)\widetilde u_m,
\qquad
g_{jm}(s)=\frac{s+\bar\gamma_{jm}}{\chi_{jm}(s)},\qquad
g_{jm}(0)=1-\sigma_{jm}.
\label{eq:grobin}
\end{equation}

The boundary is thus a frequency-dependent reactive filter, not a static sink; $\gamma^d_{jm}=0$ recovers the classical Robin condition. A nonzero initially bound population adds $ \gamma^d_{jm}q_{jm,0}/\chi_{jm}$ to the transformed flux balance and $\bar\gamma_{jm}q_{jm,0}/\chi_{jm}$ to the output.

With the effective reactive boundary established, we now construct the global multiple-target solution via matched asymptotics. Let $\nu=-1/\log\varepsilon$ and let $G_m(\mathbf x,\mathbf z;s)$ solve $D_m\Delta G_m-(s+\gamma_m)G_m=-\delta(\mathbf x-\mathbf z)$ with reflecting outer boundary, with local structure $G_m=-\log|\mathbf x-\mathbf z|/(2\pi D_m)+R_m$. Put $f_m^{\mathrm{src}}=u_{m,0}+\widetilde I_m$ and $\Gamma_m(\mathbf x,s)=\int_\Omega G_mf_m^{\mathrm{src}}$. Since the physical radius is $r_j=\varepsilon\ell_j$, the scaling $\kappa_{jm}=\kappa'_{jm}/\varepsilon$ holds the surface Damk\"ohler number fixed: \begin{equation} \mathrm{Da}^\partial_{jm}=\frac{\kappa_{jm}r_j}{D_m} =\frac{\kappa'_{jm}\ell_j}{D_m}=O(1). \label{eq:damkohler} 
\end{equation} 
This constant $\mathrm{Da}^\partial_{jm}$ ensures that the apparent divergence of $\kappa_{jm}$ precisely compensates for the shrinking radius. Consequently, the reactive term enters at the same $O(\nu)$ order as the target--target interaction. By matching the outer representation $\widetilde u_m\sim\Gamma_m-2\pi D_m\sum_kA_{km}G_m(\cdot,\mathbf x_k)$ to the inner Robin solution yields the finite-dimensional amplitude system
\begin{equation}
\mathbf A_m(s)=\nu\bigl[\mathbf I+2\pi\nu D_m\bm{\mathcal G}_m(s)
+\nu\bm\Psi_m(s)\bigr]^{-1}\bm\Gamma_m(s),
\label{eq:amp}
\end{equation}
with $\Psi_{jm}(s)=D_m/(\kappa'_{jm}g_{jm}(s)\ell_j)$ and
\begin{equation}
(\bm{\mathcal G}_m)_{jk}(s)=
\begin{cases} G_m(\mathbf x_j,\mathbf x_k;s),& j\neq k,\\
R_m(\mathbf x_j,\mathbf x_j;s)-\log\ell_j/(2\pi D_m),& j=k.\end{cases}
\label{eq:gmat}
\end{equation}

Because \eqref{eq:amp} is kept non-perturbative in $\nu$ it resums the
logarithms; the residual is algebraic in $\varepsilon$. Using
$\widetilde U_{jm}(\ell_j,s)=A_{jm}\Psi_{jm}$ and
$|\partial\mathcal U_j|=2\pi\varepsilon\ell_j$, the factors of $g_{jm}$,
$\gamma^d_{jm}$ and $\kappa'_{jm}$ cancel exactly, since
$\chi_{jm}g_{jm}=s+\bar\gamma_{jm}$:
\begin{equation}
\widetilde{\mathcal J}_{jm}(\mathbf x_0,s)\sim
\frac{2\pi D_m\bar\gamma_{jm}}{s+\bar\gamma_{jm}}\,A_{jm}(\mathbf x_0,s).
\label{eq:fluxred}
\end{equation}

The ratio of productive flux to net boundary flux is $\bar\gamma_{jm}/(s+\bar\gamma_{jm})$ exactly, at every order and in every dimension; reversible kinetics still act indirectly through $g_{jm}$ inside $\Psi_{jm}$.

\subsection{Probabilities and times}
\label{subsec:times}
To quantify the long-term behavior of a unit localized release, we evaluate the internalization probabilities and mean conditional times via the $s\to0$ limit of the Laplace-transformed fluxes. Specifically, we define $\pi_{jm}(\mathbf x_0)=\lim_{s\to0}\widetilde{\mathcal J}_{jm}$ and $\mathcal T^{\mathrm{cond}}_{jm}=-\pi_{jm}^{-1}\partial_s\widetilde{\mathcal J}_{jm}\vert{}_{s=0}$.

\begin{remark}[The $s\to0$ limit]\label{rem:pseudo}
For $\gamma_m>0$, these limits are literal. Under Assumption~\ref{ass:cons}, the operator has a zero mode and $G_m=[(s+\gamma_m)\vert{}\Omega\vert{}]^{-1}+G^0_m+O(s)$, with $G^0_m$ being the Neumann pseudo-Green's function. Because the singular mode enters $\bm\Gamma_m$ and every row of $\bm{\mathcal G}_m$ identically, it cancels out in \eqref{eq:amp}. Consequently, the limits remain finite and are evaluated through $G^0_m$ rather than factor by factor.
\end{remark}

Evaluating this limit for the reset protocol in \eqref{eq:reset} yields the steady-state probabilities. In general, taking $s\to0$ gives $\pi_{jm}=(1-\sigma_{jm})\overline\pi_{jm}/(1-\sum_k\sigma_{km}\overline\pi_{km})$. Under the conservative setting of Assumption~\ref{ass:cons}, this simplifies to:
\begin{equation}
\pi_{jm}=\frac{(1-\sigma_{jm})\overline\pi_{jm}}{\sum_k(1-\sigma_{km})\overline\pi_{km}},
\qquad
\mathcal T^{\mathrm{tot}}m=\frac{\mathcal T^{\mathrm{ads}}m
+\sum_k\overline\pi{km}\langle\tau\rangle{km}}{\sum_k(1-\sigma_{km})\overline\pi_{km}},
\label{eq:renewalpi}
\end{equation}
where $\mathcal T^{\mathrm{ads}}_m$ is the mean first-adsorption time. Crucially, the numerator of $\mathcal T^{\mathrm{tot}}_m$ explicitly accounts for the diffusive search; omitting it leads to an underestimate of the internalization time.

To further characterize the time distribution, we extract the higher moments by defining the generating functions for successful and failed cycles: $\mathfrak S_m(s)=\sum_k(1-\sigma_{km})\widetilde\phi_{km}\widetilde J_{km}$ and $\mathfrak F_m(s)=\sum_k\sigma_{km}\widetilde\phi_{km}\widetilde J_{km}$. Let $p_m=\mathfrak S_m(0)=1-\mathfrak F_m(0)$ denote the overall success probability. Furthermore, by letting $\mathfrak K_{km}(s)=\widetilde\phi_{km}\widetilde J_{km}$, we can define the auxiliary moment terms $\mathfrak m^{(1)}_{km}=-\mathfrak K'_{km}(0)=T^{(1)}_{km}+\overline\pi_{km}\langle\tau\rangle_{km}$ and $\mathfrak m^{(2)}_{km}=\mathfrak K''_{km}(0)=T^{(2)}_{km}+2T^{(1)}_{km}\langle\tau\rangle_{km}+\overline\pi_{km}\langle\tau^2\rangle_{km}$. Differentiating $\mathfrak S_m/(1-\mathfrak F_m)$ then provides the first and second moments of the total time:
\begin{equation}
\mathcal T^{\mathrm{tot}}_m=\frac{\sum_k\mathfrak m^{(1)}_{km}}{p_m},
\qquad
\mathcal T^{\mathrm{tot}}_{m,2}=\frac{\sum_k\mathfrak m^{(2)}_{km}}{p_m}
+\frac{2\bigl(\sum_k\sigma_{km}\mathfrak m^{(1)}_{km}\bigr)
\bigl(\sum_k\mathfrak m^{(1)}_{km}\bigr)}{p_m^2}.
\label{eq:moments}
\end{equation}
Notably, the second term in $\mathcal T^{\mathrm{tot}}_{m,2}$ captures the dispersion generated by a geometric number of failed cycles; omitting this term underestimates the coefficient of variation.

\begin{assumption}[Finite moments]\label{ass:moments}
Every statement below obtained by differentiating a transform at $s=0$ --- \eqref{eq:renewalpi}, \eqref{eq:moments}, the timing clause of Corollary~\ref{cor:remix} and the entries of $\partial_s\mathbf P^{(\eta)}_m(0)$ --- presupposes that the differentiated transforms are $C^1$ (respectively $C^2$) at $s=0$ from the right. Concretely: the first-adsorption laws $\widetilde J_{jm}$, the residence laws $\widetilde\phi_{jm}$ and the re-adsorption kernels $P_{jk,m}$ have finite first moments for \eqref{eq:renewalpi} and the timing clause, and finite second moments for $\mathcal T^{\mathrm{tot}}_{m,2}$ in \eqref{eq:moments}. This is not automatic: Remark~\ref{rem:tails} exhibits the heavy-tailed regime in which it fails and the corresponding derivatives do not exist.
\end{assumption}
\begin{remark}[Heavy tails]\label{rem:tails}
If $\widetilde\phi_{jm}(s)=1-c_\tau s^\beta+o(s^\beta)$ or $\widetilde\psi^\partial_{jm}(z)=1-c_\partial z^\alpha+o(z^\alpha)$ with exponents in $(0,1)$, the corresponding means diverge and $\mathcal J_{m,\mathrm{tot}}(t)\asymp t^{-1-\min(\alpha,\beta)}$ \cite{bressloff2025diffusion,yu2026pattern}; finite-horizon uptake fractions and quantiles should then replace mean-based summaries.
\end{remark}
\begin{remark}[Lossy search]\label{rem:lossy}
Without Assumption~\ref{ass:cons}, $\sum_k\overline\pi_{km}<1$, $\sum_j\pi_{jm}<1$, and the two denominators above no longer coincide. Writing $\widetilde{\mathcal F}_m=\mathfrak S_m/(1-\mathfrak F_m)$, the eventual internalization probability is $\widetilde{\mathcal F}_m(0)<1$ and $\mathbb E[T_m\mid\text{internalization}]=-\mathfrak S'_m(0)/\mathfrak S_m(0) -\mathfrak F'_m(0)/(1-\mathfrak F_m(0))$.
\end{remark}
\begin{remark}[Per-encounter consistency is not protocol equivalence]
The identity $g_{jm}(0)=1-\sigma_{jm}$ identifies the static transmission factor of the Markovian filter with the per-encounter internalization probability of the renewal model. It does not make the two protocols equivalent: the Markovian boundary condition returns a desorbed molecule at its desorption
point and so corresponds to continued search, whereas \eqref{eq:reset} assumes complete remixing. Their one-target eventual probabilities agree under Assumption~\ref{ass:cons} because both equal one; their time laws do not.
\end{remark}

\subsection{Theoretical consequences and the 3D extensions}
\label{subsec:props}
We now formalize the core theoretical guarantees of our reduced model. We refer the reader to Appendix~\ref{app:proofs} for their proofs.

\begin{proposition}[Reduced uptake matrix]\label{prop:red}
Under Assumption~\ref{ass:dilute}, the transformed productive flux is determined to all logarithmic orders in $\nu$ by \eqref{eq:amp}, and away from the targets $\widetilde u_m=\Gamma_m-2\pi D_m\sum_jA_{jm}G_m(\cdot,\mathbf x_j)+O(\varepsilon)$. If in addition each $\mathcal U_j$ is circular, then for arbitrary $O(1)$-separated source placement and fixed $s>0$
\begin{equation}
\widetilde{\mathcal J}^{\,\mathrm{full}}_{jm}(s)
-\widetilde{\mathcal J}^{\,\mathrm{red}}_{jm}(s)=O(\varepsilon^2),
\label{eq:eps2}
\end{equation}
provided the second-order composite residual satisfies
\eqref{eq:app_residual_bound}. The first inner correction is a dipole and has
zero integrated normal flux; this removes the $O(\varepsilon)$ flux term but,
by itself, does not establish the uniform remainder estimate. The pointwise
field error remains $O(\varepsilon)$. Appendix~\ref{app:matched2d} gives the
derivation and precise scope of \eqref{eq:eps2}.
\end{proposition}

Having established that the reduced system captures the productive flux with $O(\varepsilon^2)$ accuracy, we can reliably use it to evaluate the long-time statistical observables.

\begin{proposition}[Probabilities and times]\label{prop:pt}
Under Assumption~\ref{ass:moments} for the mean and second-moment statements (the splitting probabilities themselves require no moment hypothesis), $\pi_{jm}$ and $\mathcal T^{\mathrm{cond}}_{jm}$ are the stated $s\to0$ limits, interpreted through Remark~\ref{rem:pseudo} when $\gamma_m=0$; for the reset model they reduce to \eqref{eq:renewalpi}, and for continued search the splitting vector is $\bm\pi^{(\eta)}_m=\bm W_m[\mathbf I-\mathbf P^{(\eta)}_m\bm\Sigma_m]^{-1} \overline{\bm\pi}_m$, normalized in the conservative case. In particular, for $N=1$ under Assumption~\ref{ass:cons},
\begin{equation}
\pi_{1m}=1\qquad\text{for every }0\le\sigma_{1m}<1 .
\label{eq:onetarget}
\end{equation}
\end{proposition}

\begin{corollary}[What partial remixing can change]\label{cor:remix}
Assume conservative search, well-separated two-dimensional targets and $\max_j\sigma_{jm}\le1-\delta$. Writing $\mathbf P^{\mathrm{cont}}_m(0)=\mathbf I+\mathbf E_m$ with $\|\mathbf E_m\|_1=O(\nu)$ and using $\bm W_m=\mathbf I-\bm\Sigma_m$ exactly,
\begin{equation}
\bm\pi^{(1)}_m=\overline{\bm\pi}_m+O(\nu),
\qquad
\bm\pi^{(0)}_m=\frac{\bm W_m\overline{\bm\pi}_m}
{\mathbf1^\top\bm W_m\overline{\bm\pi}_m},
\label{eq:endpoints}
\end{equation}
The $\eta_m=0$ expression in \eqref{eq:endpoints} is exact, whereas $\bm\pi^{(1)}_m$ carries an $O(\nu)$ residual; the two endpoints therefore differ by the normalized reweighting $\overline\pi_{jm}\mapsto(1-\sigma_{jm})\overline\pi_{jm}$ to leading order in $\nu$, not exactly. If $\bm\Sigma_m=\sigma_m\mathbf I$ then $\bm\pi^{(\eta)}_m=\overline{\bm\pi}_m+O(\nu)$ uniformly in $\eta_m$. The implied constant is proportional to $\max_j\sigma_{jm}/(1-\sigma_{jm})$, since $\bm\Sigma_m\bm W_m^{-1}=\diag(\sigma_{jm}/(1-\sigma_{jm}))$, so the expansion is controlled only when $\nu\max_j\sigma_{jm}/(1-\sigma_{jm})\ll1$. At
$\varepsilon=0.02$ ($\nu=0.256$) we use $\max_j\sigma_{jm}\lesssim0.3$ as a working range; this is an empirical accuracy guideline read off the numerical residuals of \S\ref{subsec:retention}. Under Assumption~\ref{ass:moments}, timing generally depends on $\eta_m$, through $\partial_s\mathbf
P^{(\eta)}_m(0)$, which is not annihilated by $\bm W_m$; particular geometries and residence laws may nevertheless produce exceptional cancellations in which the $\eta_m$-dependence of a given timing functional vanishes, and we do not exclude them. Consequently $\eta_m$ is identifiable from endpoint uptake fractions only in arrays with heterogeneous retention, whereas time-resolved uptake identifies
it generically --- that is, for all but exceptional $(\text{geometry},\ \text{residence law})$ pairs --- even when retention is homogeneous.
\end{corollary}

Beyond the spatial and kinetic features of the search process, evaluating biological plausibility requires analyzing how sensitive these outputs are to parameter variations:

\begin{proposition}[Well-posedness and continuous dependence]
\label{prop:lip}
Let $\Omega_\varepsilon$ be $C^2$, let $D_m>0$, and assume nonnegative
kinetic rates, nonnegative $L^2$ initial data, and
$I_m\in L^1(0,T;L^2(\Omega_\varepsilon))$. The bulk--surface system
\eqref{eq:bulk}--\eqref{eq:surface} has a unique nonnegative mild solution and
satisfies the mass identity of Remark~\ref{rem:mass}. If, in addition, the
matrix in \eqref{eq:amp} is uniformly invertible on a compact admissible
parameter set and the state remains in its compact invariant region, then every
Lipschitz output obeys
\[
\|\mathbf Z(T;\boldsymbol\theta)-\mathbf Z(T;\boldsymbol\theta_0)\|
\leq C_T\|\boldsymbol\theta-\boldsymbol\theta_0\|.
\]
Acceptance probabilities are continuous only under the additional condition
\begin{equation}
\Pr\{\mathbf Z(T;\boldsymbol\theta)\in\partial\mathcal A\}=0,
\label{eq:nullbdry}
\end{equation}
which is assumed rather than inferred from continuity of the parameter law.
\end{proposition}

Formally, the three-dimensional analogue is regular in $\varepsilon$. For a
spherical target,
\[
\Lambda_{jm}=\frac{\kappa'_{jm}g_{jm}\ell_j^2}
{\kappa'_{jm}g_{jm}\ell_j+D_m},\qquad
\widetilde{\mathcal J}_{jm}\sim4\pi D_m\varepsilon
\frac{\bar\gamma_{jm}}{s+\bar\gamma_{jm}}\Lambda_{jm}\Gamma_m(\mathbf x_j,s).
\]
Off-diagonal Green interactions enter at $O(\varepsilon)$; for nonspherical
targets $\Lambda_{jm}$ is replaced by capacitance
\cite{smoluchowski1917versuch,collins1949diffusion,wang2025multi,grebenkov2020diffusion}.
No remainder estimate or numerical verification is claimed in three dimensions.

While the preceding propositions focus on the logarithmic expansions inherent to two-dimensional domains, the methodology naturally extends to 3D. In three dimensions, the expansion proceeds directly in $\varepsilon$. With the same reactivity scaling, the inner field $\widetilde U_{jm}=\Gamma_m(\mathbf x_j,s)(1-\Lambda_{jm}/\varrho)$ satisfies the inner Robin condition with $\Lambda_{jm}=\kappa'_{jm}g_{jm}\ell_j^2/(\kappa'_{jm}g_{jm}\ell_j+D_m)$, and the total current into $\partial\mathcal U_j$ is $4\pi D_m\varepsilon\Lambda_{jm}\Gamma_m(\mathbf x_j,s)$; multiplying by the
internalized fraction $\bar\gamma_{jm}/(s+\bar\gamma_{jm})$ gives the flux $\widetilde{\mathcal J}_{jm}\sim4\pi D_m\varepsilon\, [\bar\gamma_{jm}/(s+\bar\gamma_{jm})]\Lambda_{jm}(s)\Gamma_m(\mathbf x_j,s)$, interpolating between reaction-limited and Smoluchowski behaviour \cite{smoluchowski1917versuch,collins1949diffusion,grebenkov2020diffusion,gao2022rolling}.
The far field is a monopole of strength $4\pi D_m\varepsilon\Lambda_{jm}\Gamma_m$, and matching monopoles to the outer representation gives an $N\times N$ interaction system with diagonal terms set by $\Lambda_{jm}$ and off-diagonal terms proportional to $G_m(\mathbf x_i,\mathbf x_j;s)$, so pairwise competition enters at $O(\varepsilon)$. For non-spherical compartments $\Lambda_{jm}$ is replaced by
the shape-dependent electrostatic capacitance. Unlike two dimensions, the expansion is regular in $\varepsilon$ and no logarithmic resummation is needed. We refer the reader to Appendix~\ref{app:formal_3d} for the formal matched-asymptotic expansion for three-dimensional spherical targets.

\section{Numerical setup}
\label{sec:numerics}
Two regimes are used and never substituted for one another. The conservative regime has reflecting walls and $\gamma_m=0$; it is the regime of Assumption~\ref{ass:cons} and carries all normalized splitting-probability statements. The lossy regime has $\gamma_m>0$ (baseline $0.4$) and is used only for the annular verification, the loss-channel diagnostic, the spacing calculation and the sensitivity audits, where Remark~\ref{rem:lossy} governs.

Except for the annular benchmark the chamber is $\Omega=[0,1]^2$ with $\varepsilon=0.02$ ($\nu=0.2556$), $D=1$ and $\ell_j=1$. The Green's function is evaluated from the Neumann eigenfunction series
\begin{equation}
G(\mathbf x,\mathbf z;s)=\sum_{\mathbf n\ge\mathbf 0}
\frac{\varphi_{\mathbf n}(\mathbf x)\varphi_{\mathbf n}(\mathbf z)}
{\lambda_{\mathbf n}+s+\gamma},\quad
\varphi_{\mathbf n}=c_{n_1}c_{n_2}\cos(n_1\pi x_1)\cos(n_2\pi x_2),\quad
\lambda_{\mathbf n}=D\pi^2|\mathbf n|^2,
\label{eq:series}
\end{equation}
with $c_0=1$, $c_n=\sqrt2$; the $(0,0)$ mode is omitted in the conservative case, giving the pseudo-Green's function of Remark~\ref{rem:pseudo}. Off the diagonal this converges: at ${\bf x}=(0.5,0.5)$, ${\bf z}=(0.25,0.25)$ the change from cutoff 200 to 300 in each index is $7\times 10^{-7}$. On the diagonal it does not --- the raw sum diverges logarithmically, growing by $\approx0.11$ per doubling, which is $\log 2/2\pi$ ($3.0969$, $3.2057$, $3.3152$, $3.4251$, $3.5352$ at cutoffs $50$--$800$, evaluated at ${\bf x}={\bf z}=(0.5,0.5)$ in the lossy regime $\gamma=0.4$, where the ${\bf 0}$ mode is finite and retained; the conservative sum diverges identically once that mode is dropped) --- so the regular part is obtained by an Ewald splitting. To specify both regimes, put $a=s+\gamma>0$ and
\[
\mathcal I_a(r;t_*)=\int_0^{t_*}\frac{\exp\bigl[-at-r^2/(4Dt)\bigr]}{4\pi Dt}\dd t.
\]
For the unit square the Neumann images are $\mathbf z_{\mathbf k,\boldsymbol\epsilon}=2\mathbf k+
(\epsilon_1z_1,\epsilon_2z_2)$, $\mathbf k\in\mathbb Z^2$ and $\boldsymbol\epsilon\in\{-1,1\}^2$. The lossy Green function and its regular part are evaluated as
\begin{align}
G_a(\mathbf x,\mathbf z)&=\sum_{\mathrm{img}}
\mathcal I_a(|\mathbf x-\mathbf z_{\mathrm{img}}|;t_*)
+\sum_{\mathbf n\ge\mathbf0}\frac{\varphi_{\mathbf n}(\mathbf x)
\varphi_{\mathbf n}(\mathbf z)e^{-(\lambda_{\mathbf n}+a)t_*}}
{\lambda_{\mathbf n}+a},\nonumber\\
R_a(\mathbf z,\mathbf z)&=\frac{\log(4Dt_*)-\gamma_{\mathrm{EM}}}{4\pi D}
+\int_0^{t_*}\frac{e^{-at}-1}{4\pi Dt}\dd t
+\sum_{\mathrm{img}\ne\mathrm{self}}\mathcal I_a(
|\mathbf z-\mathbf z_{\mathrm{img}}|;t_*)\nonumber\\
&\hspace{2cm}+\sum_{\mathbf n\ge\mathbf0}
\frac{\varphi_{\mathbf n}(\mathbf z)^2e^{-(\lambda_{\mathbf n}+a)t_*}}
{\lambda_{\mathbf n}+a}.
\label{eq:ewaldloss}
\end{align}
The conservative pseudo-Green's function is the zero-mode-subtracted limit
$a\downarrow0$:
\begin{align}
G^{0}(\mathbf x,\mathbf z)&=\frac{1}{4\pi D}\sum_{\mathrm{img}}
E_1\Bigl(\tfrac{|\mathbf x-\mathbf z_{\mathrm{img}}|^2}{4Dt_*}\Bigr)
-\frac{t_*}{|\Omega|}
+\sum_{\mathbf n\neq\mathbf 0}\frac{\varphi_{\mathbf n}(\mathbf x)
\varphi_{\mathbf n}(\mathbf z)e^{-\lambda_{\mathbf n}t_*}}{\lambda_{\mathbf n}},
\nonumber\\
R^{0}(\mathbf z,\mathbf z)&=\frac{\log(4Dt_*)-\gamma_{\mathrm{EM}}}{4\pi D}
+\frac{1}{4\pi D}\!\!\sum_{\mathrm{img}\neq\mathrm{self}}\!\!
E_1\Bigl(\tfrac{|\mathbf z-\mathbf z_{\mathrm{img}}|^2}{4Dt_*}\Bigr)
-\frac{t_*}{|\Omega|}
+\sum_{\mathbf n\neq\mathbf 0}\frac{\varphi_{\mathbf n}(\mathbf z)^2
e^{-\lambda_{\mathbf n}t_*}}{\lambda_{\mathbf n}},
\label{eq:ewald}
\end{align}

All image and spectral sums now converge exponentially. Here $12$ image shells means $\|\mathbf k\|_\infty\le12$.

The splitting is not merely a convenience for the diagonal. The obvious alternative in the lossy case --- the direct modified-Bessel image sum $G_a=\sum_{\mathrm{img}}K_0(\sqrt{a/D}\,r_{\mathrm{img}})/(2\pi D)$ --- converges like $\exp(-2\sqrt{a/D}\,\|\mathbf k\|_\infty)$, so a cutoff adequate at $\gamma=O(1)$ silently fails as $\gamma\downarrow0$: a fixed $\|\mathbf k\|_\infty\le6$ reproduces $G$ to $5\times10^{-4}$ relative error at $\gamma=0.4$ but underestimates it by $64\%$ at $\gamma=0.004$, because the screening length $\sqrt{D/\gamma}$ then far exceeds the chamber. Any implementation using that representation must take $\|\mathbf k\|_\infty\gtrsim13\sqrt{D/\gamma}$. In \eqref{eq:ewaldloss} the
images enter only through $\mathcal I_a$, whose Gaussian factor $\exp[-r^2/(4Dt_*)]$ is independent of $a$, so $12$ shells suffice uniformly in $\gamma$; \eqref{eq:ewaldloss} and the converged Bessel sum agree to machine precision at every $\gamma$ in Table~\ref{tab:loss}. With $t_*=0.06$, $12$ image shells and $60$ modes per index the values are stable to ten significant figures and unchanged over $t_*\in[0.03,0.10]$; for $\mathbf x_1=(0.5,0.5)$, $\mathbf x_0=(0.25,0.25)$ they are $G^0=-0.0137897250$ and $R^0=-0.2085777932$.

\begin{table}[htbp]\centering\small
\caption{Eventual uptake $\pi_1$ against $\sigma$ as bulk loss is removed.}
\label{tab:loss}
\begin{tabular}{@{}c|rrrrr|r@{}}
\toprule
$\gamma$&$\sigma=0$&$0.30$&$0.60$&$0.90$&$0.95$&spread\\
\midrule
$0.400$&0.8236&0.8114&0.7824&0.6257&0.4938&0.330\\
$0.100$&0.9493&0.9452&0.9351&0.8703&0.7966&0.153\\
$0.020$&0.9894&0.9885&0.9863&0.9711&0.9515&0.038\\
$0.004$&0.9979&0.9977&0.9972&0.9941&0.9899&0.008\\
\bottomrule
\end{tabular}
\end{table}

To model the physical delivery, the source is configured as a constant-rate infusion over the assay window,
\begin{equation}
I_m(\mathbf x,t)=a_{\mathrm{dose}}\,\delta(\mathbf x-\mathbf x_0)
\mathbf 1_{[0,t_{\mathrm d}]}(t),
\qquad a_{\mathrm{dose}}=1.0,\qquad t_{\mathrm d}=T=6,
\label{eq:dosing}
\end{equation}
so that
\begin{equation}
\mathcal J_j(t)=a_{\mathrm{dose}}\pi_j
\bigl[H_j(t)-H_j(t-t_{\mathrm d})\bigr],
\label{eq:pulseconv}
\end{equation}
where $\pi_j=\widetilde{\mathcal J}_j(0)$ and $H_j$ is the conditional internalization-time distribution of a unit release, matched to the exact first two conditional moments of \eqref{eq:fluxred}. The splitting mass is therefore carried explicitly when $\pi_j<1$; in the conservative one-target calculation below, $\pi_1=1$. In a conservative chamber every infused molecule is eventually internalized, so for a single organoid the steady-state productive flux equals $a_{\mathrm{dose}}$ identically, independently of $\sigma$, $\kappa'$, $\ell$ and geometry: the dose rate is the plateau flux, and no dimensionless dose multiplier is required. Since $\pi_{1m}=1$ for every $\sigma$, retention acts only on the approach to that plateau, so the readout horizon is set to a small multiple of the mean first-adsorption time
$\mathcal T^{\mathrm{ads}}=0.5074$ (standard deviation $0.5014$) rather than to a long horizon, at which the terminal state is $\sigma$-independent to three decimals.

To complete the baseline numerical specification, unless a subsection explicitly overrides the transport regime, every one-target phenotype calculation uses $\kappa'=2.0$, $\ell=1$, $\bar\gamma=1$, $\gamma=0$, $\varepsilon=0.02$, $\mathbf x_1=(0.5,0.5)$, $\mathbf x_0=(0.25,0.25)$, and
\begin{gather}
(r_{\mathrm{grow}},k_{PD},b_{PD},k_{DM},b_{DM},d_0,k_{\mathrm{stress}},
k_{\mathrm{ap,clr}})=(0.35,0.08,0.60,0.05,0.45,0.10,0.02,0.15),\nonumber\\
(\alpha_Q,\alpha_R,\Theta,n_{\mathrm H})=(0.55,1.70,0.80,3),\qquad
\boldsymbol\upsilon(0)=(0.70,0.20,0.05,0.05,0.60).
\label{eq:params}
\end{gather}
Baseline audit ($\sigma=0.33$): $\upsilon_{\mathrm{mt}}(T)=0.6860$, $\upsilon_{\mathrm{ds}}(T)=0.0493$.

For the time-domain reconstruction, the transform \eqref{eq:fluxred} is not inverted exactly. Because two moments do not determine a distribution, the reconstruction proceeds through several structured steps:
\begin{enumerate}[label=(\roman*)]
\item \emph{Moments.} $\widetilde{\mathcal J}_{jm}(s)$ is available analytically.
For the conservative one-target calculation define $G_s=G(\mathbf x_1,\mathbf x_0;s)$ and $R_s=R(\mathbf x_1,\mathbf x_1;s)$. Each is expanded as $c_0+\sum_{k\ge1}(-1)^kc_ks^k$ with $c_k=\sum_{\mathbf n\neq\mathbf 0} \varphi_{\mathbf n}(\mathbf x)\varphi_{\mathbf n}(\mathbf z) \lambda_{\mathbf n}^{-(k+1)}$ for $k\ge1$ --- absolutely convergent, summed to mode cutoff $400$ --- and $c_0$ from \eqref{eq:ewald}; the series is truncated at order $10$ (radius of convergence $\lambda_{\min}=D\pi^2\approx9.87$; order $4$ already suffices to ten digits). The first two moments
\[
T^{(1)}_j=-\frac{\partial_s\widetilde{\mathcal J}_{jm}(0)}
{\widetilde{\mathcal J}_{jm}(0)},\qquad
\varsigma_j=\Bigl[\frac{\partial_s^2\widetilde{\mathcal J}_{jm}(0)}
{\widetilde{\mathcal J}_{jm}(0)}-(T^{(1)}_j)^2\Bigr]^{1/2}
\]
are taken by fourth-order central differences with $h=10^{-3}$; the values are stable to ten significant figures over $h\in[10^{-4},3\times10^{-3}]$. 
\item \emph{Family.} $H_j$ is the Gamma distribution matched to those two moments,
\begin{equation}
H_j(t)=P\!\left(k_j,\ \frac{t}{\vartheta_j}\right),\qquad
k_j=\Bigl(\frac{T^{(1)}_j}{\varsigma_j}\Bigr)^{2},\qquad
\vartheta_j=\frac{\varsigma_j^{2}}{T^{(1)}_j},
\label{eq:gammafit}
\end{equation}
with $P$ the regularized lower incomplete gamma function. For the baseline ($\sigma=0.33$) this gives $T^{(1)}=1.5466$, $\varsigma=1.1705$, hence $k=1.745888$ and $\vartheta=0.885854$. The Gamma family is chosen for support on $[0,\infty)$ and an exponential tail, matching the analytic structure of $\widetilde{\mathcal J}_{jm}$ away from the heavy-tailed regime of Remark~\ref{rem:tails}, where this reconstruction must not be used. 
\item \emph{The washout term.} $H_j$ is a distribution function supported on $[0,\infty)$, so $H_j(t-t_{\mathrm d})\equiv0$ for $t<t_{\mathrm d}$; this is imposed by clamping the argument at zero, which is exact rather than an approximation since $P(k,0)=0$. In every phenotype experiment reported in Tables~\ref{tab:null}, \ref{tab:dosing}, \ref{tab:sigma} and \ref{tab:residence}, $t_{\mathrm d}=T$, so the term is identically zero on $[0,T]$ and $\mathcal J_j(t)=a_{\mathrm{dose}}\pi_jH_j(t)$ throughout; it contributes only in the dose-then-washout diagnostic of \S\ref{subsec:retention}, where $t_{\mathrm d}=9<T=20$.
\item \emph{Tolerances.} The state equations \eqref{eq:state} are integrated by fixed-step classical Runge--Kutta with $900$ steps on $[0,T]$. This is verified against an adaptive DOP853 solution at relative tolerance $10^{-12}$ and absolute tolerance $10^{-14}$: the two agree to eight significant figures ($\upsilon_{\mathrm{mt}}(T)=0.68598423$, $\upsilon_{\mathrm{ds}}(T)=0.04932634$), and the fixed-step result is unchanged from $200$ steps upward. No quadrature enters the flux, which is closed-form
given \eqref{eq:gammafit}.
\item \emph{Uniformity across experiments.} The same reconstruction is used everywhere, with one exception. In the null model of \S\ref{sec:null} the organoids differ in their share $\pi_j$ but are given a common ramp $(T^{(1)},\varsigma)=(1.55,1.15)$ rather than a per-organoid one. We have measured what that costs. Differentiating the multi-target reduced solve in $s$ gives the conditional $(T^{(1)}_j,\varsigma_j)$ of every organoid at no extra modelling cost, and repeating the ensemble of \S\ref{subsec:nullspread} with per-organoid ramps gives median $\mathrm{CV}(\upsilon_{\mathrm{mt}})=0.609$ against $0.617$ for the common ramp, and median $\mathrm{CV}(\upsilon_{\mathrm{ds}})=0.857$ against $0.838$. The common ramp overstates the maturation CV by $1.42\%$ (median; 10th to 90th percentile $0.62\%$ to $2.52\%$) and understates the disease-score CV by
$2.25\%$. Across the whole ensemble the per-organoid $T^{(1)}$ spans $0.965$--$1.481$.

The sign is worth dwelling on, because the plausible argument runs the other
way: distant organoids have both a smaller plateau and a slower ramp, and
slower ramps lower maturation, so the two effects look as though they should
reinforce. That reasoning fails for $\upsilon_{\mathrm{mt}}$. The common ramp is slower than the typical
per-organoid ramp over most of the ensemble, which compresses the high-flux tail
where the maturation response saturates and so inflates the CV. The two signs
differ between the maturation and disease-score readouts, which is why the
direction cannot be argued from monotonicity alone. Either way the effect is
about one per cent and does not move any conclusion.
\end{enumerate}

\begin{table}[htbp]\centering\small
\caption{Transport-only spread under identical biology, $N=10$, $2000$ layouts
from \eqref{eq:ensemblelaw} at $d_{\mathrm{src}}=0.10$. Brackets are $95\%$
bootstrap intervals for the percentile itself.}
\label{tab:null}
\begin{tabular}{@{}lccc@{}}
\toprule
&10th pct&median&90th pct\\
\midrule
$\mathrm{CV}$ of $\pi_j$&$0.820$ \tiny$[0.802,0.834]$&$1.126$ \tiny$[1.113,1.142]$&$1.503$ \tiny$[1.474,1.516]$\\
$\max_j\pi_j/\min_j\pi_j$&$16.7$ \tiny$[15.5,17.9]$&$44.0$ \tiny$[42.1,45.3]$&$95.2$ \tiny$[90.2,100.1]$\\
$\mathrm{CV}$ of $\upsilon_{\mathrm{mt}}(T)$&$0.488$ \tiny$[0.478,0.497]$&$\mathbf{0.623}$ \tiny$[0.618,0.627]$&$0.739$ \tiny$[0.730,0.746]$\\
$\mathrm{CV}$ of $\upsilon_{\mathrm{ds}}(T)$&$0.681$ \tiny$[0.674,0.690]$&$\mathbf{0.835}$ \tiny$[0.829,0.843]$&$1.020$ \tiny$[1.005,1.034]$\\
\bottomrule
\end{tabular}
\end{table}

\begin{table}[htbp]\centering\small
\caption{Transport-only maturation CV by dosing geometry, after equalizing
array-mean exposure. Brackets are $95\%$ bootstrap intervals.}
\label{tab:dosing}
\begin{tabular}{@{}lccc@{}}
\toprule
&10th pct&median&90th pct\\
\midrule
point source (localized port)&$0.488$ \tiny$[0.478,0.497]$&$0.623$ \tiny$[0.618,0.627]$&$0.739$ \tiny$[0.730,0.746]$\\
uniform source (medium exchange)&$0.082$ \tiny$[0.079,0.084]$&$\mathbf{0.124}$ \tiny$[0.122,0.126]$&$0.181$ \tiny$[0.178,0.186]$\\
\bottomrule
\end{tabular}
\end{table}

\begin{table}[htbp]\centering\small
\caption{Conservative $\sigma$ sweep, continued search, $\kappa'=2.0$.}
\label{tab:sigma}
\begin{tabular}{@{}rrrrrrr@{}}
\toprule
$\sigma$&$\pi_1$&$\mathcal T^{\mathrm{tot}}_1$&sd&
$\upsilon_{\mathrm{mt}}(T)$&$\upsilon_{\mathrm{ds}}(T)$&$\upsilon_{\mathrm{ap}}(T)$\\
\midrule
0.00&1.0000&1.5074&1.1187&0.6899&0.0474&0.1038\\
0.10&1.0000&1.5162&1.1305&0.6890&0.0478&0.1042\\
0.20&1.0000&1.5273&1.1451&0.6879&0.0484&0.1046\\
0.33&1.0000&1.5466&1.1705&0.6860&0.0493&0.1053\\
0.50&1.0000&1.5870&1.2229&0.6819&0.0514&0.1069\\
0.60&1.0000&1.6268&1.2737&0.6780&0.0535&0.1083\\
0.70&1.0000&1.6931&1.3567&0.6714&0.0571&0.1108\\
0.80&1.0000&1.8257&1.5178&0.6585&0.0647&0.1155\\
0.90&1.0000&2.2236&1.9758&0.6212&0.0913&0.1291\\
0.95&1.0000&3.0194&2.8380&0.5543&0.1584&0.1534\\
\bottomrule
\end{tabular}
\end{table}

\begin{table}[htbp]\centering\small
\caption{Matched mean and matched $\sigma$; the laws differ in dispersion and
in every higher moment. $\mathcal T^{\mathrm{tot}}$ and its standard deviation
are analytic; the terminal states are from exact numerical inversion of
the reduced transform (\S\ref{subsec:gammaverify}). The surrogate would give $0.4818$, $0.4929$, $0.5354$ and $0.2468$, $0.2396$, $0.2051$; it preserves the maturation ordering but reverses the first two disease-score entries, whose exact separation is only $0.002$.}
\label{tab:residence}
\begin{tabular}{@{}lrrrr@{}}
\toprule
residence law&$\mathcal T^{\mathrm{tot}}$&sd&
$\upsilon_{\mathrm{mt}}(T)$&$\upsilon_{\mathrm{ds}}(T)$\\
\midrule
Gamma, $\mathrm{CV}_\tau=0.5$&3.7685&3.1264&0.4879&0.2295\\
exponential, $\mathrm{CV}_\tau=1.0$&3.7685&3.4132&0.4954&0.2316\\
long-tailed, $\mathrm{CV}_\tau=2.1$&3.7685&4.4916&0.5358&0.1848\\
\bottomrule
\end{tabular}
\end{table}

\emph{Size of the approximation.} Within the two-moment family, replacing the
Gamma by a lognormal or an inverse-Gaussian law with the \emph{same} two exact
moments moves $\upsilon_{\mathrm{mt}}(T)$ and $\upsilon_{\mathrm{ds}}(T)$ by at
most $0.009$ across the $\sigma$ sweep of Table~\ref{tab:sigma} and preserves
the orderings there; the terminal state is dominated by the plateau rather than
by the shape of the ramp, because $T$ is several multiples of
$\mathcal T^{\mathrm{ads}}$.

That family comparison is not the relevant test for the residence-law experiment. Measured against exact inversion rather than against another two-moment family, the surrogate reverses the first two disease-score entries of Table~\ref{tab:residence} (\S\ref{subsec:gammaverify}), which is why that table reports exact values. Accordingly, the conclusions drawn from the Gamma surrogate --- the null model of \S\ref{sec:null}, Table~\ref{tab:sigma} and \S\ref{sec:acceptance} --- do not rest on third or higher moments of the internalization time; \S\ref{subsec:retention} does, and uses exact inversion for that reason. Note that $\sigma=0.33$ with $\bar\gamma=1$ forces $\gamma^d=0.5$ and hence $\langle\tau\rangle=2/3$ by \eqref{eq:kinetic}; there is no separate baseline residence time.

As a final note on the numerical configurations, the desorption protocol forms an integral part of every specification. The verification, the loss-channel diagnostic, the null-model and spacing calculations, and the acceptance-criterion study use continued search ($\eta=1$, the Markovian boundary reduction); the residence-law comparison uses resetting ($\eta=0$); \S\ref{subsec:retention} sweeps $\eta\in[0,1]$. 

\section{Transport heterogeneity as a null model}
\label{sec:null}

\subsection{How much spread does geometry alone produce?}
\label{subsec:nullspread}

To quantify the baseline variability induced solely by spatial transport, we construct a purely geometric null model. The ensemble layout protocol is specified as follows, with every element required to strictly reproduce the reported statistics:
\begin{enumerate}[label=(\roman*)]
\item \emph{Chamber and targets.} $\Omega=[0,1]^2$ with reflecting walls,
$N=10$ targets, $\varepsilon\ell_j=0.02$, $D=1$, $\gamma=0.4$, $\kappa'=2$,
$\sigma=1/3$, $\bar\gamma=1$, so $g(0)=2/3$ and
$\Psi=D/(\kappa'g(0)\ell)=0.75$ --- identical for both dosing
geometries below.
\item \emph{Centres.} A configuration $\{\mathbf x_j\}_{j=1}^{10}$ is drawn
i.i.d.\ uniform on $[0.10,0.90]^2$, giving a wall clearance
$d_\partial=0.10$ --- like $d_{\mathrm{src}}$, a fixed chamber-scale length and
not a multiple of the radius; the whole configuration is rejected and
redrawn until
$\min_{j<k}|\mathbf x_j-\mathbf x_k|>0.12$. This samples the uniform law
conditioned on the hard-core constraint, which is not the same as sequential
dart-throwing and matters for the tails.
\item \emph{Source placement and sampling law.} The port $\mathbf x_0$ is drawn uniformly on
$[0.05,0.95]^2$ and resampled within the same centre configuration until
it clears every centre by the source-separation distance
$d_{\mathrm{src}}=0.10$, so the ensemble is
\begin{equation}
\mathcal L\bigl(\{\mathbf x_j\}\bigr)\ \times\
\mathcal L\bigl(\mathbf x_0\mid\{\mathbf x_j\},\
\min_j|\mathbf x_j-\mathbf x_0|\ge d_{\mathrm{src}}\bigr):
\label{eq:ensemblelaw}
\end{equation}
the marginal law of the centres is exactly the hard-core law and only the
conditional law of the port is modified. Rejecting the joint draw instead would
additionally reweight the centre marginal toward configurations that leave more
room for a port, which is not the intended ensemble. $d_{\mathrm{src}}$ is a fixed chamber-scale length in
\eqref{eq:sep}; it happens to equal $5\varepsilon\ell$ at $\varepsilon=0.02$,
but it is not defined as a multiple of the radius, which would vanish
with $\varepsilon$.
The condition is not cosmetic, and the obvious alternative is unsafe: drawing
the port independently of the centres places it inside an organoid in
$0.95\%$ of draws and within two radii in $5.4\%$. Those layouts are extreme by
construction, so they populate the upper tail of the CV distribution, which is what exactly
the region the 90th percentile reports. Because the condition rejects a
substantial fraction of port proposals, the sensitivity to $d_{\mathrm{src}}$ is
reported in Table~\ref{tab:nullsweep}.
\item \emph{Randomness.} One \texttt{numpy} PCG64 stream, seed $2024$, consumed
in the order (centres, centre rejections, port, port rejections); $2000$ layouts
drawn and all $2000$ retained, since resampling the port never fails.
\item \emph{Discards.} With the source-clearance condition in force, no layout returns a non-positive reduced flux, at any $d_{\mathrm{src}}$ tested. Without it, $9$ of $2000$ did --- a symptom of the source-separation failure rather than of the $O(\nu)$ closure, which is worth stating because the two are
easy to confuse.
\item \emph{Uncertainty.} Percentiles carry $95\%$ percentile-bootstrap
intervals, $4000$ resamples, seed $7$.
\end{enumerate}

\begin{table}[htbp]\centering\small
\caption{Design robustness of the transport-only null, $120$ layouts per cell,
all with the source-clearance condition in force. Each row changes one factor
from the baseline ($N=10$, $\varepsilon=0.02$, $\gamma=0.4$, unit square, random
port, wall clearance $d_\partial=0.10$, hard-core radius $d_{\min}=0.12$,
source clearance $d_{\mathrm{src}}=0.10$). Entries are median/90th percentile,
with $95\%$ percentile-bootstrap intervals ($2000$ resamples, seed $7$) on the
phenotype-CV columns; ``ratio'' is the median point-source CV divided by the
median uniform-source CV. No cell produced a non-positive reduced flux. For the
fixed-port rows every
port must clear every organoid, which is imposed on the centres because the
ports cannot move; for the random-port rows the port is resampled within each
centre configuration, as in \eqref{eq:ensemblelaw}.}
\label{tab:nullsweep}
\begin{tabular}{@{}l|rr|rr|rr|r@{}}
\toprule
&\multicolumn{2}{c|}{$\mathrm{CV}(\pi_j)$}
&\multicolumn{2}{c|}{$\mathrm{CV}(\upsilon_{\mathrm{mt}})$}
&\multicolumn{2}{c|}{$\max\pi/\min\pi$}&ratio\\
&med&90th&med&90th&med&90th&\\
\midrule
baseline&1.08&1.49&0.608 \tiny$[.584,.623]$&0.703 \tiny$[.687,.740]$&43&82&$5.0$\\
\midrule
$N=5$&0.65&0.92&0.396 \tiny$[.351,.422]$&0.507 \tiny$[.497,.535]$&6.0&12&$3.9$\\
$N=20$ ($d_{\min}\,0.085$)&1.74&2.14&0.859 \tiny$[.844,.876]$&0.938 \tiny$[.919,.951]$&839&1947&$5.2$\\
\midrule
$\varepsilon=0.01$&0.97&1.33&0.540 \tiny$[.522,.557]$&0.631 \tiny$[.613,.655]$&23&43&$5.1$\\
$\varepsilon=0.04$&1.26&1.70&0.703 \tiny$[.691,.714]$&0.819 \tiny$[.788,.869]$&119&259&$4.7$\\
\midrule
$\gamma=0$ (conservative)&1.07&1.48&0.605 \tiny$[.579,.620]$&0.698 \tiny$[.682,.735]$&42&79&$4.9$\\
$\gamma=0.1$&1.08&1.48&0.606 \tiny$[.580,.621]$&0.700 \tiny$[.684,.736]$&42&79&$5.0$\\
\midrule
port at corner&1.45&1.85&0.691 \tiny$[.680,.724]$&0.815 \tiny$[.799,.829]$&83&163&$5.5$\\
port at centre&0.76&0.94&0.402 \tiny$[.385,.427]$&0.509 \tiny$[.480,.532]$&10&19&$3.3$\\
two ports&0.51&0.61&0.314 \tiny$[.304,.324]$&0.369 \tiny$[.358,.374]$&5.2&7.4&$2.6$\\
four ports&0.40&0.49&0.265 \tiny$[.254,.276]$&0.319 \tiny$[.304,.326]$&3.8&4.8&$2.1$\\
\midrule
chamber $2\times0.5$&1.24&1.86&0.754 \tiny$[.722,.793]$&0.954 \tiny$[.916,.980]$&361&2613&$5.1$\\
$d_{\min}=0.06$&1.08&1.50&0.590 \tiny$[.571,.605]$&0.716 \tiny$[.681,.739]$&32&76&$3.7$\\
$d_{\min}=0.20$&1.19&1.58&0.662 \tiny$[.634,.676]$&0.770 \tiny$[.735,.784]$&49&132&$8.4$\\
\midrule
wall clearance $d_\partial=0.05$&1.12&1.42&0.607 \tiny$[.571,.632]$&0.726 \tiny$[.705,.755]$&42&89&$4.5$\\
wall clearance $d_\partial=0.20$&1.15&1.42&0.631 \tiny$[.621,.647]$&0.702 \tiny$[.689,.725]$&39&65&$4.2$\\
$d_{\mathrm{src}}=0.05$&1.26&1.77&0.645 \tiny$[.624,.665]$&0.762 \tiny$[.736,.774]$&55&133&$5.6$\\
$d_{\mathrm{src}}=0.20$&1.08&1.36&0.605 \tiny$[.591,.622]$&0.689 \tiny$[.673,.716]$&38&72&$5.1$\\
\bottomrule
\end{tabular}
\end{table}

By assigning every organoid identical biology and kinetics, solving the reduced amplitude system in \eqref{eq:amp}, and propagating the resulting plateau fluxes through the state equations in \eqref{eq:state} using a single common parameter vector, we isolate the effects of spatial layout.

Under this transport-only null model, chamber geometry alone produces a median maturation coefficient of variation of $62\%$ and a disease-score coefficient of variation of $84\%$. For an assay whose layouts can defensibly be treated as draws from this particular ensemble, the maturation 90th percentile of $74\%$ (bootstrap interval $[0.730,0.746]$) serves as an illustrative upper reference. It is not a universal decision threshold; an assay with precisely measured coordinates should instead use the layout-specific calculation detailed in \S\ref{subsec:recipe}. Crucially, failure to exceed either reference metric means only that the data are indistinguishable from the specified transport-only null. It does not confirm biological homogeneity, as altering the chamber dimensions, port distribution law, or target count will naturally shift this reference distribution.

\subsection{How much of this survives a change of design?}
\label{subsec:nullrobust}
To determine whether the pronounced spatial variability observed in the baseline model is a general feature or an artifact of specific parameter choices, we perform a systematic sensitivity analysis. While Tables~\ref{tab:null}--\ref{tab:dosing} describe one specific ensemble, varying each design choice in turn about that baseline—at 120 layouts per test condition—yields the results in Table~\ref{tab:nullsweep}. (Note that the slight shift in the baseline median, 0.608 here compared to 0.623 in Table~\ref{tab:null}, is simply sampling noise from a smaller sample size of $n=120$ versus $n=2000$, and falls comfortably within the bootstrap interval.)

Across all tested variations, the central qualitative claim—that transport alone generates a large between-organoid spread—remains robust. The median $\mathrm{CV}(\upsilon_{\mathrm{mt}})$ never falls below 0.27 and can reach as high as 0.86. This variability is highly sensitive to certain geometric factors: it scales significantly with target count (0.40 at $N=5$, rising to 0.86 at $N=20$), target size (0.54 at $\varepsilon=0.01$, 0.70 at $\varepsilon=0.04$), and chamber aspect ratio (0.75 for a $2\times0.5$ domain). Most critically, it depends on the dosing location: the median CV is 0.69 from a corner port, 0.40 from the centre, and drops to 0.27 when distributed across four ports.

Conversely, two factors have negligible effects. Bulk loss yields essentially identical median CVs (0.605, 0.606, 0.608 at $\gamma=0$, 0.1, and 0.4, respectively), because the plateau splitting is a ratio and clearance scales all targets nearly equally. Similarly, the source-clearance distance itself introduces only a 6\% spread (0.645, 0.608, 0.605 at $d_{\mathrm{src}}=0.05$, 0.10, and 0.20)—confirming that the reported figures do not hinge on exactly where this ensemble-fixing threshold is set.

The dosing benefit likewise has a range, not a value:
the point-to-uniform improvement runs from $2.1\times$ to $8.4\times$
across these designs, and it is largest exactly where the localized geometry is
worst (a corner port, widely spaced targets) and smallest where the localized
geometry is already good (four ports, closely spaced targets). 
Consequently, while a "fivefold improvement" roughly summarizes the baseline case, it poorly captures the full dynamic range. What remains truly design-independent, however, is the positive sign and the consistent order of magnitude of this spatial benefit.

Beyond these parametric sensitivities, two practical design insights emerge. First, from an experimental standpoint, multi-port dosing recovers most of the uniformity benefits of true medium exchange. Using just four ports lowers the CV to 0.27, compared to 0.12 for a perfectly uniform source, demonstrating that design recommendations to distribute the input do not demand perfect mixing. Second, regarding the mathematical consistency of the reduced model, maintaining the source-clearance condition guarantees stability. With this condition enforced, no configuration produces a non-positive reduced flux, even at the larger target size of $\varepsilon=0.04$, where the earlier unconditioned sweep had to discard 2.5\% of the layouts.

\subsection{Uniform dosing removes forcing heterogeneity exactly in the reduced model}
\label{subsec:dosing_design}

Integrating \eqref{eq:bulk} against the reflecting boundary condition gives
$\int_\Omega G_m(\mathbf x_j,\mathbf z;s)\dd\mathbf z=(s+\gamma_m)^{-1}$
independently of $j$. A spatially uniform source therefore produces an
identical transformed forcing $\Gamma_{jm}$ at every target, so within
the reduced description \eqref{eq:amp} all residual spread comes from the
Green's matrix \eqref{eq:gmat} alone. The statement is exact for the reduced
system.

To meaningfully compare this theoretically uniform forcing against a localized source, we must carefully standardize the inputs. Specifically, we distinguish between equal injected dose and equal array-mean exposure, because a coefficient of variation (CV) comparison after a nonlinear state map is not invariant to the latter.
The localized port is a unit-rate point
injection $I=\delta(\mathbf x-\mathbf x_0)$; the uniform source is
$I\equiv I_0=1$ on $|\Omega|=1$, i.e.\ the same total injection rate.
Because $\int_\Omega G\dd\mathbf z=1/\gamma$ independently of position, the two modalities deliver nearly the same total productive exposure to the array in this ensemble. Over the 2000 layouts, the realized ratio $\sum_jJ^{\mathrm{unif}}_j/\sum_jJ^{\mathrm{pt}}_j$ has a mean of 1.0029 and a range of [0.980, 1.041]. In addition, before propagation through the state equations in \eqref{eq:state}, both flux vectors are analytically rescaled to a unit mean across targets. As a result, the phenotype CVs are compared at exactly matched array-mean exposures, differing only in the shape of the flux distribution. Both cases retain the same effective parameter $\Psi$.

Under this rigorously matched comparison, the physical distinction becomes clear: localized dosing makes geometric position the dominant determinant of exposure. In contrast, distributed dosing completely removes the forcing asymmetry and leaves only inter-target competition, which is strictly bounded and independent of any hypothetical port location. Figure~\ref{fig:null} illustrates one representative layout and the full ensemble comparison.

\begin{figure}[htbp]
\centering
\includegraphics[width=\linewidth]{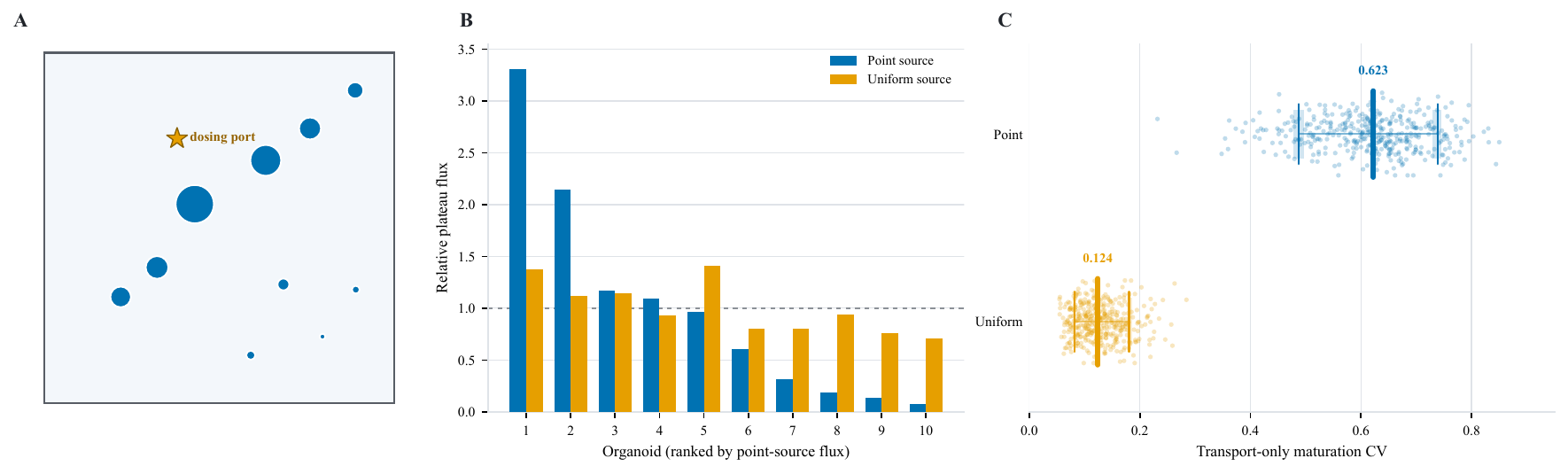}
\caption{Transport-only null model. (A) Representative realization of the
corrected random-layout protocol with ten identical organoids and one localized
dosing port; marker area is proportional to the productive plateau splitting
mass $\pi_j$. The displayed layout has
$\max_j\pi_j/\min_j\pi_j=43.95$, compared with the ensemble median $43.97$.
(B) Relative plateau flux under localized and spatially uniform dosing for the
same layout. Each flux vector is divided by its own array mean, so this is an
equal-mean-exposure analysis normalization rather than a claim of equal physical
injection rates. (C) Transport-only maturation coefficient of variation over
2000 random arrays with identical biology. Thick bars denote medians, thin bars
the 10th and 90th percentiles, and translucent bands the corresponding 95%
percentile-bootstrap intervals (4000 resamples; seed 7). Points show a fixed
random subsample of 400 layouts for legibility.}
\label{fig:null}
\end{figure}

\subsection{The recipe}
\label{subsec:recipe}

To apply this theoretical framework in practice, experimentalists can construct a layout-specific null model directly from empirical array coordinates. Given the measured target centers $\{\mathbf x_j\}$ and a known dosing geometry, the protocol proceeds as follows. First, evaluate the amplitude system in \eqref{eq:amp} to obtain the transport-only plateau flux vector $\bm\pi^{\mathrm{tr}}$. Next, compute the transport-only summary metrics: \begin{equation} \mathrm{CV}_{\mathrm{tr}}=\frac{\bigl[N^{-1}\sum_j(\pi^{\mathrm{tr}}_j -\overline\pi^{\mathrm{tr}})^2\bigr]^{1/2}}{\overline\pi^{\mathrm{tr}}}, \qquad R_{\mathrm{tr}}=\frac{\max_j\pi^{\mathrm{tr}}_j}{\min_j\pi^{\mathrm{tr}}_j}. \label{eq:nullmetrics} \end{equation} Then, propagate these fluxes through the state equations in \eqref{eq:state} using a single common parameter vector to generate the expected transport-only phenotype distribution. Finally, compare the empirically measured phenotype CV against this transport-only benchmark.

Under this formulation, the null hypothesis posits that all observed between-organoid variation arises solely from the measured geometry and transport dynamics represented by the model. Rejection of this null hypothesis is the primary informative outcome: a measured CV exceeding the model prediction strictly identifies biological or kinetic variation that spatial transport alone cannot account for. Conversely, a failure to reject the null hypothesis does not constitute proof of biological homogeneity. Rather, it merely indicates that the observation is statistically indistinguishable from transport-induced noise, which could equally reflect genuine biological variability obscured by compensating model approximations.

When performing this formal comparison, experimental uncertainties—including coordinate measurements, port placement, and assay readout noise—must be rigorously propagated through the model prior to evaluation. Under no circumstances should observed and predicted CVs be naively subtracted, as variances mapped through non-linear state dynamics do not combine linearly.

\section{What the null model needs to be right}
\label{sec:support}
\subsection{Verification of the spatial reduction}
\label{subsec:verify}

To validate the local reactive asymptotics, we first consider an exactly solvable configuration.
Let the domain be the annulus $r_\varepsilon<r<R$ with $r_\varepsilon=\varepsilon\ell$,
a unit release on the circle $r=r_0$, and $\xi(s)=\sqrt{(s+\gamma)/D}$. 
Under these conditions, the transformed field is defined piecewise as
$A_-I_0(\xi r)+B_-K_0(\xi r)$ below $r_0$ and
$A_+I_0(\xi r)+B_+K_0(\xi r)$ above it. Put
$\kappa_\varepsilon=\kappa'/\varepsilon$, $g=(s+\bar\gamma)/\chi$ and
$\chi=s+\gamma^d+\bar\gamma$. The four coefficients are the solution of
\begin{equation*}
\resizebox{0.98\linewidth}{!}{$\displaystyle
\begin{pmatrix}
D\xi I_1(\xi r_\varepsilon)-\kappa_\varepsilon gI_0(\xi r_\varepsilon)&
-D\xi K_1(\xi r_\varepsilon)-\kappa_\varepsilon gK_0(\xi r_\varepsilon)&0&0\\
0&0&I_1(\xi R)&-K_1(\xi R)\\
I_0(\xi r_0)&K_0(\xi r_0)&-I_0(\xi r_0)&-K_0(\xi r_0)\\
D\xi I_1(\xi r_0)&-D\xi K_1(\xi r_0)&-D\xi I_1(\xi r_0)&D\xi K_1(\xi r_0)
\end{pmatrix}
\begin{pmatrix}A_-\\B_-\\A_+\\B_+\end{pmatrix}
=\begin{pmatrix}0\\0\\0\\(2\pi r_0)^{-1}\end{pmatrix}.
$}
\end{equation*}
Consequently,
$\widetilde{\mathcal J}^{\mathrm{full}}(s)=2\pi r_\varepsilon
\kappa_\varepsilon\bar\gamma\chi(s)^{-1}
[A_-I_0(\xi r_\varepsilon)+B_-K_0(\xi r_\varepsilon)]$.
The reduced counterpart uses
$\Gamma(0,s)=[K_0(\xi r_0)+\beta_R(s)I_0(\xi r_0)]/(2\pi D)$ and
$R(0,0;s)=[-\log(\xi/2)-\gamma_{\mathrm{EM}}+\beta_R(s)]/(2\pi D)$ with
$\beta_R=K_1(\xi R)/I_1(\xi R)$.

Our analysis shows that the error falls monotonically, remaining below $0.27\%$ at $\varepsilon=0.1$.
Halving from $\varepsilon=0.2$ gives successive error ratios $3.95$, $3.95$, $3.96$, $3.96$, $3.97$, i.e.\ a measured convergence order of $1.985$, consistent with \eqref{eq:eps2}. 
While concentricity makes this benchmark analytically solvable, it is not responsible for the improved flux order; rather, any circular target intrinsically possesses a zero-net-flux inner dipole.

Having established this single-target baseline to validate the local reactive asymptotics, we next verify the approach against the full multi-target problem by directly solving the boundary-value problem in the perforated square. To achieve this, we solve the full boundary-value problem in the perforated square directly. Specifically, in $\Omega_\varepsilon=[0,1]^2\setminus\bigcup_j B(\mathbf x_j,\varepsilon\ell_j)$ we solve, at fixed $s$, $D\Delta\widetilde u-(s+\gamma)\widetilde u=-\delta(\mathbf x-\mathbf x_0)$ with $D\nabla\widetilde u\cdot\mathbf n=0$ on the outer square and the dynamic condition \eqref{eq:grobin} on every disk. Writing $\widetilde u=w+v$ with $w$ the free-space fundamental solution removes the point source exactly—since the two $\delta$'s cancel, $v$ solves the homogeneous equation. This leaves a smooth Robin/Neumann problem for $v$ that is discretized with $P_2$ elements on a graded unstructured mesh. Three mesh levels $h_{\min}=\varepsilon\ell/12, /24$, and $/48$ are run at every configuration. The observed order in $h$ is 2 (set by the polygonal representation of the circles), and the reported full-PDE flux is the Richardson extrapolant. The worst extrapolation increment over all 54 configurations is $1.3\times10^{-5}$ relative, ensuring that every model error exceeding $10^{-4}$ is resolved by more than an order of magnitude. 

With the above highly resolved numerical framework in place, we establish an experimental design where targets are drawn from one fixed ten-point layout, utilizing $N=2$ and $N=5$ nested subsets alongside a fixed off-centre port at (0.62, 0.30), with $\kappa'=1.5$, $\gamma^d=0.5$, $\bar\gamma=1$, and $\ell=1$. Both the conservative ($\gamma=0$) and lossy ($\gamma=0.4$) regimes are run across transform variables $s=0, 1, 5$. Because the minimum centre separation of the ten-point layout is 0.22, setting $\varepsilon=0.10$ would put the $N=10$ disks two radii apart, violating Assumption~\ref{ass:dilute} by construction. Therefore, the $N=10$ sweep runs at $\varepsilon=0.05, 0.025$, and 0.0125, with each row recording $d_{\min}/(\varepsilon\ell)$. Using the relative error definitions$$E_\infty=\frac{\max_j\vert{}\mathcal J^{\mathrm{full}}_j-\mathcal J^{\mathrm{red}}_j\vert{}}{\max_j\vert{}\mathcal J^{\mathrm{full}}_j\vert{}},\qquad E_1=\frac{\sum_j\vert{}\mathcal J^{\mathrm{full}}_j-\mathcal J^{\mathrm{red}}_j\vert{}}{\sum_j\vert{}\mathcal J^{\mathrm{full}}_j\vert{}},\label{eq:pdeerr}$$
Table~\ref{tab:multiverify} reports the extremes over $s$ at each $(N,\varepsilon)$, establishing several key insights regarding the system's convergence and accuracy:
\begin{enumerate}[label=(\roman*)]
\item \emph{The multi-target closure converges at the predicted rate.} The
fitted slope is $2.00$--$2.15$ at every $N$ and both regimes, confirming
\eqref{eq:eps2} for the interacting problem and not only for one isolated
target. At $s=0$ convergence is faster still ($p\approx3$--$4$), so the
splitting probabilities are more accurate than the transforms that generate them.
\item \emph{The reported quantities are far more accurate than the individual fluxes.} At $\varepsilon=0.02$ in this widely separated fixed layout ($d_{\min}/\varepsilon\ell=20.2$), the flux error is approximately $5\times10^{-4}$, the target-allocation error is roughly $10^{-4}$, and the relative error in the transport-only phenotype CV is $5.6\times10^{-6}$.
\item \emph{Separation matters more than target count, but not uniformly across
functionals.} Order the rows by $d_{\min}/\varepsilon\ell$. 
At $d_{\min}/\varepsilon\ell\approx4$ the flux error is $0.5$--$2\%$ whether $N$ is
$5$ or $10$ --- though they do not collapse exactly, and the $N=2$ block does not
join them. 
\end{enumerate}

\begin{table}[htbp]\centering\small
\caption{Full PDE versus the reduced system in the perforated square. Every
entry is dimensionless; the transform variable and the normalization differ by
column and are stated here rather than left to inference.
$E_\infty$ and $E_1$ are the relative errors \eqref{eq:pdeerr},
maximized over the three transform points $s\in\{0,1,5\}$.
``alloc.'' is the absolute error
$\max_j|\widehat\pi^{\mathrm{full}}_j-\widehat\pi^{\mathrm{red}}_j|$ in the
target-allocation vector $\widehat{\bm\pi}=\bm{\mathcal J}/\sum_k\mathcal J_k$,
also maximized over the three $s$; at $s>0$ this is an allocation of transformed
flux rather than of probability, and only the $s=0$ entry is a splitting
probability.
``$\mathrm{CV}(\pi)$ err'' is the relative error in
$\mathrm{CV}_j(\mathcal J_j)$, maximized over the three $s$.
``$\mathrm{CV}(\upsilon_{\mathrm{mt}})$ err'' is the relative error in the
transport-only phenotype CV and uses $s=0$ only, since the plateau flux
is what \eqref{eq:state} is driven by; it is the worse of the conservative and
lossy regimes.
``$p$'' is the fitted slope of $E_\infty$ against $\varepsilon$, over entries
exceeding ten times the Richardson mesh-error estimate.}
\label{tab:multiverify}
\begin{tabular}{@{}rrr|rr|rrr|r@{}}
\toprule
$N$&$\varepsilon$&$d_{\min}/\varepsilon\ell$&$E_\infty$&$E_1$&
alloc.&$\mathrm{CV}(\pi)$ err&$\mathrm{CV}(\upsilon_{\mathrm{mt}})$ err&$p$\\
\midrule
$2$&$0.10$&$9.5$&$1.4\e{-2}$&$1.3\e{-2}$&$2.9\e{-4}$&$5.5\e{-3}$&$6.5\e{-4}$&$2.03$\\
$2$&$0.05$&$19.0$&$3.4\e{-3}$&$3.3\e{-3}$&$6.2\e{-5}$&$1.3\e{-3}$&$6.5\e{-5}$&\\
$2$&$0.02$&$47.4$&$5.2\e{-4}$&$5.1\e{-4}$&$8.2\e{-6}$&$1.8\e{-4}$&$5.0\e{-6}$&\\
\midrule
$5$&$0.10$&$4.0$&$1.5\e{-2}$&$1.4\e{-2}$&$5.4\e{-3}$&$3.6\e{-3}$&$7.2\e{-4}$&$2.15$\\
$5$&$0.05$&$8.1$&$3.2\e{-3}$&$3.2\e{-3}$&$1.1\e{-3}$&$1.2\e{-3}$&$2.8\e{-5}$&\\
$5$&$0.02$&$20.2$&$4.6\e{-4}$&$4.6\e{-4}$&$1.5\e{-4}$&$2.3\e{-4}$&$5.6\e{-6}$&\\
\midrule
$10$&$0.05$&$4.4$&$5.6\e{-3}$&$9.3\e{-3}$&$2.0\e{-3}$&$9.7\e{-3}$&$3.4\e{-4}$&$2.00$\\
$10$&$0.025$&$8.8$&$1.4\e{-3}$&$2.3\e{-3}$&$4.6\e{-4}$&$2.7\e{-3}$&$2.6\e{-5}$&\\
$10$&$0.0125$&$17.6$&$3.5\e{-4}$&$5.6\e{-4}$&$1.0\e{-4}$&$7.3\e{-4}$&$3.9\e{-6}$&\\
\bottomrule
\end{tabular}
\end{table}

We note that Assumption~\ref{ass:dilute} is best read as a condition on separation
in units of the radius rather than on target count. Consequently, the benefit gained from a given separation depends heavily on the specific functional desired, as the three do not scale identically:

\begin{itemize}[itemsep=2pt,topsep=3pt]
\item \emph{Phenotype CV} drops below $10^{-3}$ relative already at about four radii ($7.2\times10^{-4}$ at $N=5$, $3.4\times10^{-4}$ at $N=10$).
\item \emph{Target allocation} converges more slowly: at four radii it is $5.4\times10^{-3}$ ($N=5$) and $2.0\times10^{-3}$ ($N=10$), only approaching $10^{-3}$ near eight to ten radii ($1.1\times10^{-3}$ at $N=5$, $\varepsilon=0.05$; $4.6\times10^{-4}$ at $N=10$, $\varepsilon=0.025$).
\item \emph{Individual fluxes} are the most configuration-dependent. Sub-$10^{-3}$ flux error occurs in the most refined $N=5$ and $N=10$ tests, but not in the $N=2$ test at nineteen radii, which still exhibits an error of $3.4\times10^{-3}$.
\end{itemize}

Ultimately, the null-model ensemble of \S\ref{sec:null} maintains $d_{\min}/\varepsilon\ell>6$ and $d_{\mathrm{src}}/\varepsilon\ell\ge5$ throughout, and the design sweep of \S\ref{subsec:nullrobust} reaches about 4 in its coarsest cell. Both lie in the range where the reported CVs are accurate to well under 1\%, yet neither lies in the range where individual fluxes or allocations are strictly accurate to $10^{-3}$. This provides functional-specific empirical guidance derived directly from the tested configurations.

\subsection{Validating the time-domain reconstruction}
\label{subsec:gammaverify}

Almost every phenotype number rests on the two-moment Gamma surrogate of
\eqref{eq:gammafit}, and two moments do not determine a law. We therefore
compare it against numerical inversion of the exact reduced transform.
To perform this inversion, the square's Green's function is continued to complex $s$ through its
eigenfunction series relative to a real base point. In the lossy case $a_0=\gamma>0$ and
\begin{equation}
G(\mathbf x,\mathbf z;s)=G_{a_0}(\mathbf x,\mathbf z)
-\sum_{\mathbf n\ge\mathbf 0}\varphi_{\mathbf n}(\mathbf x)
\varphi_{\mathbf n}(\mathbf z)
\Bigl[\tfrac1{\lambda_{\mathbf n}+a_0}-\tfrac1{\lambda_{\mathbf n}+a_0+s}\Bigr].
\label{eq:contlossy}
\end{equation}
In the conservative case $a_0=0$ and the $\mathbf n=\mathbf0$ term of
\eqref{eq:contlossy} is undefined, so the Neumann zero mode is separated first:
\[
G(\mathbf x,\mathbf z;s)=\frac1{|\Omega|s}+G^0(\mathbf x,\mathbf z)
-\sum_{\mathbf n\neq\mathbf 0}\varphi_{\mathbf n}(\mathbf x)
\varphi_{\mathbf n}(\mathbf z)
\Bigl[\tfrac1{\lambda_{\mathbf n}}-\tfrac1{\lambda_{\mathbf n}+s}\Bigr],
\]
with $G^0$ from \eqref{eq:ewald}; the $1/(|\Omega|s)$ term is carried through
the amplitude system and cancels there exactly as in Remark~\ref{rem:pseudo}.
Both corrections converge like $\lambda_{\mathbf n}^{-2}$, and no complex Bessel
evaluation is needed. Exploiting this rapid $\lambda_{\mathbf n}^{-2}$ convergence, the inversion process efficiently employs the Abate--Whitt Euler algorithm, where every node lies on a vertical line $\Re s=A/2t>0$. This avoids the pole-management that a deformed contour would
require given the poles at $s=-\lambda_{\mathbf n}$. On the analytic test
transforms $1/[s(s+1)]$ and $(1+\vartheta s)^{-k}/s$, the inversion is accurate
to $5\times10^{-9}$ over $t\in[0.3,6]$. To quantify the discrepancy, we report
$E_{\mathrm{CDF}}=\sup_t|H^{\mathrm{ref}}_j-H^\Gamma_j|$ and
$E_{\mathcal J}=\int_0^T|\mathcal J^{\mathrm{ref}}_j-\mathcal J^\Gamma_j|
/\int_0^T\mathcal J^{\mathrm{ref}}_j$.

Applying these error metrics to all calculations utilizing the Markovian boundary --- Tables~\ref{tab:null}, \ref{tab:dosing} and \ref{tab:sigma}, the null model, the dosing-geometry result and \S\ref{sec:acceptance} --- reveals that the Gamma reconstruction introduces negligible discrepancy: it changes the principal between-organoid coefficient of variation by less than $0.06\%$, and the terminal maturation by less than $0.4\%$. The reconstruction does displace the internalization-time quantiles: the exact $10$th percentile is
$0.39$ against $0.36$ for the surrogate at the baseline, and up to
$0.54$ against $0.18$ for the long-tailed reset case.

While the Gamma surrogate performs exceptionally well under Markovian conditions, its accuracy degrades by an order of magnitude when applied to the reset protocol with non-Markovian residence. One critical consequence must be recorded plainly, i.e., the reconstruction reverses one of the orderings it is used to
establish. Exact inversion gives
$\upsilon_{\mathrm{ds}}(T)=0.2295$, $0.2316$, $0.1848$ for the Gamma,
exponential and long-tailed laws, whereas the surrogate gives $0.2468$,
$0.2396$, $0.2051$: the first two rows swap. Their exact separation is $0.002$,
so this is a near-tie either way. The maturation ordering is
unchanged --- exact $0.4879$, $0.4954$, $0.5358$ against surrogate $0.4818$,
$0.4929$, $0.5354$, increasing in both, though the values themselves shift by up
to $1.3\%$ --- and it is on that ordering that
\S\ref{subsec:retention} relies. Table~\ref{tab:residence} and
Figure~\ref{fig:residence} now report the exact-inversion values throughout.

\begin{figure}[htbp]\centering
\includegraphics[width=\linewidth]{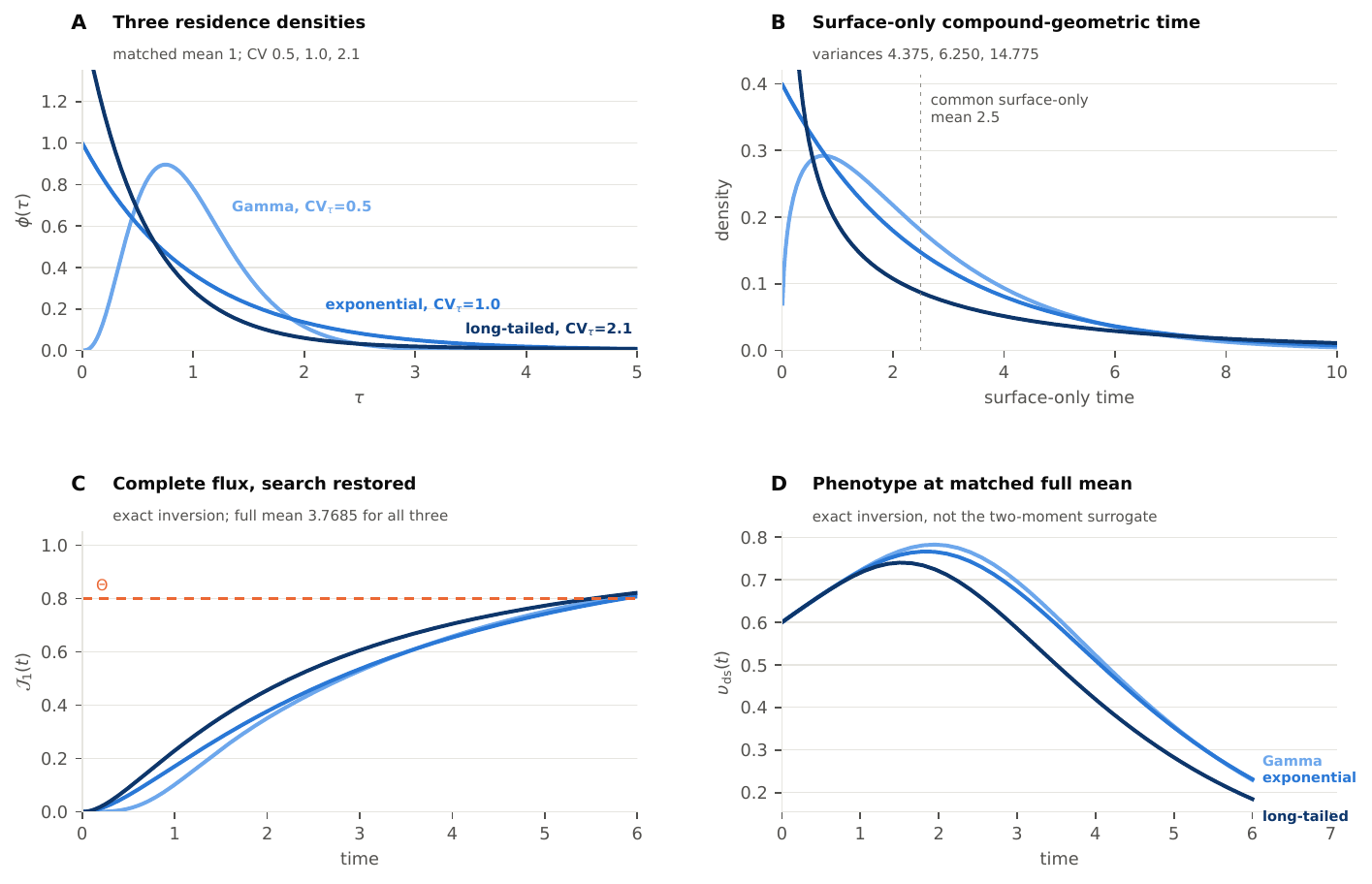}
\caption{(A) The three residence densities. (B) Surface-only compound-geometric
time, common mean $2.5$. (C) Complete flux with the diffusive search restored.
(D) Phenotype trajectories at matched full mean.}
\label{fig:residence}\end{figure}

\subsection{Phenotypic response: retention, competition, and residence laws}
\label{subsec:retention}

\medskip
\noindent\textbf{Retention and timing.} Repeating the boundary comparison while decreasing bulk loss, with everything
else fixed ($\mathbf x_1=(0.5,0.5)$, $\mathbf x_0=(0.25,0.25)$, $\kappa'=1.5$,
$\varepsilon=0.02$), we find that the spread collapses as $\gamma\to0$, and at $\gamma=0$ it is exactly
zero by \eqref{eq:onetarget}. To be explicit, with reflecting walls and no loss a desorbed
molecule cannot escape, and the probability of failing $n$ successive
internalization attempts is $\sigma_1^n\to0$. The apparent dose-response effect
of retention is therefore a $\sigma\times\gamma$ interaction --- desorption
creates additional search cycles during which clearance can act --- and not an
effect of partial accessibility alone. Define the static-Robin overestimation
factor as
$\max_{\sigma\in[0,0.95]}\pi_1^{\mathrm{stat}}(\sigma)/
\pi_1^{\mathrm{rev}}(\sigma)$, where the static model sets $g(s)\equiv1$ while
retaining the same $\kappa'$ and bulk parameters, so that
$\pi_1^{\mathrm{stat}}=\pi_1^{\mathrm{rev}}(0)$ is the $\sigma=0$ column. It is
$1.668$ at $\gamma=0.4$, $1.192$ at $\gamma=0.1$, $1.040$ at $\gamma=0.02$,
$1.008$ at $\gamma=0.004$, and exactly one at $\gamma=0$. Restricted to
$\sigma\le0.90$ the same factors are $1.316$, $1.091$, $1.019$ and $1.004$: the
overestimation is concentrated in the last few per cent of retention, where the
number of failed cycles diverges.
Static sinks may still overstate finite-horizon uptake, but
that is a statement about $U_1(T)=\int_0^T\mathcal J_1$, not about $\pi_1$.

Enforcing Assumption~\ref{ass:cons} and varying
$\gamma^d=\sigma\bar\gamma/(1-\sigma)$ at fixed $\bar\gamma=1$ (so $\langle\tau\rangle=1-\sigma$ by \eqref{eq:kinetic}), the reset specialization of \eqref{eq:renewalpi} is $\mathcal T^{\mathrm{tot}}_{1}=\mathcal T^{\mathrm{ads}}/(1-\sigma)+1$, strictly increasing. Table~\ref{tab:sigma} gives the continued-search values that the Markovian reduction actually produces, together with the finite-horizon phenotype. Since $\int_0^\infty\mathcal J_1=1$ and $\mathcal J_1>0$ beyond $T$,
the delivered fraction $U_1(T)$ is strictly less than one and decreases as
delivery is delayed; the finite-horizon phenotype can therefore deteriorate with
$\sigma$ even though eventual internalization is certain.
Consistent with this certainty, the computed eventual uptake remains $\pi_1=1$ to twelve digits throughout our tests. However, the sign of the residual terminal effect is not universal: with dosing maintained to the readout it is monotone at every
$a_{\mathrm{dose}}$ we tested, but under dose-then-washout with a plateau well
above $\Theta$ it reverses (at $a_{\mathrm{dose}}=2$, $t_{\mathrm d}=9$,
$T=20$ we obtain
$\upsilon_{\mathrm{ds}}(T)=0.108$ at $\sigma=0.33$ against $0.070$ at
$\sigma=0.90$), because a saturating response rewards a longer supra-threshold
window and high retention supplies one.
The loss-channel collapse, conservative timing and finite-horizon response are
collected in Figure~\ref{fig:sigma}.

\begin{figure}[htbp]\centering
\includegraphics[width=\linewidth]{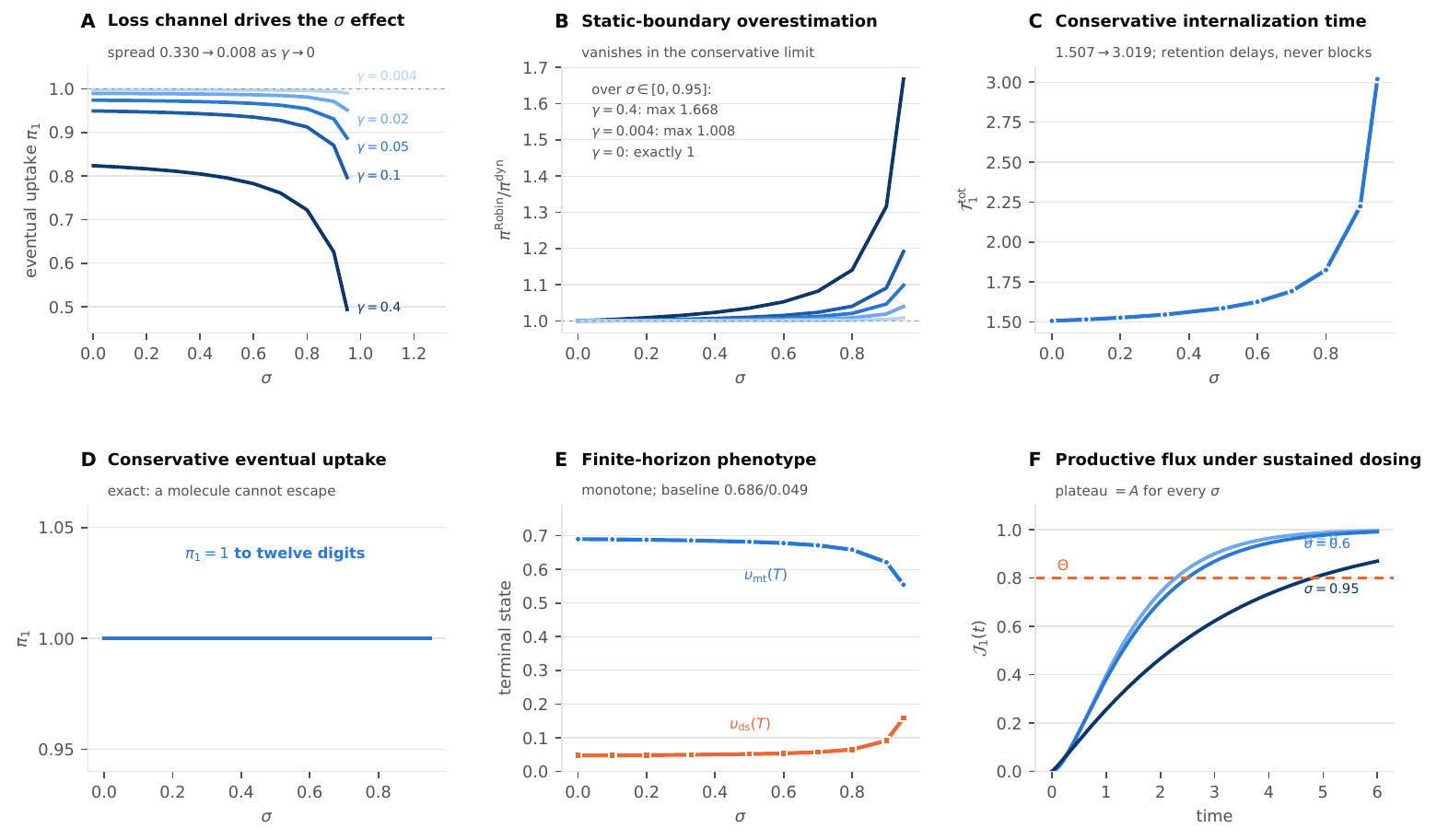}
\caption{(A) Eventual uptake as bulk loss is removed. (B) Static-boundary
overestimation, vanishing in the conservative limit. (C) Conservative
internalization time. (D) Conservative eventual uptake, exactly one.
(E) Finite-horizon phenotype. (F) Productive flux under sustained dosing.}
\label{fig:sigma}\end{figure}

\medskip
\noindent\textbf{Residence-law dispersion.} While retention primarily governs the delay and timing of uptake, the specific mathematical shape of the residence time distribution also fundamentally reshapes the transient response. To illustrate this, we compare three residence densities with a common mean of 1 and common $\sigma=0.6$: a Gamma with $\mathrm{CV}_\tau=0.5$, an exponential, and the
hyperexponential $0.1r_1e^{-r_1\tau}+0.9r_2e^{-r_2\tau}$ with
$r_1=(1+\sqrt{15.345})^{-1}=0.203365$ and
$r_2=(1-\sqrt{15.345}/9)^{-1}=1.770699$, which has mean exactly $1$ and
$\mathrm{CV}_\tau=2.1$. The surface-only compound-geometric time has common mean
$\langle\tau\rangle/(1-\sigma)=2.5$ and variances $4.375$, $6.250$, $14.775$,
with $\Pr(T_{\mathrm{surf}}<1)=0.239$, $0.330$, $0.450$: the long-tailed law
front-loads delivery. These are not the release-to-internalization
moments. Restoring the diffusive search through \eqref{eq:renewalpi} gives a
common full mean $\mathcal T^{\mathrm{tot}}=3.7685$ and standard deviations
$3.1264$, $3.4132$, $4.4916$, and hence equal mean exposure and equal per-encounter internalization
probability therefore give measurably different transients, so a single-rate
surface model, which fixes the whole law to the exponential one, cannot
reproduce the response. Because the three laws differ in all higher moments and
not only in $\mathrm{CV}_\tau$, this is a statement about the shape of the
residence law. Figure~\ref{fig:residence} displays the residence laws and the corresponding flux and phenotype trajectories.

\medskip
\noindent\textbf{Target competition.} Beyond the temporal dynamics governed by individual target retention, multi-target configurations introduce the spatial effect of local resource depletion. For two lossy targets at $(0.5\mp d/2,0.5)$, the raw total uptake $U_{12}(d)$ decreases monotonically with $d$, because increasing the spacing also moves both targets away from the source. Only the normalized ratio
\begin{equation}
\varrho_{\mathrm{comp}}(d)=\frac{U_{12}(d)}
{U^{\mathrm{iso}}_1(d)+U^{\mathrm{iso}}_2(d)}
\label{eq:comp}
\end{equation}
isolates competition. For a pair of identical targets with equal
self-interaction the amplitude matrix is
$\bigl(\begin{smallmatrix}b&c\\c&b\end{smallmatrix}\bigr)$ with
$b=1+2\pi\nu D\mathcal G_{11}+\nu\Psi$ and
$c=2\pi\nu DG(\mathbf x_1,\mathbf x_2)$, so
$A_1+A_2=\nu(\Gamma_1+\Gamma_2)/(b+c)$ while the isolated targets give
$\nu(\Gamma_1+\Gamma_2)/b$, whence
\begin{equation}
\varrho_{\mathrm{comp}}(d)=\frac{b}{b+c}.
\label{eq:compclosed}
\end{equation}
The forcing cancels identically, so $\varrho_{\mathrm{comp}}$ is
independent of the source position and of the dose, and depends only on the
target--target Green's function --- which is precisely what makes it the correct measure of competition. Sources at $(0.5,0.5)$,
$(0.5,0.25)$, $(0.5,0.85)$, $(0.25,0.5)$ and $(0.2,0.3)$ give identical values
to four decimals.

Because this experiment is lossy, $\sum_k\pi_k<1$ and the maps
$\pi_j(\mathbf x_0)$ are unnormalized uptake masses, not capture basins;
when a target-allocation map is wanted we report
$\widehat\pi_j=\pi_j/\sum_k\pi_k$ separately from the success mass.
Figure~\ref{fig:geometry} plots the raw and competition-normalized quantities.

\begin{figure}[htbp]\centering
\includegraphics[width=\linewidth]{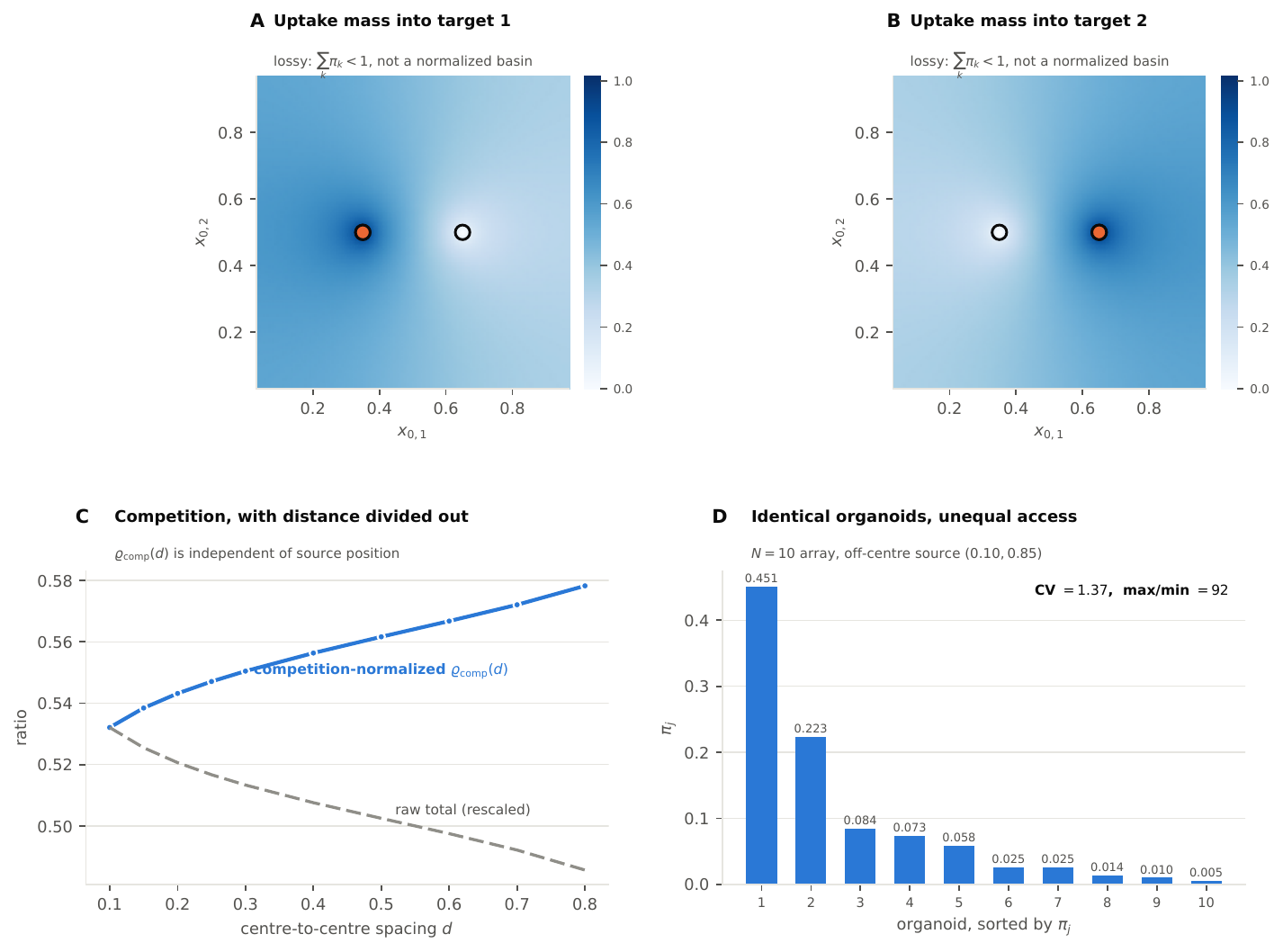}
\caption{(A,B) Lossy uptake masses for two competing targets. (C) The
competition-normalized ratio \eqref{eq:comp}, with the raw total shown
separately; the $d=0.05$ point is excluded from any asymptotic reading.
(D) A fully specified $N=10$ array under an off-centre source.}
\label{fig:geometry}\end{figure}

\medskip
\noindent\textbf{Remixing and rebinding.} Closely related to spatial competition between distinct targets is the probability of a desorbed molecule remixing and rebinding to the same target before escaping into the bulk. To quantify this, Corollary~\ref{cor:remix} is exercised on two disks at $(0.35,0.5)$ and
$(0.65,0.5)$ with source $(0.20,0.50)$, $\kappa'=2.0$, giving
$\overline{\bm\pi}=(0.7086,0.2914)$ and, from the matched-asymptotic evaluation
of \cite{bressloff2025asymptotic},
$\mathbf P(0)=\bigl(\begin{smallmatrix}0.9256&0.0744\\0.0744&0.9256\end{smallmatrix}\bigr)$
at $\varepsilon=0.02$. With $\bm\sigma=(0.1,0.9)$ the target-1 uptake at
$\eta=0,\tfrac14,\tfrac12,\tfrac34,1$ is $0.9563$, $0.9468$, $0.9314$, $0.9017$,
$0.8214$, an endpoint gap of $0.135$; replacing $\mathbf P$ by its leading-order
value $\mathbf I$ would give $0.248$. The correction is $O(\nu)$ and therefore
decays only logarithmically --- the fraction of the idealized gap recovered is
$0.545$ at $\varepsilon=0.02$, $0.692$ at $10^{-3}$ and $0.889$ at $10^{-10}$
--- so same-organoid rebinding is not a small correction at any experimentally
relevant size. With homogeneous retention $\bm\sigma=(0.5,0.5)$ the uptake
drifts only from $0.7086$ to $0.6815$ across $\eta$, the $O(\nu)$ residual of
Corollary~\ref{cor:remix}. Identifying $\eta$ therefore requires either
heterogeneous retention or time-resolved data.

\subsection{Global uncertainty and sensitivity}
\label{subsec:sens}

To evaluate the robustness of the identified biophysical mechanisms across the full parameter space rather than isolated deterministic sweeps, we supplement the analysis with Latin hypercube sampling, partial rank correlation, and variance-based Sobol indices \cite{mckay2000comparison,sobol2001global,marino2008methodology}, using the
first-order estimator of \cite{saltelli2010variance} and the total-effect
estimator of \cite{jansen1999analysis} (Appendix~\ref{app:mc}). As a primary baseline, we first perform a loss-channel audit on the lossy one-organoid benchmark at $s=0$, $\varepsilon=0.02$, with independent uniform ranges $D\in[0.5,1.5]$, $\kappa'\in[0.5,2.5]$, $\gamma\in[0.1,0.8]$,
$\gamma^d\in[0.1,2.0]$, $\bar\gamma\in[0.5,2.0]$, $\ell\in[0.6,1.4]$,
$r_0\in[0.3,0.8]$, and an LHS of $4096$ points (seed $12345$), gives uptake
quantiles $(0.270,0.491,0.802)$ and Table~\ref{tab:sens}; Sobol indices use two seven-dimensional base matrices from one scrambled fourteen-dimensional design of size $2^{15}$ (seed $54321$).

\begin{table}[htbp]\centering\small
\caption{Loss-channel audit (left) and eight-input closure audit (right).}
\label{tab:sens}
\begin{tabular}{@{}lrrr@{\hskip 2em}lrrr@{}}
\toprule
input&PRCC&$S_i$&$S_{T_i}$&input&PRCC&$S_i$&$S_{T_i}$\\
\midrule
$D$&$0.793$&$0.112$&$0.117$        &$D$&$0.799$&$0.118$&$0.123$\\
$\kappa'$&$0.666$&$0.056$&$0.061$  &$\kappa'$&$0.459$&$0.019$&$0.021$\\
$\gamma$&$-0.955$&$0.731$&$0.734$  &$\gamma$&$-0.944$&$0.636$&$0.643$\\
$\gamma^d$&$-0.449$&$0.017$&$0.021$&$\gamma^d$&$-0.803$&$0.114$&$0.123$\\
$\bar\gamma$&$0.348$&$0.009$&$0.013$&$\bar\gamma$&$0.659$&$0.052$&$0.060$\\
$\ell$&$0.630$&$0.042$&$0.044$     &$\ell$&$0.513$&$0.025$&$0.026$\\
$r_0$&$-0.487$&$0.020$&$0.022$     &$r_0$&$-0.448$&$0.018$&$0.019$\\
&&&                                 &$\eta$&$-0.196$&$0.003$&$0.004$\\
\midrule
$\sum S_i$&&$0.988$&                &$\sum S_i$&&$0.984$&\\
\bottomrule
\end{tabular}
\end{table}

First-order indices sum to $0.988$ and every $S_{T_i}-S_i<0.006$, so
interactions are weak over these ranges. Within this benchmark bulk clearance
dominates, which is direct evidence that the large $\sigma$-dependence of the
lossy sweep is generated by the loss channel.

While the single-target audit confirms the primary influence of bulk clearance, parameters governing spatial re-binding such as $\eta$ require a multi-target environment to manifest. Because $\eta$ has an identically zero index in a single-target setup, it is exercised in a separate two-target closure audit: 
with $q_{\mathrm{ads}}$ the irreversible first-adsorption mass, $\overline{\bm p}=(0.7,0.3)$,
$\sigma_1=\gamma^d/(\gamma^d+\bar\gamma)$,
$\sigma_2=\gamma^d/(\gamma^d+\bar\gamma/4)$,
$\mathbf P^{\mathrm{reset}}=q_{\mathrm{ads}}\overline{\bm p}\mathbf1^\top$ and
$\mathbf P^{\mathrm{cont}}=q_{\mathrm{ads}}\mathbf I$, the output
$\mathbf e_1^\top\bm W[\mathbf I-\mathbf P^{(\eta)}\bm\Sigma]^{-1}
q_{\mathrm{ads}}\overline{\bm p}$ with $\eta\sim\mathcal U(0,1)$ gives quantiles
$(0.117,0.282,0.568)$ and the right half of Table~\ref{tab:sens}. The $\eta$
indices are nonzero but small over this deliberately broad lossy box, and the
negative PRCC has the sign predicted by Corollary~\ref{cor:remix}: resetting
reweights uptake toward the lower-retention target, and increasing $\eta$ moves
the allocation back toward the first-adsorption split. Overall, these global audits demonstrate that while specific kinetic parameters quantitatively tune local uptake, the reduced spatial model remains structurally stable across broad parametric bounds.

\section{Acceptance criteria and batch size}
\label{sec:acceptance}
The null model of \S\ref{sec:null} asks how much between-organoid variance is attributable to transport. The same question at the next scale ---
how much between-batch variance is attributable to process rather than
biology --- has a structural answer worth recording, because it bounds what
batch size can achieve.

To quantify these bounds mathematically, we introduce a variance model with nested line, batch, and process random effects through a positive, centred, bounded multiplier. Let $Z$ denote independent standard normals
conditioned on $|Z|\le4$, let
$\mathcal N(c)=\mathbb E[e^{cZ}\mid|Z|\le4]$, and set
\begin{equation}
\theta_{bj}=\theta_0\prod_{r\in\{\mathrm{line},\mathrm{batch},\mathrm{proc}\}}
\frac{\exp(w_\theta s_rZ_{r,bj})}{\mathcal N(w_\theta s_r)},
\label{eq:varmodel}
\end{equation}
where the hierarchy is explicit:
$Z_{\mathrm{line},bj}=Z_{\mathrm{line}}$ is shared by all batches from the
line, $Z_{\mathrm{batch},bj}=Z_b$ is shared within batch $b$, and
$Z_{\mathrm{proc},bj}=Z_{bj}$ is independent across organoids. The calculation
uses
\begin{align*}
\mathcal P_+&=\{\kappa',\bar\gamma,\ell,r_{\mathrm{grow}},k_{PD},b_{PD},
k_{DM},b_{DM},k_{\mathrm{ap,clr}},\alpha_R\},\\
\mathcal P_-&=\{\gamma^d,d_0,k_{\mathrm{stress}},\alpha_Q,\Theta\},
\end{align*}
with $w_\theta=+1$ for $\theta\in\mathcal P_+$,
$w_\theta=-1$ for $\theta\in\mathcal P_-$, and $w_\theta=0$ for every other
parameter. Thus a positive quality shift raises reactivity, internalization,
size and maturation-promoting rates while lowering desorption and
disease-promoting rates. The exponential prevents negative rates, the centring
preserves $\mathbb E\theta_{bj}=\theta_0$, and the truncation gives compact
support --- which is what Proposition~\ref{prop:lip} needs. A signed loading,
rather than one unsigned scalar, is what allows a common quality shift to raise
reactivity while lowering desorption.

Having defined how process variability shifts the underlying kinetic parameters, we now map these variations to the final phenotypic criteria to determine batch success. Specifically, an organoid passes if
$\upsilon_{j,\mathrm{mt}}(T)\ge\upsilon^{\min}_{\mathrm{mt}}$ and
$\upsilon_{j,\mathrm{ds}}(T)\le\upsilon^{\max}_{\mathrm{ds}}$; a batch is
accepted if a fraction at least $r_{\min}$ of its $N_b$ organoids pass
and the batch mean maturation exceeds
$\overline\upsilon^{\min}_{\mathrm{mt}}$, and we write
$P_{\mathrm{accept}}$ for the probability of acceptance. Thresholds are set
relative to the qualified baseline of \S\ref{sec:numerics}
($\upsilon_{\mathrm{mt}}=0.686$, $\upsilon_{\mathrm{ds}}=0.049$): $\upsilon^{\min}_{\mathrm{mt}}=0.58$,
$\upsilon^{\max}_{\mathrm{ds}}=0.18$,
$\overline\upsilon^{\min}_{\mathrm{mt}}=0.62$, $r_{\min}=0.80$,
$s_{\mathrm{line}}=0.05$; $4\times10^4$ batches per grid point, seed $24680$.

The full grid of resulting acceptance probabilities is presented in Table~\ref{tab:accept}, which exhibits two key structural facts regarding batch size scaling. First. \textbf{variance scaling and shared effects.}
Linearizing any smooth potency output as
$Y_{bj}=Y_0+c_Y(s_{\mathrm{line}}Z_{\mathrm{line}}
+s_{\mathrm{batch}}Z_b+s_{\mathrm{proc}}Z_{bj})+o(s)$ and writing
$v_4=\Var(Z\mid|Z|\le4)=0.998929$ for the truncated-normal variance gives
\[
\Var(\overline Y_b)=c_Y^2v_4\Bigl(s^2_{\mathrm{line}}+s^2_{\mathrm{batch}}
+\frac{s^2_{\mathrm{proc}}}{N_b}\Bigr)+o(s^2).
\]
Only the independent within-batch term is averaged away, and only as $N_b^{-1}$;
the two shared terms are unaffected by batch size. (If the line effect is held
fixed --- a single qualified line --- the calculation is conditional on
$Z_{\mathrm{line}}$ and the $s^2_{\mathrm{line}}$ term is absent.)

\begin{table}[htbp]\centering\small
\caption{Acceptance probability $P_{\mathrm{accept}}$ at $s_{\mathrm{line}}=0.05$
throughout. The final column is the $N_b\to\infty$ limit
\eqref{eq:acceptlimit}, evaluated by quadrature; note that it is interior in
every row. Rows are labelled by which variance component is varied, not by which
is present.}
\label{tab:accept}
\begin{tabular}{@{}rr|rrrrr|r@{}}
\toprule
$s_{\mathrm{batch}}$&$s_{\mathrm{proc}}$&$N_b{=}5$&$10$&$20$&$50$&$100$&limit\\
\midrule
\multicolumn{8}{@{}l}{\emph{varying batch variance, at fixed line variance}}\\
0.00&0.000&0.978&0.979&0.979&0.980&0.978&0.980\\
0.05&0.000&0.923&0.924&0.922&0.923&0.921&0.923\\
0.10&0.000&0.807&0.806&0.805&0.810&0.805&0.808\\
0.15&0.000&0.717&0.715&0.714&0.717&0.714&0.717\\
0.20&0.000&0.659&0.655&0.656&0.653&0.656&0.656\\
0.25&0.000&0.613&0.609&0.612&0.607&0.611&0.609\\
\midrule
\multicolumn{8}{@{}l}{\emph{varying process variance, at fixed line variance}}\\
0.00&0.050&0.963&0.971&0.974&0.974&0.977&0.976\\
0.00&0.100&0.890&0.900&0.909&0.914&0.919&0.917\\
0.00&0.125&0.824&0.825&0.821&0.821&0.823&0.820\\
0.00&0.150&0.745&0.729&0.704&0.684&0.676&0.662\\
0.00&0.200&0.605&0.535&0.449&0.374&0.336&0.280\\
0.00&0.250&0.493&0.377&0.268&0.153&0.104&0.046\\
\midrule
\multicolumn{8}{@{}l}{\emph{both varied}}\\
0.10&0.100&0.743&0.738&0.741&0.729&0.726&0.717\\
0.15&0.150&0.610&0.578&0.559&0.539&0.535&0.523\\
\bottomrule
\end{tabular}
\end{table}

Second. \textbf{Non-degenerate large-batch limits.} Consequently, because the shared variance components do not average out, the overall acceptance probability does not tend to zero or one. Conditional on the shared effects $H=(Z_{\mathrm{line}},Z_{\mathrm{batch}})$ the organoid outcomes are
independent and identically distributed, so the pass fraction concentrates not
on a constant but on the random conditional pass probability. Writing
\[
p(H)=\Pr\bigl(\text{organoid passes}\mid H\bigr),
\qquad
\mu(H)=\mathbb E\bigl[\upsilon_{j,\mathrm{mt}}(T)\mid H\bigr],
\]
the law of large numbers applied conditionally gives
\begin{equation}
P_{\mathrm{accept}}\;\longrightarrow\;
\Pr_H\bigl\{\,p(H)\ge r_{\min}\ \text{and}\ \mu(H)\ge
\overline\upsilon^{\min}_{\mathrm{mt}}\,\bigr\},
\qquad N_b\to\infty,
\label{eq:acceptlimit}
\end{equation}
which is a probability over the shared effects and is generally strictly
interior. It degenerates to zero or one only when the shared effects are absent
($s_{\mathrm{line}}=s_{\mathrm{batch}}=0$), or after conditioning on them. The
final column of Table~\ref{tab:accept} evaluates \eqref{eq:acceptlimit} by
quadrature and the finite-$N_b$ columns converge to it in every row.

It is worth being explicit about the statistical fallacy this conditional limit corrects, as it is easy to conflate the marginal expected pass rate with the actual limit. The marginal pass probability $\mathbb E[p(H)]$ is not the relevant threshold. At $s_{\mathrm{proc}}=0.15$ we have
$\mathbb E[p(H)]=0.826>r_{\min}$, which would suggest
$P_{\mathrm{accept}}\to1$; the true limit is $0.662$. At
$s_{\mathrm{proc}}=0.25$, $\mathbb E[p(H)]=0.693<r_{\min}$ would suggest
$P_{\mathrm{accept}}\to0$; the true limit is $0.046$. In these two rows the
batch-mean criterion is nonbinding, so the governing quantity reduces to
$\Pr_H\{p(H)\ge r_{\min}\}$ --- respectively $0.662$ and $0.046$ --- not
whether the average of $p(H)$ clears the threshold. In general both conditions
in \eqref{eq:acceptlimit} must be retained.

\medskip
\noindent\textbf{The limits of batch sizing.} Increasing batch size therefore buys down only the sampling noise in the pass fraction. It leaves an irreducible acceptance risk set by the shared effects
and enlarging batches past the point where
$s^2_{\mathrm{proc}}/N_b\ll s^2_{\mathrm{line}}+s^2_{\mathrm{batch}}$ buys
nothing. The direction of the $N_b$-dependence still changes sign --- rising
where the typical $p(H)$ sits above $r_{\min}$ and falling where it sits below,
with the crossing near $s_{\mathrm{proc}}\approx0.125$ --- but between two
interior values, not between $0$ and $1$. Any statement ranking process against
batch variability is conditional on $N_b$, on the potency margin and on the
thresholds, and we make none.
Figure~\ref{fig:robust} visualizes the finite-$N_b$ values
and their nondegenerate large-batch limits.

\begin{figure}[htbp]\centering
\includegraphics[width=\linewidth]{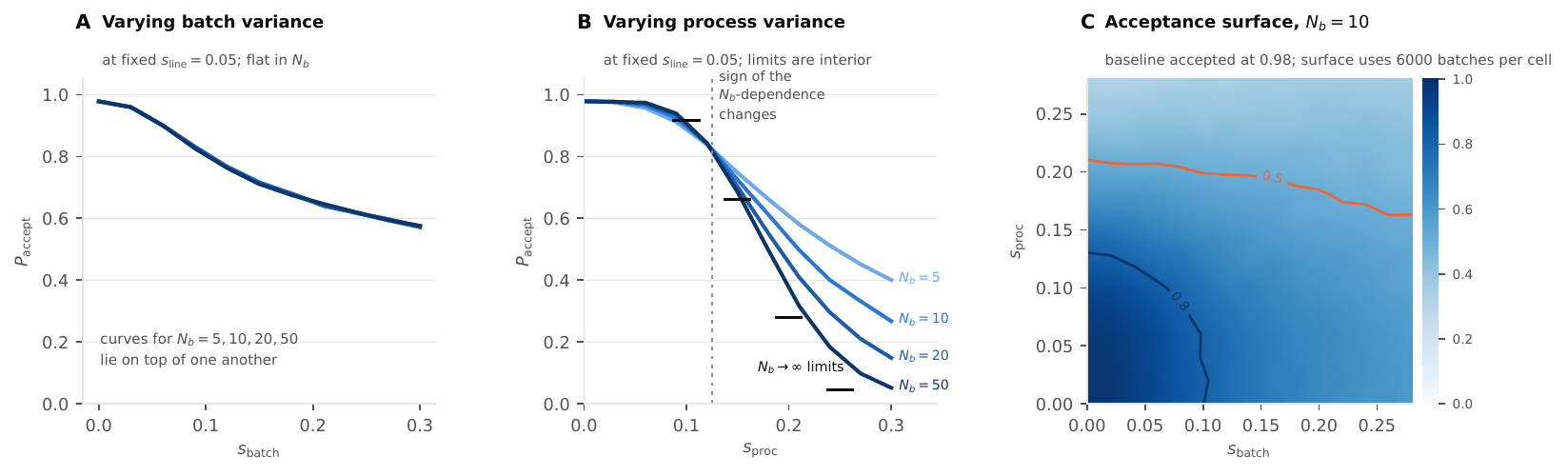}
\caption{Acceptance probability; $s_{\mathrm{line}}=0.05$ is active
throughout. (A) Varying the batch variance at fixed line variance: flat in
$N_b$, because neither shared component averages out. (B) Varying the process
variance at fixed line variance: the $N_b$-dependence changes sign near
$s_{\mathrm{proc}}\approx0.125$, but between interior values --- dashed marks
give the $N_b\to\infty$ limits \eqref{eq:acceptlimit}, which are $0.917$ and
$0.046$ at $s_{\mathrm{proc}}=0.10$ and $0.25$ and are neither $1$ nor $0$.
(C) The acceptance surface at $N_b=10$; the baseline batch is accepted with
probability $0.979$.}
\label{fig:robust}\end{figure}

\section{Discussion}
\label{sec:discussion}

This study quantifies the heterogeneity generated by geometry and delivery when
organoid biology is identical. Across 2000 ten-organoid arrays, the median
transport-only coefficient of variation is $0.623$ for maturation and $0.835$
for disease score; the one-factor design audit moves the maturation median from
$0.27$ to $0.86$. These are model- and layout-specific reference values, not
empirical thresholds. Equation~\eqref{eq:nullmetrics} should therefore be used
as a null against which an imaged array is tested: exceeding it rejects the
specified transport-only explanation, whereas failing to exceed it does not
establish biological homogeneity.

The analysis separates three effects that raw uptake conflates. First, in a
conservative chamber eventual productive uptake is one; retention changes its
timing and allocation, and changes its total only through interaction with
clearance. This agrees with the sensitivity audit, where bulk clearance has
$S_\gamma=0.731$. Second, equal mean residence times do not imply equal
transients: the Gamma, exponential, and hyperexponential laws produce distinct
flux and phenotype histories, so residence-law shape must be retained when
timing matters. Third, the normalized two-target ratio
\eqref{eq:compclosed} cancels source forcing and isolates competition. In lossy
systems, $\pi_j$ are uptake masses rather than probabilities and allocation
must use $\widehat\pi_j=\pi_j/\sum_k\pi_k$.

Partial remixing yields a complementary identifiability result. Continued
search gives $\bm\pi^{(1)}=\overline{\bm\pi}+O(\nu)$, whereas resetting
reweights first-adsorption probabilities by $1-\sigma_j$. Endpoint data are
therefore informative about remixing chiefly when retention is heterogeneous;
time-resolved uptake remains generically informative under homogeneous
retention. This is structural, not practical, identifiability: recovery from
noisy data is not established here.

The strongest design prediction is that distributed dosing equalizes the
transformed forcing. After matching array-mean exposure, it reduces the median
maturation spread fivefold in the baseline ensemble and by factors
$2.1$--$8.4$ across the tested designs. Four separated ports already lower the
median coefficient of variation to $0.27$, versus $0.12$ for ideal uniform
dosing. Increasing batch size addresses a different problem: it averages only
independent process noise, while shared line and batch effects survive and
produce the generally nondegenerate limit \eqref{eq:acceptlimit}. Hence neither
dosing benefits nor variance rankings are universal effect sizes.

The individual mathematical ingredients are not presented as new. The renewal formulation, encounter-based adsorption, Dirichlet-to-Neumann reduction, and target-to-target re-adsorption machinery are inherited from earlier work \cite{bressloff2025diffusion,bressloff2025asymptotic,bressloff2026renewal}. Resetting mixtures and geminate re-encounters are also classical \cite{kusmierz2014first,riascos2020random,gonzalez2021diffusive,
janson2012hitting,noyes1955kinetics,
northrup1979short,agmon1990theory,berg1977physics}. The contribution here is their composition into a transport-to-phenotype null model for liver-cancer organoid arrays, together with the forcing-equalization result, the corrected interpretation of retention, the competition normalization, the remixing identifiability consequence, and the propagation to a batch-acceptance limit.

The two principal numerical approximations were tested separately. First, the Green-function reduction was compared with an exact radial benchmark and with $P_2$ finite-element solutions of the full perforated-domain problem for $N=2,5,10$ targets, both conservative and lossy, at transform points $s\in\{0,1,5\}$. The measured convergence slopes are $2.00$--$2.15$, consistent with the $O(\varepsilon^2)$ productive-flux error proved for circular targets. The verification is functional-specific: at separations of about four target
radii, the phenotype-CV error is already below $10^{-3}$ relative, whereas individual flux errors remain approximately $0.5$--$2\%$. The null ensemble and the coarsest design cells lie in the tested range where phenotype coefficients of variation are accurate to well below $1\%$, but this does not imply $10^{-3}$ accuracy for every individual flux. Second, numerical Laplace inversion shows that the Gamma reconstruction changes the principal between-organoid coefficient of variation by less than $0.06\%$ and terminal maturation by less than $0.4\%$ in the Markovian calculations supporting the null model. The non-Markovian residence comparison is reported from exact inversion because the same surrogate is less reliable there.

The model is deliberately limited. Parameters are literature-scaled rather
than line-specific, and the two-dimensional compartment model omits
intra-organoid gradients, cell composition, necrotic cores, matrix mechanics,
immune interactions, convection, and realistic perfusion. The three-dimensional
reduction is formal and not numerically verified; accessibility is
time independent; the design audit is one factor at a time; and the cell-state
and release thresholds are illustrative. The immediate tests are therefore
specific: compare localized with distributed dosing in the same imaged array,
and use heterogeneous-retention or time-resolved experiments to probe remixing.
The central conclusion is correspondingly narrow: transport alone can generate
substantial, design-dependent apparent heterogeneity and should be quantified
before residual variation is attributed to biology.

\section*{Appendix}

\begin{appendices}

\section{Well-posedness, positivity, and mass balance}
\label{app:wellposedness}

Let
\[
\Gamma_0=\partial\Omega,\qquad
\Gamma_j=\partial\mathcal U_j,\qquad
\partial\Omega_\varepsilon=\Gamma_0\cup\bigcup_{j=1}^N\Gamma_j.
\]
The unit normal $\mathbf n_j$ points from $\mathcal U_j$ into
$\Omega_\varepsilon$; hence the outward normal $\boldsymbol\nu_\varepsilon$
of $\Omega_\varepsilon$ satisfies
$\boldsymbol\nu_\varepsilon=-\mathbf n_j$ on $\Gamma_j$.

\begin{proposition}[Well-posedness and positivity]
\label{prop:bulk_surface_wellposed}
Assume that $\Omega_\varepsilon$ is a bounded $C^2$ domain, $D_m>0$, and
\[
\gamma_m,\ \kappa_{jm},\ \gamma^d_{jm},\ \bar\gamma_{jm}\geq0.
\]
Let
\[
u_{m,0}\in L^2(\Omega_\varepsilon),\qquad
q_{jm,0}\in L^2(\Gamma_j),\qquad
I_m\in L^1\!\left(0,T;L^2(\Omega_\varepsilon)\right).
\]
Then \eqref{eq:bulk}--\eqref{eq:surface} has a unique mild solution
\[
u_m\in C\!\left([0,T];L^2(\Omega_\varepsilon)\right),\qquad
q_{jm}\in C\!\left([0,T];L^2(\Gamma_j)\right).
\]
If $u_{m,0}$, $q_{jm,0}$, and $I_m$ are nonnegative almost everywhere, then
$u_m$ and $q_{jm}$ remain nonnegative almost everywhere.
\end{proposition}

\begin{proof}
Fix $m$ and suppress its index. Set
\[
\mathbb H=L^2(\Omega_\varepsilon)\oplus
\bigoplus_{j=1}^N L^2(\Gamma_j),\qquad
\mathbb V=H^1(\Omega_\varepsilon)\oplus
\bigoplus_{j=1}^N L^2(\Gamma_j).
\]
For $\mathbf z=(u,q_1,\ldots,q_N)$ and
$\mathbf w=(v,r_1,\ldots,r_N)$ in $\mathbb V$, define
\begin{equation}
\mathfrak a(\mathbf z,\mathbf w)
=D\int_{\Omega_\varepsilon}\nabla u\cdot\nabla v
+\gamma\int_{\Omega_\varepsilon}uv +\sum_{j=1}^N\left[
\kappa_j\int_{\Gamma_j}uv
-\gamma^d_j\int_{\Gamma_j}q_jv
-\kappa_j\int_{\Gamma_j}ur_j
+(\gamma^d_j+\bar\gamma_j)\int_{\Gamma_j}q_jr_j
\right].
\label{eq:appendix_form}
\end{equation}
The trace inequality
\[
\|v\|_{L^2(\Gamma_j)}^2
\leq \delta\|\nabla v\|_{L^2(\Omega_\varepsilon)}^2
+C_\delta\|v\|_{L^2(\Omega_\varepsilon)}^2
\]
implies continuity of $\mathfrak a$. Moreover,
\begin{equation}
\mathfrak a(\mathbf z,\mathbf z)
=D\|\nabla u\|_2^2+\gamma\|u\|_2^2
+\sum_j\kappa_j\|u\|_{L^2(\Gamma_j)}^2
+\sum_j(\gamma^d_j+\bar\gamma_j)\|q_j\|_2^2 -\sum_j(\gamma^d_j+\kappa_j)\int_{\Gamma_j}u q_j.
\end{equation}
Young's inequality and the trace inequality give $\omega\geq0$ and $c_0>0$
such that
\[
\mathfrak a(\mathbf z,\mathbf z)+\omega\|\mathbf z\|_{\mathbb H}^2
\geq c_0\|\mathbf z\|_{\mathbb V}^2.
\]
Thus $\mathfrak a$ is a densely defined, closed, quasi-coercive form. Its
associated operator $\mathcal A$ generates an analytic $C_0$-semigroup on
$\mathbb H$, and
\begin{equation}
\mathbf z(t)=e^{-t\mathcal A}\mathbf z_0+
\int_0^t e^{-(t-\tau)\mathcal A}
\bigl(I(\tau),0,\ldots,0\bigr)\,\mathrm d\tau
\label{eq:appendix_mild}
\end{equation}
is the unique mild solution. The operator domain encodes
\[
D\partial_{\boldsymbol\nu_\varepsilon}u=0\quad\hbox{on }\Gamma_0,
\qquad
D\partial_{\boldsymbol\nu_\varepsilon}u=-\kappa_j u+\gamma^d_jq_j
\quad\hbox{on }\Gamma_j.
\]

For positivity, put $u^-=\max\{-u,0\}$,
$q_j^-=\max\{-q_j,0\}$, and
$\mathcal N_-=\|u^-\|_2^2+\sum_j\|q_j^-\|_2^2$. Testing with
$-u^-$ and $-q_j^-$ gives
\begin{align}
\frac12\frac{\mathrm d}{\mathrm dt}\mathcal N_-
&+D\|\nabla u^-\|_2^2+\gamma\|u^-\|_2^2
+\sum_j\kappa_j\|u^-\|_{L^2(\Gamma_j)}^2
+\sum_j(\gamma^d_j+\bar\gamma_j)\|q_j^-\|_2^2 \nonumber\\
&+\sum_j\left[
\gamma^d_j\int_{\Gamma_j}q_j u^-
+\kappa_j\int_{\Gamma_j}u q_j^-
\right]
=-\int_{\Omega_\varepsilon}Iu^-\leq0.
\label{eq:appendix_negative}
\end{align}
Because
$\gamma^d_jq_ju^-+\kappa_juq_j^-
\geq-(\gamma^d_j+\kappa_j)u^-q_j^-$, the trace and Young inequalities yield
$\mathcal N_-'\leq C\mathcal N_-$. Gronwall's inequality proves positivity.
The argument follows first for Galerkin approximants and then by weak lower
semicontinuity.
\end{proof}

\section{Two-dimensional matched-asymptotic reduction}
\label{app:matched2d}

Fix $m$, suppress it where unambiguous, and set
\[
a(s)=s+\gamma_m,\quad
\chi_j(s)=s+\gamma^d_{jm}+\bar\gamma_{jm},\quad
g_j(s)=\frac{s+\bar\gamma_{jm}}{\chi_j(s)},\quad
K_j(s)=\kappa'_{jm}g_j(s).
\]
Assume $q_{jm,0}=0$ and $\kappa_{jm}=\kappa'_{jm}/\varepsilon$. Then
\begin{subequations}
\begin{align*}
D_m\Delta\widetilde u_m-a(s)\widetilde u_m
&=-f_m^{\mathrm{src}}
&&\text{in }\Omega_\varepsilon,\\
D_m\partial_{\mathbf n}\widetilde u_m&=0
&&\text{on }\partial\Omega,\\
D_m\partial_{\mathbf n_j}\widetilde u_m
&=\frac{K_j(s)}{\varepsilon}\widetilde u_m
&&\text{on }\partial\mathcal U_j.
\end{align*}
\end{subequations}
Moreover,
\begin{equation*}
\widetilde{\mathcal J}^{\mathrm{full}}_{jm}(s)
=\frac{\bar\gamma_{jm}\kappa'_{jm}}{\varepsilon\chi_j(s)}
\int_{\partial\mathcal U_j}\widetilde u_m\,\mathrm dS.
\end{equation*}

\begin{assumption}[Uniform admissible family]
\label{ass:app_uniform_family}
Let $\Omega\subset\mathbb R^2$ be $C^4$ and
$\mathcal U_j=B_{\varepsilon\ell_j}(\mathbf x_j)$, where
$0<\ell_-\leq\ell_j\leq\ell_+$. For some $d_*>0$,
\[
|\mathbf x_i-\mathbf x_j|\geq d_*,\qquad
\operatorname{dist}(\mathbf x_j,\partial\Omega)\geq d_*,\qquad
\operatorname{dist}(\mathbf x_j,\mathcal S_m)\geq d_*.
\]
Source strengths and the regular source component are uniformly bounded. On a
compact interval $\mathcal K_s\subset[0,\infty)$, assume
$a(s)\geq a_->0$, $0<K_-\leq K_j(s)\leq K_+$ and
\begin{equation*}
\sup_{s\in\mathcal K_s}
\left\|\left[
\mathbf I+2\pi\nu D_m\bm{\mathcal G}_m(s)+\nu\bm\Psi_m(s)
\right]^{-1}\right\|\leq C_M,
\qquad \nu=-\frac1{\log\varepsilon}.
\end{equation*}
\end{assumption}

The condition $a(s)\geq a_->0$ includes $s=0$ when $\gamma_m>0$; the
conservative point is treated in Appendix~\ref{app:zeromode}.

\subsection{Outer solution}

Let
\begin{equation*}
D_m\Delta_{\mathbf x}G_m-(s+\gamma_m)G_m=-\delta_{\mathbf z},
\qquad \partial_{\mathbf n}G_m=0\quad\text{on }\partial\Omega,
\end{equation*}
and define
\begin{equation*}
\Gamma_m(\mathbf x,s)=\int_\Omega
G_m(\mathbf x,\mathbf z;s)f_m^{\mathrm{src}}(\mathbf z,s)
\,\mathrm d\mathbf z.
\end{equation*}
The monopole outer approximation is
\begin{equation*}
\widetilde u_m^{\mathrm{out}}(\mathbf x,s)
=\Gamma_m(\mathbf x,s)-2\pi D_m\sum_{k=1}^N
A_{km}(s)G_m(\mathbf x,\mathbf x_k;s).
\end{equation*}
Locally,
\begin{equation*}
G_m(\mathbf x,\mathbf x_j;s)
=-\frac{\log|\mathbf x-\mathbf x_j|}{2\pi D_m}
+R_m(\mathbf x_j,\mathbf x_j;s)
+O\!\left(r_j^2(1+|\log r_j|)\right),
\end{equation*}
where $r_j=|\mathbf x-\mathbf x_j|$. With
$\mathbf x=\mathbf x_j+\varepsilon\mathbf y$, $\rho=|\mathbf y|$, and
$\widehat{\mathbf y}=\mathbf y/\rho$,
\begin{align*}
\widetilde u_m^{\mathrm{out}}
={}&A_{jm}\log\rho+\Gamma_{jm}-\frac{A_{jm}}{\nu}
-2\pi D_m\left[A_{jm}R_{jj,m}
+\sum_{k\ne j}A_{km}G_{jk,m}\right] \nonumber\\
&+\varepsilon\rho\,\mathbf h_{jm}\cdot\widehat{\mathbf y}
+O\!\left(\varepsilon^2\rho^2
(1+|\log(\varepsilon\rho)|)\right),
\end{align*}
where $\Gamma_{jm}=\Gamma_m(\mathbf x_j,s)$,
$G_{jk,m}=G_m(\mathbf x_j,\mathbf x_k;s)$,
$R_{jj,m}=R_m(\mathbf x_j,\mathbf x_j;s)$,
and $\mathbf h_{jm}$ is the gradient at $\mathbf x_j$ of the regular outer
field.

\subsection{Inner problem and matching}

Write
\[
\widetilde u_m(\mathbf x_j+\varepsilon\mathbf y,s)
\sim U^{(0)}_{jm}(\mathbf y,s)+\varepsilon U^{(1)}_{jm}(\mathbf y,s)+\cdots.
\]
At leading order,
\[
\Delta_{\mathbf y}U^{(0)}_{jm}=0,\qquad
D_m\partial_\rho U^{(0)}_{jm}=K_j(s)U^{(0)}_{jm}
\quad\text{at }\rho=\ell_j.
\]
The radial solution with logarithmic strength $A_{jm}$ is
\[
U^{(0)}_{jm}(\rho,s)
=A_{jm}\log\rho+A_{jm}\bigl[\Psi_{jm}(s)-\log\ell_j\bigr],
\qquad
\Psi_{jm}(s)=\frac{D_m}{K_j(s)\ell_j}.
\]
Matching constants gives
\[
\Gamma_{jm}=\frac{A_{jm}}{\nu}
+2\pi D_m\sum_{k\ne j}G_{jk,m}A_{km}
+\left[2\pi D_mR_{jj,m}-\log\ell_j+\Psi_{jm}\right]A_{jm}.
\]
Define
\[
(\bm{\mathcal G}_m)_{jk}=
\begin{cases}
G_m(\mathbf x_j,\mathbf x_k;s),&j\ne k,\\[1mm]
R_m(\mathbf x_j,\mathbf x_j;s)
-\dfrac{\log\ell_j}{2\pi D_m},&j=k,
\end{cases}
\qquad
\bm\Psi_m=\operatorname{diag}(\Psi_{1m},\ldots,\Psi_{Nm}).
\]
Then
\begin{equation}
\left[\nu^{-1}\mathbf I+2\pi D_m\bm{\mathcal G}_m(s)
+\bm\Psi_m(s)\right]\mathbf A_m(s)=\bm\Gamma_m(s),
\label{eq:app_matching_system}
\end{equation}
or
\begin{equation}
\mathbf A_m(s)=\nu\left[
\mathbf I+2\pi\nu D_m\bm{\mathcal G}_m(s)+\nu\bm\Psi_m(s)
\right]^{-1}\bm\Gamma_m(s).
\label{eq:app_amplitude}
\end{equation}
In particular, $\sup_{s\in\mathcal K_s}\|\mathbf A_m(s)\|\leq C\nu$.

\subsection{Productive flux}

Since $U^{(0)}_{jm}(\ell_j,s)=A_{jm}\Psi_{jm}$,
\begin{equation}
\widetilde{\mathcal J}^{\mathrm{red}}_{jm}(s)
=\frac{\bar\gamma_{jm}\kappa'_{jm}}{\varepsilon\chi_j(s)}
\int_0^{2\pi}U^{(0)}_{jm}(\ell_j,s)\,
\varepsilon\ell_j\,\mathrm d\theta =\frac{2\pi D_m\bar\gamma_{jm}}{s+\bar\gamma_{jm}}A_{jm}(s).
\label{eq:app_reduced_flux}
\end{equation}

\subsection{Remainder and scope}

The first angular correction is
\begin{equation}
U^{(1)}_{jm}(\rho,\theta;s)
=(\mathbf h_{jm}\cdot\widehat{\mathbf y})
\left(\rho+\beta_{jm}(s)\frac{\ell_j^2}{\rho}\right),
\qquad
\beta_{jm}(s)=\frac{D_m-K_j(s)\ell_j}{D_m+K_j(s)\ell_j}.
\label{eq:app_dipole}
\end{equation}
It satisfies the Robin condition and
\begin{equation}
\int_0^{2\pi}U^{(1)}_{jm}(\ell_j,\theta;s)\,\mathrm d\theta
=\int_0^{2\pi}\partial_\rho U^{(1)}_{jm}(\ell_j,\theta;s)
\,\mathrm d\theta=0.
\label{eq:app_zero_dipole_flux}
\end{equation}

\begin{proposition}[Conditional uniform flux estimate]
\label{prop:app_flux_error}
Suppose, in addition to Assumption~\ref{ass:app_uniform_family}, that a
second-order composite approximation $\widetilde u_m^{\mathrm{app}}$ has
residuals
$\mathcal R_{\varepsilon,m}=D_m\Delta\widetilde u_m^{\mathrm{app}}
-a(s)\widetilde u_m^{\mathrm{app}}+f_m^{\mathrm{src}}$,
$\mathcal B_{\varepsilon,jm}=D_m\partial_{\mathbf n_j}\widetilde u_m^{\mathrm{app}}
-\dfrac{K_j(s)}{\varepsilon}\widetilde u_m^{\mathrm{app}}$,
and that
\begin{equation}
\|\mathcal R_{\varepsilon,m}\|_{H^{-1}(\Omega_\varepsilon)}
+\left[\sum_{j=1}^N\frac{\varepsilon}{K_j(s)}
\|\mathcal B_{\varepsilon,jm}\|_{L^2(\partial\mathcal U_j)}^2
\right]^{1/2}\leq C_R\varepsilon^2.
\label{eq:app_residual_bound}
\end{equation}
If the mean trace of its second inner corrector is uniformly bounded, then
\[
\sup_{s\in\mathcal K_s}\max_{1\leq j\leq N}
\left|\widetilde{\mathcal J}^{\mathrm{full}}_{jm}(s)
-\widetilde{\mathcal J}^{\mathrm{red}}_{jm}(s)\right|
\leq C\varepsilon^2.
\]
The constant is uniform in $j$, $s$, and all admissible source and target
positions, but not as $d_*\downarrow0$, $a_-\downarrow0$, $K_-\downarrow0$,
$N\to\infty$, or $C_M\to\infty$.
\end{proposition}

\begin{proof}
Let $w_m=\widetilde u_m^{\mathrm{full}}-\widetilde u_m^{\mathrm{app}}$ and
$\|v\|_{\mathcal E}^2=D_m\|\nabla v\|_2^2+a(s)\|v\|_2^2
+\sum_j\frac{K_j(s)}{\varepsilon} \|v\|_{L^2(\partial\mathcal U_j)}^2$.
The error equation tested with $w_m$ and \eqref{eq:app_residual_bound} yield
$\|w_m\|_{\mathcal E}\leq C\varepsilon^2$. Hence
\[
\left|\widetilde{\mathcal J}^{\mathrm{full}}_{jm}
-\widetilde{\mathcal J}^{\mathrm{app}}_{jm}\right| \leq\frac{\bar\gamma_{jm}}{s+\bar\gamma_{jm}}
\left(\frac{K_j|\partial\mathcal U_j|}{\varepsilon}\right)^{1/2}
\left(\frac{K_j}{\varepsilon}
\|w_m\|_{L^2(\partial\mathcal U_j)}^2\right)^{1/2}\leq C\varepsilon^2,
\]
because $K_j|\partial\mathcal U_j|/\varepsilon=2\pi K_j\ell_j$.
Equation~\eqref{eq:app_zero_dipole_flux} removes the $O(\varepsilon)$
contribution to the integrated flux; the bounded second mean gives
$\widetilde{\mathcal J}^{\mathrm{app}}_{jm}
=\widetilde{\mathcal J}^{\mathrm{red}}_{jm}+O(\varepsilon^2)$.
\end{proof}

\paragraph{Status of the remainder.}
Equations~\eqref{eq:app_matching_system} and \eqref{eq:app_reduced_flux} are the
formal matched expansion. The $O(\varepsilon^2)$ estimate is rigorous only
under the explicit residual hypothesis \eqref{eq:app_residual_bound}; the
zero-mean dipole alone does not prove that hypothesis.

\section{Conservative zero-mode cancellation}
\label{app:zeromode}

Fix $m$, set $\gamma_m=0$, and write $V=|\Omega|$. Let
\[
-D_m\Delta\varphi_n=\lambda_n\varphi_n,\qquad
\partial_{\mathbf n}\varphi_n=0,\qquad
\int_\Omega\varphi_n\varphi_k\,\mathrm d\mathbf x=\delta_{nk},
\]
where $\lambda_0=0$, $\varphi_0=V^{-1/2}$, and
$0<\lambda_1\leq\lambda_2\leq\cdots$.

\begin{proposition}[Neumann zero mode]
\label{prop:app_neumann_zero_mode}
For $s>0$,
\begin{equation*}
G_m(\mathbf x,\mathbf z;s)=\frac1{sV}+G_m^0(\mathbf x,\mathbf z)
+sG_m^1(\mathbf x,\mathbf z)+O(s^2),
\end{equation*}
where
\[
G_m^0(\mathbf x,\mathbf z)=\sum_{n\geq1}\frac{\varphi_n(\mathbf x)\varphi_n(\mathbf z)}{\lambda_n},\quad
G_m^1(\mathbf x,\mathbf z)=-\sum_{n\geq1}\frac{\varphi_n(\mathbf x)\varphi_n(\mathbf z)}{\lambda_n^2}.
\]
Moreover,
\begin{equation}
D_m\Delta_{\mathbf x}G_m^0=-\delta_{\mathbf z}+\frac1V,\qquad
\partial_{\mathbf n}G_m^0=0,\qquad
\int_\Omega G_m^0(\mathbf x,\mathbf z)\,\mathrm d\mathbf x=0.
\label{eq:app_pseudogreen_problem}
\end{equation}
The expansion is locally uniform away from $\mathbf x=\mathbf z$, and for the
diagonal regular part after subtraction of
$-\log|\mathbf x-\mathbf z|/(2\pi D_m)$.
\end{proposition}

\begin{proof}
Use
$G_m(\mathbf x,\mathbf z;s)=\sum_{n\geq0}
\dfrac{\varphi_n(\mathbf x)\varphi_n(\mathbf z)}{\lambda_n+s}$
and expand $(\lambda_n+s)^{-1}$ for $n\geq1$.
\end{proof}

Let
$Q_m(s)=\int_\Omega f_m^{\mathrm{src}}(\mathbf z,s)\,\mathrm d\mathbf z$ and $Q_m(s)\longrightarrow Q_{m,0}<\infty$.
Then
\begin{equation}
\bm\Gamma_m(s)=\frac{Q_m(s)}{sV}\mathbf1+\bm r_m(s),
\qquad
(r_m^0)_j=\int_\Omega G_m^0(\mathbf x_j,\mathbf z)
f_m^{\mathrm{src}}(\mathbf z,0)\,\mathrm d\mathbf z,
\label{eq:app_forcing_decomposition}
\end{equation}
where $\bm r_m(s)\to\bm r_m^0$. Likewise,
$\bm{\mathcal G}_m(s)=\dfrac1{sV}\mathbf1\mathbf1^\top
+\bm{\mathcal G}_m^0+O(s)$,
with
\[
(\bm{\mathcal G}_m^0)_{jk}=
\begin{cases}
G_m^0(\mathbf x_j,\mathbf x_k),&j\ne k,\\[1mm]
R_m^0(\mathbf x_j,\mathbf x_j)
-\dfrac{\log\ell_j}{2\pi D_m},&j=k.
\end{cases}
\]

\begin{proposition}[Rank-one cancellation]
\label{prop:app_rank_one_cancellation}
Suppose $\bm\Psi_m(s)\to\bm\Psi_m^0$, where
\[
\bm\Psi_m^0=\operatorname{diag}\!\left(
\frac{D_m}{\kappa'_{1m}g_{1m}(0)\ell_1},\ldots,
\frac{D_m}{\kappa'_{Nm}g_{Nm}(0)\ell_N}\right).
\]
Define
\begin{align*}
\mathbf M_m(s)&=\mathbf I+2\pi\nu D_m\bm{\mathcal G}_m(s)
+\nu\bm\Psi_m(s), \\
\mathbf B_m(s)&=\mathbf M_m(s)
-\frac{2\pi\nu D_m}{sV}\mathbf1\mathbf1^\top, \\
\mathbf B_m^0&=\mathbf I+2\pi\nu D_m\bm{\mathcal G}_m^0
+\nu\bm\Psi_m^0. 
\end{align*}
Assume $\mathbf B_m(s)^{-1}$ is uniformly bounded near $s=0$ and
$c_m^0=\mathbf1^\top(\mathbf B_m^0)^{-1}\mathbf1\ne0$. Then
$\mathbf A_m(s)=\nu\mathbf M_m(s)^{-1}\bm\Gamma_m(s)$ has the limit
\[
\mathbf A_m^0=\frac{Q_{m,0}}{2\pi D_m}
\frac{\mathbf b_m^0}{c_m^0}+\nu\mathbf P_m^0\bm r_m^0,
\]
where
$\mathbf b_m^0=(\mathbf B_m^0)^{-1}\mathbf1$, $c_m^0=\mathbf1^\top\mathbf b_m^0$,
$\mathbf P_m^0=(\mathbf B_m^0)^{-1}
-\dfrac{\mathbf b_m^0\mathbf1^\top(\mathbf B_m^0)^{-1}}{c_m^0}$.
Furthermore,
$\mathbf1^\top\mathbf P_m^0=\mathbf0^\top$ and
$\mathbf1^\top\mathbf A_m^0=\frac{Q_{m,0}}{2\pi D_m}$.
\end{proposition}

\begin{proof}
Put $\alpha_m=2\pi\nu D_m/V$,
$\mathbf b_m(s)=\mathbf B_m(s)^{-1}\mathbf1$, and
$c_m(s)=\mathbf1^\top\mathbf b_m(s)$. Then
$\mathbf M_m(s)=\mathbf B_m(s)+\dfrac{\alpha_m}{s}\mathbf1\mathbf1^\top$.
Sherman--Morrison gives
$\mathbf M_m(s)^{-1}=\mathbf B_m(s)^{-1}
-\dfrac{\mathbf b_m(s)\mathbf1^\top\mathbf B_m(s)^{-1}}
{s/\alpha_m+c_m(s)}$.
Substitution of \eqref{eq:app_forcing_decomposition} yields
$\mathbf M_m^{-1}\bm\Gamma_m =\dfrac{Q_m(s)}{V\alpha_m}
\dfrac{\mathbf b_m(s)}{s/\alpha_m+c_m(s)}
+\mathbf B_m(s)^{-1}\bm r_m(s) -\dfrac{\mathbf b_m(s)\mathbf1^\top
\mathbf B_m(s)^{-1}\bm r_m(s)}{s/\alpha_m+c_m(s)}$.
The $s^{-1}$ poles cancel. Since $V\alpha_m=2\pi\nu D_m$, multiply by
$\nu$ and let $s\downarrow0$. The sum identities follow directly from the
definition of $\mathbf P_m^0$.
\end{proof}

\begin{corollary}[Conservative normalization]
\label{cor:app_conservative_normalization}
If $\bar\gamma_{jm}>0$, then
\[
\bm\pi_m=\lim_{s\downarrow0}\widetilde{\bm{\mathcal J}}_m(s)
=2\pi D_m\mathbf A_m^0,
\qquad
\mathbf1^\top\bm\pi_m=Q_{m,0}.
\]
For a unit release, $\sum_j\pi_{jm}=1$.
\end{corollary}

\begin{proposition}[Uniform-source identity]
\label{prop:app_uniform_source}
For $s+\gamma_m>0$,
\[
\int_\Omega G_m(\mathbf x,\mathbf z;s)\,\mathrm d\mathbf z
=\frac1{s+\gamma_m},
\]
independently of $\mathbf x$.
\end{proposition}

\begin{proof}
By symmetry, apply the Green equation in $\mathbf z$ and integrate over
$\Omega$; the Neumann boundary term vanishes.
\end{proof}

\begin{corollary}[Distributed dosing]
\label{cor:app_uniform_dosing}
If $f_m^{\mathrm{src}}(\mathbf z,s)=Q_m(s)/V$, then
$\Gamma_m(\mathbf x_j,s)=\dfrac{Q_m(s)}{V(s+\gamma_m)}$ for every $j$.
In the conservative limit $\bm r_m^0=\mathbf0$, so
\[
\bm\pi_m^{\mathrm{unif}}=Q_{m,0}
\frac{(\mathbf B_m^0)^{-1}\mathbf1}
{\mathbf1^\top(\mathbf B_m^0)^{-1}\mathbf1}.
\]
Thus uniform dosing removes source-position heterogeneity in the reduced model.
\end{corollary}

\section{Renewal resolvents and moment formulae}
\label{app:renewal_resolvents}

Fix $m$ and suppress it. Let
$\widetilde{\mathbf J}(s)=(\widetilde J_1(s),\ldots,
\widetilde J_N(s))^\top$ be the first-adsorption flux vector, and set
\[
\bm\Sigma=\operatorname{diag}(\sigma_1,\ldots,\sigma_N),\qquad
\bm W=\mathbf I-\bm\Sigma,\qquad
\bm\Phi(s)=\operatorname{diag}(\widetilde\phi_1(s),\ldots,
\widetilde\phi_N(s)).
\]
Assume $0\leq\sigma_j\leq\sigma_*<1$, $\widetilde\phi_j(0)=1$, and let
$P_{jk}(s)$ be the Laplace transform of the nonnegative re-adsorption density
into target $j$ following desorption from target $k$. Thus, for
$\operatorname{Re}s\geq0$,
\[
|P_{jk}(s)|\leq P_{jk}(\operatorname{Re}s),\qquad
|\widetilde\phi_j(s)|\leq\widetilde\phi_j(\operatorname{Re}s)\leq1,
\qquad \sum_jP_{jk}(0)\leq1.
\]

\subsection{Continued search}

Define
$\mathbf K(s)=\mathbf P(s)\bm\Sigma\bm\Phi(s)$ and
$\mathbf R(s)=[\mathbf I-\mathbf K(s)]^{-1}$.

\begin{proposition}[Resolvent convergence]
\label{prop:app_resolvent_convergence}
For $\operatorname{Re}s\geq0$,
$\rho(\mathbf K(s))\leq\|\mathbf K(s)\|_1\leq\sigma_*<1$.
Consequently, $\mathbf R(s)=\sum_{n=0}^\infty\mathbf K(s)^n$ with absolute convergence in the induced $1$-norm.
\end{proposition}

\begin{proof}
For each column $k$,
$\sum_j|K_{jk}(s)|
\leq\sigma_k\widetilde\phi_k(\operatorname{Re}s)
\sum_jP_{jk}(0)\leq\sigma_*$.
\end{proof}

After exactly $n$ failed cycles the adsorption transform is
$\mathbf K^n\widetilde{\mathbf J}$. Hence
$\widetilde{\bm{\mathcal J}}^{\mathrm{cont}}(s)
=\bm W\bm\Phi(s)\mathbf R(s)\widetilde{\mathbf J}(s)$.
If search is conservative,
$\mathbf1^\top\mathbf P(0)=\mathbf1^\top$ and
$\mathbf1^\top\widetilde{\mathbf J}(0)=1$. Since
$\mathbf1^\top[\mathbf I-\mathbf P(0)\bm\Sigma]
=\mathbf1^\top(\mathbf I-\bm\Sigma)=\mathbf1^\top\bm W$,
we obtain
$\mathbf1^\top\widetilde{\bm{\mathcal J}}^{\mathrm{cont}}(0)=1$.

\subsection{Continued-search moments}

Let
$\mathbf L(s)=\bm W\bm\Phi(s)\mathbf R(s)\widetilde{\mathbf J}(s)$.
Assuming the displayed derivatives exist,
\begin{align*}
\mathbf R'&=\mathbf R\mathbf K'\mathbf R,  & \mathbf R''&=\mathbf R\mathbf K''\mathbf R
+2\mathbf R\mathbf K'\mathbf R\mathbf K'\mathbf R,\\[4pt]
\mathbf K'&=\mathbf P'\bm\Sigma\bm\Phi
+\mathbf P\bm\Sigma\bm\Phi',  &
\mathbf K''&=\mathbf P''\bm\Sigma\bm\Phi
+2\mathbf P'\bm\Sigma\bm\Phi'
+\mathbf P\bm\Sigma\bm\Phi''.
\end{align*}
Therefore
\begin{align}
\mathbf L'
={}&\bm W\left[
\bm\Phi'\mathbf R\widetilde{\mathbf J}
+\bm\Phi\mathbf R\mathbf K'\mathbf R\widetilde{\mathbf J}
+\bm\Phi\mathbf R\widetilde{\mathbf J}'\right],
\label{eq:app_L_first}\\
\mathbf L''
={}&\bm W\left[
\bm\Phi''\mathbf R\widetilde{\mathbf J}
+2\bm\Phi'\mathbf R'\widetilde{\mathbf J}
+2\bm\Phi'\mathbf R\widetilde{\mathbf J}'
+\bm\Phi\mathbf R''\widetilde{\mathbf J}\right.\left.
+2\bm\Phi\mathbf R'\widetilde{\mathbf J}'
+\bm\Phi\mathbf R\widetilde{\mathbf J}''\right].
\label{eq:app_L_second}
\end{align}
With $\pi_j=L_j(0)$ and
$p_{\mathrm{int}}=\mathbf1^\top\mathbf L(0)$,
\begin{align*}
\mathbb E[T_j\mid j]&=-\frac{L_j'(0)}{\pi_j},
&\mathbb E[T\mid\mathrm{int}]&=-\frac{\mathbf1^\top\mathbf L'(0)}
{p_{\mathrm{int}}},\\
\mathbb E[T_j^2\mid j]&=\frac{L_j''(0)}{\pi_j},
&\operatorname{Var}(T_j\mid j)&=\frac{L_j''(0)}{\pi_j}
-\left(\frac{L_j'(0)}{\pi_j}\right)^2.
\end{align*}

\subsection{Dependence on partial remixing}

Let
$\mathbf P^{(\eta)}(s)=(1-\eta)\mathbf P^{\mathrm{reset}}(s)
+\eta\mathbf P^{\mathrm{cont}}(s)$ and
$\Delta\mathbf P=\mathbf P^{\mathrm{cont}}-\mathbf P^{\mathrm{reset}}$.
Then
$\partial_\eta\mathbf K^{(\eta)}=\Delta\mathbf P\bm\Sigma\bm\Phi$,
$\partial_\eta\mathbf R^{(\eta)}=\mathbf R^{(\eta)}(\partial_\eta\mathbf K^{(\eta)})
\mathbf R^{(\eta)}$, and
$\partial_\eta\mathbf L^{(\eta)}
=\bm W\bm\Phi\mathbf R^{(\eta)}\Delta\mathbf P\bm\Sigma\bm\Phi
\mathbf R^{(\eta)}\widetilde{\mathbf J}$.
Differentiating \eqref{eq:app_L_first} also gives
\begin{equation*}
\partial_\eta\mathbf L'
=\bm W\bigl[\bm\Phi'\mathbf R_\eta\widetilde{\mathbf J}
+\bm\Phi\mathbf R_\eta\mathbf K'\mathbf R\widetilde{\mathbf J}
+\bm\Phi\mathbf R\mathbf K_\eta'\mathbf R\widetilde{\mathbf J}
+\bm\Phi\mathbf R\mathbf K'\mathbf R_\eta\widetilde{\mathbf J}
+\bm\Phi\mathbf R_\eta\widetilde{\mathbf J}'\bigr],
\end{equation*}
where
$\mathbf K_\eta=\Delta\mathbf P\bm\Sigma\bm\Phi$,
$\mathbf K_\eta'=\Delta\mathbf P'\bm\Sigma\bm\Phi
+\Delta\mathbf P\bm\Sigma\bm\Phi'$
and $\mathbf R_\eta=\mathbf R\mathbf K_\eta\mathbf R$.
Thus
\begin{equation}
\partial_\eta\mathbb E[T_j\mid j]
=-\frac{(\partial_\eta L_j'(0))L_j(0)
-L_j'(0)\partial_\eta L_j(0)}{L_j(0)^2}.
\label{eq:app_eta_mean}
\end{equation}

\subsection{Complete resetting}

Under complete resetting,
$\mathbf P^{\mathrm{reset}}(s)=\widetilde{\mathbf J}(s)\mathbf1^\top$.
Define
$\mathfrak S_j(s)=(1-\sigma_j)\widetilde\phi_j(s)\widetilde J_j(s)$,
$\mathfrak S(s)=\sum_j\mathfrak S_j(s)$ and
$\mathfrak F(s)=\sum_j\sigma_j\widetilde\phi_j(s)\widetilde J_j(s)$.
Summing over failed cycles gives
\[
\widetilde{\mathcal J}^{\mathrm{reset}}_j(s)
=\frac{\mathfrak S_j(s)}{1-\mathfrak F(s)},\qquad
\widetilde{\mathcal F}^{\mathrm{reset}}(s)
=\frac{\mathfrak S(s)}{1-\mathfrak F(s)}.
\]
Hence
\[
\pi_j^{\mathrm{reset}}=\frac{\mathfrak S_j(0)}{1-\mathfrak F(0)},
\qquad
p_{\mathrm{int}}=\frac{\mathfrak S(0)}{1-\mathfrak F(0)}\leq1.
\]

\subsection{Conditional moments}

Normalize the total transform by its mass at zero. Logarithmic
differentiation gives
\begin{equation}
\mathbb E[T\mid\mathrm{int}]
=-\frac{\mathfrak S'(0)}{\mathfrak S(0)}
-\frac{\mathfrak F'(0)}{1-\mathfrak F(0)}.
\label{eq:app_lossy_mean}
\end{equation}
Moreover,
\[
\mathbb E[T^2\mid\mathrm{int}]
={}\frac{\mathfrak S''(0)}{\mathfrak S(0)}
+\frac{2\mathfrak S'(0)\mathfrak F'(0)}
{\mathfrak S(0)[1-\mathfrak F(0)]}
+\frac{\mathfrak F''(0)}{1-\mathfrak F(0)}
+\frac{2[\mathfrak F'(0)]^2}{[1-\mathfrak F(0)]^2}.
\]
Target-specific conditioning replaces $\mathfrak S$ in the first term of
\eqref{eq:app_lossy_mean} by $\mathfrak S_j$.

\subsection{Conservative specialization}

If $\sum_j\widetilde J_j(0)=1$, write
$\overline\pi_j=\widetilde J_j(0)$ and
$p=\sum_j(1-\sigma_j)\overline\pi_j=1-\mathfrak F(0)$.
Then
\[
\pi_j^{\mathrm{reset}}=
\frac{(1-\sigma_j)\overline\pi_j}
{\sum_k(1-\sigma_k)\overline\pi_k},\qquad
\sum_j\pi_j^{\mathrm{reset}}=1.
\]
Let $\mathfrak K_j=\widetilde\phi_j\widetilde J_j$ and
\[
\mathfrak m_j^{(1)}=-\mathfrak K_j'(0),\quad
\mathfrak m_j^{(2)}=\mathfrak K_j''(0),\quad
A_1=\sum_j\mathfrak m_j^{(1)},\quad
B_1=\sum_j\sigma_j\mathfrak m_j^{(1)},\quad
A_2=\sum_j\mathfrak m_j^{(2)}.
\]
Then
\[
\mathbb E[T]=\frac{A_1}{p},\qquad
\mathbb E[T^2]=\frac{A_2}{p}+\frac{2A_1B_1}{p^2}.
\]
If $T_j^{(1)}=-\widetilde J_j'(0)$,
$T_j^{(2)}=\widetilde J_j''(0)$,
$\langle\tau\rangle_j=-\widetilde\phi_j'(0)$, and
$\langle\tau^2\rangle_j=\widetilde\phi_j''(0)$, then
\begin{equation*}
\mathfrak m_j^{(1)}=T_j^{(1)}+\overline\pi_j\langle\tau\rangle_j,\quad
\mathfrak m_j^{(2)}=T_j^{(2)}+2T_j^{(1)}\langle\tau\rangle_j
+\overline\pi_j\langle\tau^2\rangle_j.
\end{equation*}

\section{Large-batch acceptance limit}
\label{app:large_batch}

Let $H=(Z_{\mathrm{line}},Z_{\mathrm{batch}})$ collect effects shared within a
batch. Conditional on $H$, assume that $(M_j,D_j)_{j\geq1}$ are independent
and identically distributed, where $0\leq M_j\leq1$. Define
\begin{align}
X_j&=\mathbf1\!\left\{M_j\geq\upsilon_{\mathrm{mt}}^{\min},\quad
D_j\leq\upsilon_{\mathrm{ds}}^{\max}\right\},\nonumber\\
\widehat p_N&=\frac1N\sum_{j=1}^NX_j,qquad
\widehat\mu_N=\frac1N\sum_{j=1}^NM_j,
\label{eq:app_empirical_acceptance}
\end{align}
and
\[
\mathcal A_N=\left\{\widehat p_N\geq r_{\min},\quad
\widehat\mu_N\geq\overline\upsilon_{\mathrm{mt}}^{\min}\right\}.
\]
Write
\[
p(H)=\mathbb E[X_1\mid H],\qquad
\mu(H)=\mathbb E[M_1\mid H],
\]
and
\[
\mathcal A_\infty=\left\{p(H)\geq r_{\min},\quad
\mu(H)\geq\overline\upsilon_{\mathrm{mt}}^{\min}\right\}.
\]

\begin{proposition}[Large-batch acceptance probability]
\label{prop:app_large_batch}
If
\begin{equation}
\mathbb P_H\!\left\{p(H)=r_{\min}\ \text{or}\
\mu(H)=\overline\upsilon_{\mathrm{mt}}^{\min}\right\}=0,
\label{eq:app_no_boundary_mass}
\end{equation}
then
\begin{equation}
P_{\mathrm{accept}}(N)=\mathbb P(\mathcal A_N)
\longrightarrow
\mathbb P_H\!\left\{p(H)\geq r_{\min},\quad
\mu(H)\geq\overline\upsilon_{\mathrm{mt}}^{\min}\right\}.
\label{eq:app_large_batch_limit}
\end{equation}
\end{proposition}

\begin{proof}
Conditional strong laws give
$\widehat p_N\to p(H)$ and $\widehat\mu_N\to\mu(H)$ almost surely. Away
from the null boundary event in \eqref{eq:app_no_boundary_mass},
$\mathbf1_{\mathcal A_N}\to\mathbf1_{\mathcal A_\infty}$. Dominated
convergence completes the proof.
\end{proof}

\subsection{Finite-batch bound}

For $\delta_p,\delta_M>0$, define
\begin{equation}
\mathcal B_{\delta_p,\delta_M}=
\left\{|p(H)-r_{\min}|\leq\delta_p\ \text{or}\
|\mu(H)-\overline\upsilon_{\mathrm{mt}}^{\min}|\leq\delta_M\right\}.
\label{eq:app_margin_event}
\end{equation}

\begin{proposition}[Finite-batch error]
\label{prop:app_finite_batch}
If $M_j\in[M_-,M_+]$ almost surely, then
\begin{align}
|P_{\mathrm{accept}}(N)-P_{\mathrm{accept}}(\infty)|
\leq{}&\mathbb P_H(\mathcal B_{\delta_p,\delta_M})
+2e^{-2N\delta_p^2}\nonumber\\
&+2\exp\!\left[-\frac{2N\delta_M^2}{(M_+-M_-)^2}\right].
\label{eq:app_finite_batch_bound}
\end{align}
For $M_j\in[0,1]$ and $\delta_p=\delta_M=\delta$,
\begin{equation}
|P_{\mathrm{accept}}(N)-P_{\mathrm{accept}}(\infty)|
\leq\mathbb P_H(\mathcal B_{\delta,\delta})+4e^{-2N\delta^2}.
\label{eq:app_finite_batch_unit}
\end{equation}
\end{proposition}

\begin{proof}
Outside $\mathcal B_{\delta_p,\delta_M}$, the finite and limiting decisions
agree whenever
$|\widehat p_N-p(H)|<\delta_p$ and
$|\widehat\mu_N-\mu(H)|<\delta_M$. Apply conditional Hoeffding inequalities
and then average over $H$.
\end{proof}

\begin{corollary}[Batch-size criterion]
For $M_j\in[0,1]$ and prescribed $\eta\in(0,1)$,
\[
N\geq\frac1{2\delta^2}\log\frac4\eta
\quad\Longrightarrow\quad
|P_{\mathrm{accept}}(N)-P_{\mathrm{accept}}(\infty)|
\leq\mathbb P_H(\mathcal B_{\delta,\delta})+\eta.
\]
If $\mathbb P_H(\mathcal B_{\delta,\delta})\leq C_{\mathrm{ac}}\delta^\alpha$,
then $\delta_N=\sqrt{\log N/N}$ gives
\[
|P_{\mathrm{accept}}(N)-P_{\mathrm{accept}}(\infty)|
\leq C_{\mathrm{ac}}\left(\frac{\log N}{N}\right)^{\alpha/2}+4N^{-2}.
\]
\end{corollary}

\subsection{Shared and independent variance}

\begin{proposition}[Exact variance decomposition]
\label{prop:app_batch_variance}
Conditional independence gives
\begin{align}
\operatorname{Var}(\widehat p_N)
&=\operatorname{Var}_H(p(H))
+\frac1N\mathbb E_H[p(H)(1-p(H))],\\
\operatorname{Var}(\widehat\mu_N)
&=\operatorname{Var}_H(\mu(H))
+\frac1N\mathbb E_H[\operatorname{Var}(M_1\mid H)],\\
\operatorname{Cov}(\widehat p_N,\widehat\mu_N)
&=\operatorname{Cov}_H(p(H),\mu(H))
+\frac1N\mathbb E_H[\operatorname{Cov}(X_1,M_1\mid H)].
\end{align}
\end{proposition}

\begin{proof}
Apply the laws of total variance and total covariance.
\end{proof}

\section{Formal three-dimensional reduction}
\label{app:formal_3d}

This section is a formal matched-asymptotic expansion for three-dimensional
spherical targets.

Fix $m$, suppress it where unambiguous, and set
\[
a(s)=s+\gamma_m,\quad
\chi_j(s)=s+\gamma^d_{jm}+\bar\gamma_{jm},\quad
g_j(s)=\frac{s+\bar\gamma_{jm}}{\chi_j(s)},\quad
K_j(s)=\kappa'_{jm}g_j(s).
\]
Let
$\mathcal U_j=B_{\varepsilon\ell_j}(\mathbf x_j)\subset\mathbb R^3$ and
$\kappa_{jm}=\kappa'_{jm}/\varepsilon$, with $O(1)$ separation. The transformed
problem is
\begin{subequations}
\label{eq:app_3d_problem}
\begin{align}
D_m\Delta\widetilde u_m-a(s)\widetilde u_m
&=-f_m^{\mathrm{src}}&&\text{in }\Omega_\varepsilon,\\
D_m\partial_{\mathbf n}\widetilde u_m&=0&&\text{on }\partial\Omega,\\
D_m\partial_{\mathbf n_j}\widetilde u_m
&=\frac{K_j(s)}{\varepsilon}\widetilde u_m
&&\text{on }\partial\mathcal U_j.
\end{align}
\end{subequations}
For $q_{jm,0}=0$,
\begin{equation}
\widetilde{\mathcal J}_{jm}^{\mathrm{full}}(s)
=\frac{\bar\gamma_{jm}\kappa'_{jm}}{\varepsilon\chi_j(s)}
\int_{\partial\mathcal U_j}\widetilde u_m\,\mathrm dS.
\label{eq:app_3d_exact_flux}
\end{equation}

\subsection{Outer monopoles}

Let
\begin{equation}
D_m\Delta_{\mathbf x}G_m-a(s)G_m=-\delta_{\mathbf z},
\qquad \partial_{\mathbf n}G_m=0\quad\text{on }\partial\Omega,
\label{eq:app_3d_green}
\end{equation}
and
\begin{equation}
\Gamma_m(\mathbf x,s)=\int_\Omega
G_m(\mathbf x,\mathbf z;s)f_m^{\mathrm{src}}(\mathbf z,s)
\,\mathrm d\mathbf z.
\label{eq:app_3d_forcing}
\end{equation}
In three dimensions,
\begin{equation}
G_m(\mathbf x,\mathbf x_j;s)=
\frac1{4\pi D_m|\mathbf x-\mathbf x_j|}
+R_m(\mathbf x_j,\mathbf x_j;s)+O(|\mathbf x-\mathbf x_j|).
\label{eq:app_3d_green_local}
\end{equation}
Introduce $Q_{jm}(s)$ through
\begin{equation}
\widetilde u_m^{\mathrm{out}}(\mathbf x,s)=
\Gamma_m(\mathbf x,s)-4\pi D_m\varepsilon
\sum_{k=1}^NQ_{km}(s)G_m(\mathbf x,\mathbf x_k;s).
\label{eq:app_3d_outer}
\end{equation}
For $\mathbf x=\mathbf x_j+\varepsilon\mathbf y$ and $\rho=|\mathbf y|$,
\begin{align}
\widetilde u_m^{\mathrm{out}}
&=-\frac{Q_{jm}}\rho+B_{jm}
+\varepsilon\rho\,\mathbf h_{jm}\cdot\widehat{\mathbf y}
+O(\varepsilon^2\rho^2),
\label{eq:app_3d_outer_local}\\
B_{jm}
&=\Gamma_m(\mathbf x_j,s)-4\pi D_m\varepsilon\left[
R_m(\mathbf x_j,\mathbf x_j;s)Q_{jm}
+\sum_{k\ne j}G_m(\mathbf x_j,\mathbf x_k;s)Q_{km}\right].
\label{eq:app_3d_Bj}
\end{align}

\subsection{Inner Robin capacitance}

The leading inner problem is
\begin{equation}
\Delta_{\mathbf y}U^{(0)}_{jm}=0,qquad
D_m\partial_\rho U^{(0)}_{jm}=K_j(s)U^{(0)}_{jm}
\quad\text{at }\rho=\ell_j,qquad
U^{(0)}_{jm}\to B_{jm}\quad(\rho\to\infty).
\label{eq:app_3d_inner_problem}
\end{equation}
Its solution is
\begin{equation}
U^{(0)}_{jm}(\rho,s)=B_{jm}(s)
\left(1-\frac{\Lambda_{jm}(s)}\rho\right),
\qquad
\Lambda_{jm}(s)=\frac{K_j(s)\ell_j^2}{D_m+K_j(s)\ell_j}.
\label{eq:app_3d_lambda}
\end{equation}
Matching the $\rho^{-1}$ terms gives
\begin{equation}
Q_{jm}(s)=\Lambda_{jm}(s)B_{jm}(s).
\label{eq:app_3d_Q_match}
\end{equation}
The limits are
\begin{align}
\Lambda_{jm}&=\frac{K_j\ell_j^2}{D_m}
-\frac{K_j^2\ell_j^3}{D_m^2}
+O\!\left(\frac{K_j^3\ell_j^4}{D_m^3}\right)
&&\text{if }K_j\ell_j/D_m\ll1,\\
\Lambda_{jm}&=\ell_j-\frac{D_m}{K_j}
+O\!\left(\frac{D_m^2}{K_j^2\ell_j}\right)
&&\text{if }K_j\ell_j/D_m\gg1.
\end{align}

\subsection{Interaction system}

Define
\[
\bm\Lambda_m=\operatorname{diag}(\Lambda_{1m},\ldots,\Lambda_{Nm})
\]
and
\begin{equation}
(\bm{\mathcal G}^{(3)}_m)_{jk}=
\begin{cases}
G_m(\mathbf x_j,\mathbf x_k;s),&j\ne k,\\[1mm]
R_m(\mathbf x_j,\mathbf x_j;s),&j=k.
\end{cases}
\label{eq:app_3d_G_matrix}
\end{equation}
Then
\begin{equation}
\left[\bm\Lambda_m(s)^{-1}
+4\pi D_m\varepsilon\bm{\mathcal G}^{(3)}_m(s)\right]
\mathbf Q_m(s)=\bm\Gamma_m(s).
\label{eq:app_3d_matrix_system}
\end{equation}
Equivalently,
\begin{align}
\mathbf Q_m
&=\bm\Lambda_m\left[
\mathbf I+4\pi D_m\varepsilon\bm{\mathcal G}^{(3)}_m\bm\Lambda_m
\right]^{-1}\bm\Gamma_m,
\label{eq:app_3d_Q_resummed}\\
&=\bm\Lambda_m\bm\Gamma_m
-4\pi D_m\varepsilon\bm\Lambda_m\bm{\mathcal G}^{(3)}_m
\bm\Lambda_m\bm\Gamma_m+O(\varepsilon^2).
\label{eq:app_3d_Q_expansion}
\end{align}

\subsection{Productive internalization flux}

Because $|\partial\mathcal U_j|=4\pi\varepsilon^2\ell_j^2$,
\begin{equation}
\widetilde{\mathcal J}^{\mathrm{red}}_{jm}(s)
\sim4\pi D_m\varepsilon
\frac{\bar\gamma_{jm}}{s+\bar\gamma_{jm}}Q_{jm}(s).
\label{eq:app_3d_productive_flux}
\end{equation}
Also,
\begin{equation}
\int_{\partial\mathcal U_j}D_m\partial_{\mathbf n_j}
\widetilde u_m\,\mathrm dS\sim4\pi D_m\varepsilon Q_{jm}.
\label{eq:app_3d_net_flux}
\end{equation}

\subsection{First angular correction}

The first harmonic is
\begin{equation}
U^{(1)}_{jm}=(\mathbf h_{jm}\cdot\widehat{\mathbf y})
\left(\rho+c_{jm}(s)\rho^{-2}\right),qquad
c_{jm}(s)=\ell_j^3
\frac{D_m-K_j(s)\ell_j}{2D_m+K_j(s)\ell_j}.
\label{eq:app_3d_dipole}
\end{equation}
Its surface mean and mean radial derivative vanish. Thus it makes no integrated
flux contribution. Formal bookkeeping places the first omitted integrated-flux
term at absolute order $O(\varepsilon^3)$ for spheres; this is not a controlled
remainder.

\subsection{Non-spherical targets}

For a reference shape $\mathcal B_j$, define the Robin capacitance
$\mathcal C^R_{jm}(s)$ by
\begin{subequations}
\begin{align}
\Delta W_{jm}&=0&&\text{in }\mathbb R^3\setminus\overline{\mathcal B_j},\\
D_m\partial_{\mathbf n}W_{jm}&=K_j(s)W_{jm}
&&\text{on }\partial\mathcal B_j,\\
W_{jm}(\mathbf y)&=1-\frac{\mathcal C^R_{jm}(s)}{|\mathbf y|}
+o(|\mathbf y|^{-1})&&\text{as }|\mathbf y|\to\infty,
\end{align}
\end{subequations}
where $\mathbf n$ points from $\mathcal B_j$ into the exterior. Replacing
$\Lambda_{jm}$ by $\mathcal C^R_{jm}$ gives the formal non-spherical reduction.

\section{Stochastically gated accessibility}
\label{app:gating}
Surface availability is a source of partial reactivity distinct from encounter
history. Let $G_{jm}(t)\in\{o,c\}$ switch at rates $\alpha_{jm}$, $\beta_{jm}$
with $\omega^g_{jm}=\alpha_{jm}+\beta_{jm}$ and stationary open fraction
$\pi^g_{o,jm}=\beta_{jm}/\omega^g_{jm}$. Gating acts \emph{before} adsorption
whereas $\sigma_{jm}$ acts after it, so the two must not be merged; in
particular, multiplying an ungated rate by the stationary open fraction is
incorrect, because after a rejected encounter the gate is known to be closed and
the next cycle retains that information. The renewal representation gives the
shift identity
$\widetilde f_{g\mid g_0}(s)=\pi^g_g\widetilde f(s)
+(\delta_{gg_0}-\pi^g_g)\widetilde f(s+\omega^g_{jm})$. Under symmetric
switching, $\alpha_{jm}=\beta_{jm}=\omega$ and
$\omega^g_{jm}=2\omega$, and the associated Dirichlet-to-Neumann combination is
$\mathcal L^g_{m,s,\omega}
=\tfrac12(\mathcal L_{m,s}+\mathcal L_{m,s+2\omega})$. Thus the transformed
shift is governed by the relaxation rate $\alpha_{jm}+\beta_{jm}$, not by an
individual switching rate; this is the form in
\cite{bressloff2026renewal}, and the two singular limits --- a Robin law only
under a fast-switching, rare-opening scaling, and an MFPT that diverges for an
initially closed target while tending to twice the ungated value for an
initially open one --- are given in \cite{bressloff2026renewal}. Composing
gating with the downstream model by replacing $\widetilde{\mathbf J}_m$ with the
successful open-gate transform while discarding gate--residence correlations is
asymptotically justified when the gate re-equilibrates over a residence time,
$\omega^g_{jm}\langle\tau\rangle_{jm}\gg1$; it is not an exact identity at a
finite switching rate. No result in this paper uses
a gated calculation.

\section{Proofs of Propositions}
\label{app:proofs}

Proposition~\ref{prop:red} is derived in Appendix~\ref{app:matched2d}; its
$O(\varepsilon^2)$ remainder has precisely the conditional scope stated in
Proposition~\ref{prop:app_flux_error}.

Proposition~\ref{prop:pt} follows from the zero-mode calculation in
Appendix~\ref{app:zeromode} and the renewal resolvents in
Appendix~\ref{app:renewal_resolvents}.

\begin{proof}[Proof of Corollary~\ref{cor:remix}]
At $\eta=1$, $\mathbf I-\mathbf P^{\mathrm{cont}}\bm\Sigma
=\bm W-\mathbf E\bm\Sigma=(\mathbf I-\mathbf E\bm\Sigma\bm W^{-1})\bm W$, and
$\max_j\sigma_j\le1-\delta$ gives $\|\bm W^{-1}\|_1\le\delta^{-1}$, so
$\bm\pi^{(1)}=(\mathbf I-\mathbf E\bm\Sigma\bm W^{-1})^{-1}\overline{\bm\pi}
=\overline{\bm\pi}+\mathbf E\bm\Sigma\bm W^{-1}\overline{\bm\pi}+O(\nu^2)$; the
constant is controlled by
$\|\mathbf E\bm\Sigma\bm W^{-1}\|_1\le\|\mathbf E\|_1\max_j\sigma_j/(1-\sigma_j)$.
At $\eta=0$, $\mathbf P^{(0)}=\overline{\bm\pi}\mathbf1^\top$ is rank one and
Sherman--Morrison gives
$(\mathbf I-\overline{\bm\pi}\mathbf1^\top\bm\Sigma)^{-1}\overline{\bm\pi}
=\overline{\bm\pi}/(1-\mathbf1^\top\bm\Sigma\overline{\bm\pi})$, whence
\eqref{eq:endpoints}. If $\bm\Sigma=\sigma\mathbf I$ then at $\mathbf E=\mathbf0$
one has $\mathbf P^{(\eta)}\overline{\bm\pi}=\overline{\bm\pi}$ for every
$\eta$, so $(\mathbf I-\sigma\mathbf P^{(\eta)})^{-1}\overline{\bm\pi}
=\overline{\bm\pi}/(1-\sigma)$ and multiplication by $\bm W=(1-\sigma)\mathbf I$
returns $\overline{\bm\pi}$; restoring $\mathbf E=O(\nu)$ changes this by
$O(\nu)$ uniformly in $\eta$. Under Assumption~\ref{ass:moments} one may
differentiate \eqref{eq:cont} at $s=0$; this introduces
$\partial_s\mathbf P^{(\eta)}(0)$, which $\bm W$ does not cancel, giving the
timing statement. The argument shows only that the $\eta$-derivative of the
timing functional is generically nonzero: it is a linear functional of
$\partial_s\mathbf P^{\mathrm{cont}}(0)-\partial_s\mathbf
P^{\mathrm{reset}}(0)$, which can vanish for particular geometries and
residence laws, so exceptional cancellations are not excluded.
\end{proof}

\begin{proof}[Proof of Proposition~\ref{prop:lip}]
Fix $m$ and set
$H=L^2(\Omega_\varepsilon)\times\prod_jL^2(\partial\mathcal U_j)$ and
$V=H^1(\Omega_\varepsilon)\times\prod_jL^2(\partial\mathcal U_j)$.
For $z=(u,(q_j))$ and $w=(v,(r_j))$, the homogeneous problem is generated by
the form
\[
\mathfrak a(z,w)=D_m(\nabla u,\nabla v)+\gamma_m(u,v)
+\sum_j\!\int_{\partial\mathcal U_j}
\{(\kappa_{jm}u-\gamma^d_{jm}q_j)v
-(\kappa_{jm}u-(\gamma^d_{jm}+\bar\gamma_{jm})q_j)r_j\}.
\]
The trace theorem and Young's inequality give
$|\mathfrak a(z,w)|\le C\|z\|_V\|w\|_V$ and
$\mathfrak a(z,z)+\omega\|z\|_H^2\ge c\|z\|_V^2$ for suitable
$c>0$, $\omega\ge0$. The associated operator therefore generates an analytic
semigroup; Duhamel's formula gives the unique mild solution for
$I_m\in L^1(0,T;L^2)$. Testing with the negative parts yields
\[
\frac12\frac{\mathrm d}{\mathrm dt}
\left(\|u_m^-\|_2^2+\sum_j\|q_{jm}^-\|_2^2\right)
\le C\left(\|u_m^-\|_2^2+\sum_j\|q_{jm}^-\|_2^2\right),
\]
because $I_m\ge0$; Gronwall proves positivity. Testing both equations with
constants and adding the internalized amount gives
\[
\frac{\mathrm d}{\mathrm dt}\left(
B_m^{\mathrm{bulk}}+\sum_jB_{jm}^{\mathrm{surf}}
+\sum_jC_{jm}^{\mathrm{int}}\right)
=\int_{\Omega_\varepsilon}I_m
-\gamma_m\int_{\Omega_\varepsilon}u_m,
\]
which is equivalent to \eqref{eq:massbal}.

For the reduction, write
$\mathbf M_m=\mathbf I+2\pi\nu D_m\bm{\mathcal G}_m+\nu\bm\Psi_m$.
The resolvent identity and the assumed inverse bound imply
\[
\|\mathbf A_m(\boldsymbol\theta)-\mathbf A_m(\boldsymbol\theta_0)\|
\le C\|\boldsymbol\theta-\boldsymbol\theta_0\|.
\]
Equation~\eqref{eq:fluxred} transfers this estimate to the productive flux.
The vector field of \eqref{eq:transformed_problem} is Lipschitz on the invariant
set $\mathcal K$, so Gronwall gives the stated output bound. Finally, if
$\boldsymbol\theta_n\to\boldsymbol\theta$ almost surely, output continuity and
\eqref{eq:nullbdry} imply convergence of the acceptance indicators; dominated
convergence proves continuity of their probabilities.
\end{proof}

\section{Estimators and reproducibility ledger}
\label{app:mc}
For LHS each marginal is stratified into $N_{\mathrm{LHS}}$ equiprobable bins
($N_{\mathrm{LHS}}=4096$ here), one point drawn per bin and coordinates permuted
independently. We write $N_{\mathrm{MC}}$ for the separate Sobol base-matrix
size, which is not the same number. PRCC uses ordinary least-squares
residuals of the rank-transformed variables. For Sobol estimation with base
matrices $\mathbf A,\mathbf B$ and hybrids $\mathbf A^{(i)}_{\mathbf B}$ \cite{saltelli2010variance,jansen1999analysis}:
\[
\widehat S_i=\frac{1}{N_{\mathrm{MC}}\widehat V}\sum_n f^n_B(f^n_{AB_i}-f^n_A),
\qquad
\widehat S_{T_i}=\frac{1}{2N_{\mathrm{MC}}\widehat V}\sum_n(f^n_A-f^n_{AB_i})^2;
\]
with $p=7$ and $N_{\mathrm{MC}}=2^{15}$ this is $294\,912$, and $327\,680$ for
the eight-input audit.
The computation is reconstructed in the following order. (i) Solve the
four-by-four Bessel system and evaluate the reduced counterpart for
Table~\ref{tab:verify}. (ii) Evaluate \eqref{eq:series} and \eqref{eq:ewald},
then \eqref{eq:amp}, keeping conservative and lossy rows separate; for lossy
rows use \eqref{eq:ewaldloss}, or a Bessel image sum with
$\|\mathbf k\|_\infty\gtrsim13\sqrt{D/\gamma}$, and record the cutoff actually
used. (iii) Recover $\mathcal J_j(t)$ and $U_j(T)$. (iv) Integrate \eqref{eq:state} with the
constants and initial state of \eqref{eq:params}. (v) For \S\ref{subsec:retention}, compute both the two-target and the two single-target masses at each $d$ before forming \eqref{eq:comp}. (vi) For
uncertainty, use the ranges, designs and seeds printed in
\S\ref{subsec:sens}. (vii) For the null model, use the layout protocol of
\S\ref{subsec:nullspread} verbatim --- centre law, boundary clearance $0.10$,
hard-core radius $0.12$ applied to the whole configuration, independent port law,
PCG64 seed $2024$, draw order (centres, rejections, port), $2000$ draws, the source-clearance resampling at $d_{\mathrm{src}}=0.10$, and $\Psi=0.75$ common to both dosing
geometries --- and normalize the uniform source to the same total injection rate
as the port before rescaling both flux vectors to unit mean; bootstrap with
$4000$ resamples at seed $7$. The minimum record for every exported curve is
\[
(\text{regime},\gamma,\varepsilon,D,\ell,\kappa',\gamma^d,\bar\gamma,\eta,
\mathbf x_0,\{\mathbf x_j\},a_{\mathrm{dose}},t_{\mathrm d},T,N_b,
s_{\mathrm{line}},s_{\mathrm{batch}},s_{\mathrm{proc}},\text{seed}),
\]
with inapplicable entries recorded as ``not used'' rather than silently
defaulted.

\begin{table}[htbp]\centering\small
\caption{Exact transformed full PDE versus the reduced model; $R=1$, $r_0=0.5$,
$\ell=1$, $D=1$, $\gamma=0.4$, $\kappa'=1.5$, $\gamma^d=0.5$, $\bar\gamma=1$.
Relative error in per cent.}
\label{tab:verify}
\begin{tabular}{@{}c|rrr|rrr|rrr@{}}
\toprule
&\multicolumn{3}{c|}{$s=0$}&\multicolumn{3}{c|}{$s=1$}&\multicolumn{3}{c}{$s=5$}\\
$\varepsilon$&full&reduced&err.&full&reduced&err.&full&reduced&err.\\
\midrule
$0.10$&0.673558&0.672974&$-0.0868$&0.199489&0.199157&$-0.1662$&0.028155&0.028081&$-0.2647$\\
$0.05$&0.616341&0.616206&$-0.0220$&0.167950&0.167876&$-0.0441$&0.021683&0.021665&$-0.0840$\\
$0.02$&0.554407&0.554387&$-0.0036$&0.139022&0.139012&$-0.0074$&0.016642&0.016639&$-0.0158$\\
\bottomrule
\end{tabular}
\end{table}

\end{appendices}

\bibliography{1references}

\end{document}